\documentclass[letterpaper,11pt,oneside,reqno]{amsart}

\usepackage[sorting=nyt,style=alphabetic,backend=biber,hyperref=true,eprint=true,doi=false,isbn=false,url=false,maxbibnames=9,maxcitenames=4,maxalphanames=9,giveninits=true]{biblatex}
\AtEveryBibitem{
  \clearlist{language}
    \clearfield{eprintclass}
}

\usepackage{amsmath,amssymb,amsthm,amsfonts}
\usepackage{bm}
\usepackage{graphicx,color}
\usepackage[colorlinks=true,linkcolor=blue,citecolor=red]{hyperref}
\usepackage{bbm}
\usepackage{upgreek}
\usepackage[mathscr]{euscript}
\usepackage{textcomp}
\usepackage{hhline}
\usepackage{relsize}

\allowdisplaybreaks
\numberwithin{equation}{section}

\usepackage{tikz}
\usetikzlibrary{shapes,arrows,positioning,decorations.markings}
\usetikzlibrary{graphs,patterns,decorations.markings,arrows.meta,matrix}
\usetikzlibrary{calc,decorations.pathmorphing,decorations.pathreplacing,shapes}

\usepackage{array}
\usepackage{adjustbox}
\usepackage{cleveref}
\usepackage{enumerate}
\usepackage{caption}
\usepackage[DIV=10]{typearea}
\colorlet{lgray}{white!60!black}
\colorlet{dgray}{white!45!black}
\colorlet{lred}{white!85!red}
\colorlet{dred}{white!35!red}
\colorlet{lgreen}{white!60!green}
\colorlet{dgreen}{black!30!green}
\colorlet{lpurple}{white!60!purple}
\colorlet{lblue}{white!60!blue}
\definecolor{green}{rgb}{0.1,0.8,0.1}
\definecolor{yellow}{rgb}{1.0,0.85,0.25}
\definecolor{purple}{rgb}{1.0, 0, 1.0}
\definecolor{blue}{rgb}{0, 0, 1.0}

\usepackage[shadow=true]{todonotes}

\newtheorem{proposition}{Proposition}[section]
\newtheorem{lemma}[proposition]{Lemma}
\newtheorem{corollary}[proposition]{Corollary}
\newtheorem{theorem}[proposition]{Theorem}
\newtheorem{theoremAlph}{Theorem}

\theoremstyle{definition}
\newtheorem{definition}[proposition]{Definition}
\newtheorem{assumption}[proposition]{Assumption}

\newtheorem{remark}[proposition]{Remark}

\renewcommand{\sl}{\mathfrak{sl}}

\DeclareMathOperator{\End}{End}

\newcommand{\tW}{\widetilde{W}}
\newcommand{\tw}{\widetilde{w}}
\newcommand{\tR}{\widetilde{R}}
\newcommand{\tK}{\widetilde{K}}

\newcommand \calA{\mathcal A}
\newcommand \calB{\mathcal B}

\newcommand \calF{\mathcal F}

\newcommand \Z{\mathbb Z}
\newcommand \F{\mathbb F}
\newcommand \T{\mathbb T}
\newcommand \B{\mathbb B}

\newcommand\sfF{\mathsf{F}}

\tikzset{bgplaq/.style={fill=lightgray!30!white}}
\tikzset{bggplaq/.style={fill=lightgray!70!white}}
\tikzset{bline/.style={line width=0.8mm}}
\tikzset{fline/.style={}}
\tikzset{gbline/.style={line width=0.8mm,draw=black!35!green}}
\tikzset{gline/.style={line width=0.3mm,draw=black!35!green}}
\tikzset{bbline/.style={line width=0.8mm,draw=black!35!blue}}
\tikzset{bline/.style={line width=0.3mm,draw=black!35!blue}}
\tikzset{wbline/.style={line width=0.6mm,draw=lightgray!10!white}}
\tikzset{wline/.style={draw=lightgray!10!white}}
\tikzset{part/.style={thick,solid,circle,fill=black,draw=black,inner sep=1.8pt,outer sep=3pt}}
\tikzset{hole/.style={thick,solid,circle,draw=black,fill=white,inner sep=1.8pt,outer sep=3pt}}
\tikzset{arrow/.style={postaction={decorate,thick,decoration={markings,mark = at position #1 with {\arrow{>}}}}},arrow/.default=0.5}

\tikzset{cross/.style={postaction={decorate,thick,decoration={markings,mark = at position #1 with{\draw (-2pt,-2pt) -- (2pt,2pt);\draw (2pt,-2pt) -- (-2pt,2pt);}}}}}
\tikzset{distort/.style={cm={1,0,-\slt,\sltb,(0,0)}}}
\def\slt{0.2}
\pgfmathsetmacro{\sltb}{sqrt(1-\slt*\slt)}

\newcommand\gbull[3]{\filldraw[fill=black!35!green, draw=black] (#1,#2) circle (#3cm);}

\tikzstyle{fused}=[lgray, line width=5pt, arrows={-Stealth[scale=1.1,length=10, width=10,dgray]}]
\tikzstyle{fused*}=[lgray, line width=5pt]
\tikzstyle{unfused}=[lgray, line width=1pt, arrows={-Stealth[scale=1.1,length=5, width=5,dgray]}]
\tikzstyle{unfused*}=[lgray, line width=1pt]
\tikzstyle{fleche}=[>=stealth', postaction={decorate}, thick]

\RequirePackage{tikz}
\RequirePackage{amsmath,amssymb,amsfonts}
\RequirePackage{bm}

\usetikzlibrary{arrows.meta,arrows}
\usetikzlibrary{calc,positioning,fit}
\usetikzlibrary{decorations.markings,decorations.pathmorphing,decorations.pathreplacing}
\usetikzlibrary{patterns,shapes,matrix,graphs}

\colorlet{lgray}{white!85!black}
\colorlet{dgray}{white!45!black}
\colorlet{lred}{white!85!red}
\colorlet{dred}{white!35!red}
\colorlet{lgreen}{white!60!green}
\colorlet{dgreen}{black!30!green}
\colorlet{lpurple}{white!60!purple}
\colorlet{lblue}{white!60!blue}

  \definecolor{green}{rgb}{0.1,0.8,0.1}
  \definecolor{yellow}{rgb}{1.0,0.85,0.25}
  \definecolor{purple}{rgb}{1.0,0,1.0}
  \definecolor{blue}{rgb}{0,0,1.0}

\newcommand{\stikzbase}[3]{%
	\begin{tikzpicture}[scale=#1,baseline={([yshift=#2]current bounding box.center)},>=stealth]#3\end{tikzpicture}%
}

\tikzset{
	fused/.style  = {lgray, line width=3pt,
		arrows={-Stealth[scale=1.1,length=10,width=10,dgray]}},
	fused*/.style = {lgray, line width=3pt}
}

\tikzset{
	vmblue-thick/.style    = {lgray, line width=3pt},
	vmblue-thick*/.style   = {lgray, line width=3pt},
	dred-thick/.style      = {lgray, line width=3pt},
	dred-thick*/.style     = {lgray, line width=3pt},
	vmpurple-thick/.style  = {lgray, line width=3pt},
	vmpurple-thick*/.style = {lgray, line width=3pt},
	black-thick/.style     = {lgray, line width=3pt},
	black-thick*/.style    = {lgray, line width=3pt},
  reflection-arrow/.style = {lgray, line width=3pt,
    line cap=round, line join=round,
    arrows={-Stealth[scale=1.1,length=10,width=10,dgray]}},
  boundary-arrow/.style = {lgray, line width=3pt,
    line cap=round, line join=round,
    arrows={-Stealth[scale=1.1,length=10,width=10,dgray]}},
  cascade-arrow/.style = {lgray, line width=3pt,
    line cap=round, line join=round,
    arrows={-Stealth[scale=1.1,length=10,width=10,dgray]}},
  wb-row-arrow/.style = {lgray, line width=3pt,
    line cap=round, line join=round,
    arrows={-Stealth[scale=1.1,length=10,width=10,dgray]}}
}

\def\part{\begin{tikzpicture}[scale=0.5,baseline=-1.7] \gbull{0}{0}{0.12}; \end{tikzpicture}}

\newmuskip\pPqskip
\mathchardef\pPcomma=\mathcode`, 

\newcommand*\pPq[6]{%
  \begingroup
  \begingroup\lccode`~=`,
    \lowercase{\endgroup\def~}{\pPcomma\mkern\pPqskip}%
  \mathcode`,=\string"8000
  {}_{#1}\phi_{#2}\biggl[\genfrac..{0pt}{}{#3}{#4};#5;#6\biggr]%
  \endgroup
}

\allowdisplaybreaks[1]

\renewcommand{\b}[1]{\bar{#1}}
\renewcommand{\leq}{\leqslant}
\renewcommand{\geq}{\geqslant}

\definecolor{dred}{RGB}{176,0,0}

\title[Reflection equations and Littlewood identities]{Reflection equations and Littlewood identities for spin Hall-Littlewood and $q$-Whittaker functions}
\author{Svetlana Gavrilova, Sergei Korotkikh}

\begin{document}

\begin{abstract}
    We describe two-parametric Littlewood identities for inhomogeneous stable spin Hall-Littlewood and inhomogeneous spin $q$-Whittaker symmetric functions. These identities generalize the Littlewood identity found by Warnaar (2006) for the ordinary Hall-Littlewood polynomials. However, instead of Rogers-Szeg\"o coefficients used by Warnaar, coefficients in our identities are given by products of Al-Salam-Chihara polynomials. 
    
    Our proof relies on integrable lattice models. A key step includes three new solutions to the reflection and boundary equations for the fully fused six-vertex model. To find these solutions we use isomorphisms between tensor products of evaluation $U_q(L\sl_2)$-modules, as well as covectors stable under a  $q$-Onsager coideal subalgebra. 
\end{abstract}

	\maketitle
    
    \tableofcontents

\section{Introduction}

\subsection{Overview} There are several classical identities evaluating a sum of Schur polynomials as a simple product \cite{Lit50}:
\begin{equation}\label{introeq: Schura}
  \sum_{\lambda} s_\lambda(x_1,\ldots,x_n)
  =\prod_{i=1}^n\frac{1}{1-x_i}\prod_{1\leq i<j\leq n}\frac{1}{1-x_i x_j},
\end{equation}
\begin{equation}\label{introeq: Schurb}
  \sum_{\lambda:\,\lambda\text{ even}} s_\lambda(x_1,\ldots,x_n)
  =\prod_{i=1}^n\frac{1}{1-x^2_i}\prod_{1\leq i<j\leq n}\frac{1}{1-x_i x_j},
\end{equation}
\begin{equation}\label{introeq: Schurc}
  \sum_{\lambda:\,\lambda'\text{ even}} s_\lambda(x_1,\ldots,x_n)
  =\prod_{1\leq i<j\leq n}\frac{1}{1-x_i x_j}.
\end{equation}
Here the sums are over partitions $\lambda$, and the conditions ``$\lambda \text{\ even}$" and ``$\lambda' \text{\ even}$" restrict to partitions with rows or columns of only even lengths respectively. These identities are known as the \emph{Littlewood} identities, and they have numerous generalizations and applications. 
For example, \eqref{introeq: Schurb} and \eqref{introeq: Schurc} encode the
multiplicity-free decompositions of $\mathrm{GL}_n(\mathbb{C})$-modules
$\operatorname{Sym}(\operatorname{Sym}^2\mathbb{C}^n)$ and
$\operatorname{Sym}(\bigwedge^2\mathbb{C}^n)$ respectively. These decompositions enter the
classical branching rules for orthogonal and symplectic groups
\cite{HTW05}. Littlewood identities and their bounded versions play an instrumental role in the combinatorics of plane partitions and tableaux \cite{Des1986}, \cite{Sten90}, \cite{Mac95}, \cite{Pan12}, \cite{BeteaW16}, \cite{RW21}; they can be used to express certain characters of Lie algebras \cite{Okada98}, \cite{Krat98}, \cite{BaWa15}; they appear in Schubert calculus \cite{LLT89} and have elliptic generalizations \cite{LSW09}, \cite{Rains12}.


In \cite{War06} Warnaar gave a two-parametric Littlewood identity for \emph{Hall-Littlewood polynomials} $P_\lambda(x_1, \dots, x_n;q)$. Namely, for $m\in\mathbb Z_{\geq0}$ let $h_m(z)$ denote the \emph{Rogers-Szeg\"o polynomial}
$$
h_m(z)=\sum_{i=0}^m z^i\frac{(q;q)_m}{(q;q)_i(q;q)_{m-i}}.
$$
Then
\begin{equation} \label{introeq: War}
  \sum_{\lambda} h_{\lambda}(u, v) P_\lambda(x_1,\ldots,x_n;q)=\prod_{i=1}^n\frac{(1+ux_i)(1+vx_i)}{(1-x_i)(1+x_i)}\prod_{1\leq i<j\leq n}\frac{1-qx_ix_j}{1-x_i x_j},
\end{equation}
where
$$
h_\lambda(u,v)=\prod_{i\ \text{even}} h_{m_i(\lambda)}(uv)\prod_{i\ \text{odd}}v^{m_i(\lambda)}h_{m_i(\lambda)}(u/v).
$$
The Littlewood identity \eqref{introeq: War} unifies the previous ones. Namely, when $q=0$ and $(u,v)=(0,1), (0,0), (-1,1)$ it reduces to \eqref{introeq: Schura}-\eqref{introeq: Schurc}. Moreover, Hall-Littlewood polynomials are dual to \emph{$q$-Whittaker polynomials} $Q_\lambda(x_1, \dots, x_n;q,0)$, and from this duality one gets
\begin{equation} \label{introeq: Wardual}
  \sum_{\lambda} h_{\lambda'}(u, v) Q_\lambda(x_1,\ldots,x_n;q,0)=\prod_{i=1}^n\frac{1}{(ux_i;q)_\infty(vx_i;q)_\infty}\prod_{1\leq i<j\leq n}\frac{1}{(x_i x_j;q)_\infty}.
\end{equation}

The first goal of this work is to generalize the two-parametric Littlewood identities of Warnaar to \emph{inhomogeneous stable spin Hall-Littlewood functions} and \emph{inhomogeneous spin $q$-Whittaker functions}. These functions generalize Hall-Littlewood and $q$-Whittaker functions by introducing two families of additional parameters originating from integrable lattice models. Spin Hall-Littlewood functions were introduced in \cite{Bor14} as partition functions of the \emph{higher-spin six-vertex model} and then extensively studied and extended in \cite{BP18}, \cite{BW22}. Spin $q$-Whittaker functions were constructed in \cite{BW21} using the \emph{$q$-Hahn vertex model}, and later generalized in \cite{MP22}, \cite{BK24}. 
These polynomials also admit an interesting interpolation characterization \cite{Kor24}. Further fundamental results were obtained in \cite{Muc25}: an explicit torus
orthogonality relation for a slightly modified version of the inhomogeneous spin
$q$-Whittaker polynomials and basis theorems for both modified and original family.
Numerous properties of the classical Hall-Littlewood and $q$-Whittaker functions have spin analogues: symmetry,  branching rules, Cauchy and dual-Cauchy summation identities and orthogonality. Many of these properties follow directly from the integrability of lattice models. For example, Cauchy-type identities are a manifestation of the \emph{Yang-Baxter equation}. 

The other goal of this work comes from the results of \cite{BWZJ15}, \cite{WZJ16}, \cite{Gav23}, \cite{CD21}, where Littlewood identities are studied with the use of \emph{reflection} or \emph{boundary Yang-Baxter} equations. This type of equations was introduced by Cherednik \cite{Cher84} in the study of factorized scattering on a half-line, and later it was used by Sklyanin \cite{Skl88} to define and study quantum integrable systems with a boundary. In the context of the six-vertex model and its generalizations the solutions to reflection equations have been studied in \cite{dVGR93}, \cite{ML19}, \cite{KoMa2024}. However, to the best of our knowledge, the relation needed to prove the Littlewood identity for inhomogeneous spin $q$-Whittaker functions has remained unknown. There are two possible reasons for that: (1) we need to work with an arbitrary spin, while the majority of the results we found only address finite spin; (2) the equation we need has a form that is different from the classical reflection equation. In particular, it involves ``$q$-Hahn vertices", which mix spin and spectral parameters of lines in the model in a unique way.


To find the needed reflection equations, we rely on representations of quantum affine algebras. The weights of lattice models are identified with coefficients of an \emph{$R$-matrix} acting on evaluation $U_q(L\sl_2)$-modules. From this point of view, it is known that the reflection equation is connected to coideal subalgebras of quantum affine algebras and quantum symmetric pairs, see \cite{NS95}, \cite{Let02}, \cite{Kol14}, \cite{RVl2016}, \cite{BaKo2019}, \cite{ApVl2022} and references therein. In the case of $U_q(L\sl_2)$, one of the coideal subalgebras of interest is a parameter-dependent embedding of the $q$-Onsager algebra \cite{Bas05}, \cite{IT10}, \cite{Kol14}. Using this as our motivation, we find the needed reflection equations by describing covectors stabilized by the action of the $q$-Onsager subalgebra on evaluation modules. Our approach is somewhat similar to \cite{Kuniba2019MatrixPS}; however, there are crucial differences. Namely, we work with infinite-dimensional evaluation modules instead of finite-dimensional ones. Moreover, it turns out that it is more suitable in our context to treat the reflection matrix as a pairing rather than an operator.





Another motivation for obtaining Littlewood identities lies in their applications to integrable probability, particularly in studying integrable half-space models. Together with Cauchy identities and branching rules, they provide useful algebraic structures underlying certain probability measures on sequences of partitions.
For Schur symmetric functions, these measures are  Pfaffian Schur processes,
which have been used to obtain exact distribution formulas and
asymptotic fluctuation results for last passage percolation in a
half-quadrant, see \cite{BBCS18}.
Macdonald analogues lead to half-space Macdonald processes studied in \cite{BBC20};
their degenerations were applied there to study the half-space log-gamma polymer.
Other applications include half-space stochastic vertex models and half-line ASEP, see
\cite{BBCW18},  
\cite{He24},
\cite{CD21}.
Also, these measures were used for discovering new integrable interacting particle systems with a non-trivial boundary dynamics such as half-space versions of geometric $q$-TASEP
and $q$-PushTASEP in \cite{BBC20}.
We hope to develop the probabilistic applications of our new identities in a
subsequent work.

During preparation of this work, the paper \cite{CD21} has been significantly updated, and it now includes the two-parametric Littlewood identity for stable spin Hall-Littlewood functions. Their approach is also based on reflection equations and it can be readily generalized to Theorem \ref{thmc}. However, the reflection equations used in \cite{CD21} only involve spin $\frac12$ reflection weights, since those are sufficient for the spin Hall-Littlewood functions. Thus, the rest of our results (Theorems \ref{thma}, \ref{thmb} and \ref{thmd}) are still new, since they are based on reflection equations with arbitrary spin. 

Below we detail our main results.

\subsection{Reflection equations} Let $V(s)$ denote the irreducible $U_q(\sl_2)$ module with highest weight $s^{-1}$, and $V(s)_z$ be the evaluation module induced from $V(s)$ along the evaluation map ${\mathrm{ev}_z: U_q(L\sl_2)\to U_q(\sl_2)}$. The modules $V(s)_z$ are irreducible, and for generic $s$ have a weight basis $\{v_n\}_{n\geq 0}$. We set $V^a_b=V(a/b)_{ab}$. 

In \cite{Kor24} we have studied isomorphisms between tensor products of $V^a_b$. Namely, for generic $a_1,a_2, b_1,b_2$ we have the following isomorphisms of irreducible $U_q(L\sl_2)$-modules:
$$
\tW^a: V^{\sqrt{a_1}}_{\sqrt{b_1}}\otimes V^{\sqrt{a_2}}_{\sqrt{b_2}}\xrightarrow{\sim} V^{\sqrt{a_2}}_{\sqrt{b_1}}\otimes V^{\sqrt{a_1}}_{\sqrt{b_2}},\qquad \tW^b: V^{\sqrt{a_1}}_{\sqrt{b_1}}\otimes V^{\sqrt{a_2}}_{\sqrt{b_2}}\xrightarrow{\sim} V^{\sqrt{a_1}}_{\sqrt{b_2}}\otimes V^{\sqrt{a_2}}_{\sqrt{b_1}}.
$$
Moreover, the matrix coefficients of these isomorphisms are given explicitly by
    $$
    \langle v_k\otimes v_l|\ \tW^a\ |v_i\otimes v_j\rangle=\delta_{i+j=k+l}\frac{(b_1/a_2)^{l/2}}{(b_1/a_1)^{j/2}}(a_2/b_1)^{i-k}\frac{(a_2/b_1;q)_k(a_1/a_2;q)_{i-k}(q;q)_i}{(a_1/b_1;q)_i(q;q)_k(q;q)_{i-k}},
    $$
    $$
    \langle v_k\otimes v_l|\ \tW^b\ |v_i\otimes v_j\rangle=\delta_{i+j=k+l}\frac{(b_2/a_1)^{l/2}}{(b_1/a_1)^{j/2}}(b_1/b_2)^{l}\frac{(a_2/b_1;q)_l(b_1/b_2;q)_{j-l}(q;q)_j}{(a_2/b_2;q)_j(q;q)_l(q;q)_{j-l}}.
    $$   
We call these coefficients \emph{$q$-Hahn weights} since they are similar to the orthogonality weights of $q$-Hahn polynomials. The tilde in our notation is used to distinguish our weights from the \emph{stochastic} weights $W^a$ and $W^b$, which are conventionally used in this context. The stochastic weights can be recovered by a certain gauge transformation, it is described in Section \ref{stochW}.  We view $\tW^a$ and $\tW^b$ as vertex weights and use the schematic notation below, see Section \ref{sec: vertex-models} for a detailed explanation of this notation:
\begin{align*}
		\begin{gathered}
			\begin{tikzpicture}[scale=1]
				\draw[fused] (-1,0) -- (1,0);
				\draw[fused] (0,-1) -- (0,1);
				\node[below] at (-1,-0.1) {\tiny \shortstack[c]{$i; a_1,b_1$}};
				\node[below] at (1,-0.1) {\tiny \shortstack[c]{$l; a_1,b_2$}};
				\node[below] at (0,-1) {\tiny $j; a_2,b_2$};
				\node[above] at (0,1) {\tiny $k; a_2,b_1$};
				\node[above right] at (0,0) {\tiny $\tW^a$};
			\end{tikzpicture}
		\end{gathered}
		&=\tW^a_{a_1,b_1,a_2,b_2}(i,j,k,l):=\langle v_k\otimes v_l|\ \tW^a\ |v_i\otimes v_j\rangle.
	\end{align*}
	\begin{align*}
		\begin{gathered}
			\begin{tikzpicture}[scale=1]
				\draw[fused] (-1,0) -- (1,0);
				\draw[fused] (0,-1) -- (0,1);
				\node[below] at (-1,-0.1) {\tiny \shortstack[c]{$i; a_1,b_1$}};
				\node[below] at (1,-0.1) {\tiny \shortstack[c]{$l; a_2,b_1$}};
				\node[below] at (0,-1) {\tiny $j; a_2,b_2$};
				\node[above] at (0,1) {\tiny $k; a_1,b_2$};
				\node[above right] at (0,0) {\tiny $\tW^b$};
			\end{tikzpicture}
		\end{gathered}
		&=\tW^b_{a_1,b_1,a_2,b_2}(i,j,k,l):=\langle v_k\otimes v_l|\ \tW^b\ |v_i\otimes v_j\rangle.
	\end{align*}

Using $\tW^a$ and $\tW^b$ we can describe the action of the $R$-matrix on evaluation modules. Namely, we have
    $$
\tR=\tW^a\circ\tW^b=\tW^b\circ\tW^a:  V^{\sqrt{a_1}}_{\sqrt{b_1}}\otimes V^{\sqrt{a_2}}_{\sqrt{b_2}}\xrightarrow{\sim}V^{\sqrt{a_2}}_{\sqrt{b_2}}\otimes V^{\sqrt{a_1}}_{\sqrt{b_1}}.
$$
Plugging the expressions for $\tW^a$ and $\tW^b$ results in an explicit expression for $\tR$, which is identical to the formula found in \cite{Man14}.

Our first group of results consists of three reflection equations involving $\tW^a$, $\tW^b$, $\tR$. All these equations involve the same reflection matrix $\tK(i,j)$, which depends on parameters $a,b, \xi,\zeta$. An explicit expression for this matrix is complicated and is deferred to Section \ref{sec: reflectioneq}. However, there is an alternative description of $\tK$, which plays a key role in our arguments.

Let
$$
\omega_{s;\xi,\zeta}(n)=\sum_{i+j=n}\zeta^{i-j}\frac{(s\xi/\zeta;q)_i(s\zeta/\xi;q)_j}{(s^2;q)_{n}}\frac{(q;q)_n}{(q;q)_i(q;q)_j}.
$$
The functions above are known as \emph{Al-Salam-Chihara} $q$-hypergeometric functions. They were introduced in \cite{ASC76} as a solution to a certain convolution problem. However, we find another application of these functions: the  covector
$$
\langle \omega_{s;\xi,\zeta}|\ :=\sum_{n\geq 0} \omega_{s;\xi,\zeta}(n)\langle v_n|\ \in \left(V(s)_z\right)^*
$$
spans the unique one-dimensional subspace of $\left(V(s)_z\right)^*$ invariant under the action of a coideal subalgebra $L(\xi,\zeta)_z\subset U_q(L\sl_2)$ generated by
$$
f_1- q^{-\frac12} e_1k_1^{-1}+\frac{\xi+\xi^{-1}}{q^{\frac12}-q^{-\frac12}} k_1^{-1},\qquad f_0-\frac{q^{\frac12}}{z^2} e_0k_0^{-1}+\frac{q^{\frac12}(\zeta+\zeta^{-1})}{z\left(q^{\frac12}-q^{-\frac12}\right)} k_0^{-1}.
$$
This subalgebra is isomorphic to the \emph{$q$-Onsager} algebra \cite{BK05}. Using the action of $L(\xi,\zeta)_z$ to characterize $\omega_{s;\xi,\zeta}$, we prove the following relation, which defines the reflection matrix $\tK$:
\begin{multline*}
\tK_{a,b;\xi,\zeta}(i,j):=\sum_{k,l\geq 0}\omega_{b^{-1};\xi,\zeta}(k)\omega_{a;\zeta,\xi}(l) \tW^a_{a,b,b^{-1},a^{-1}}(i,j,k,l)\\
=\sum_{k,l\geq 0}\omega_{a;\xi,\zeta}(k)\omega_{b^{-1};\zeta,\xi}(l) \tW^b_{a,b,b^{-1},a^{-1}}(i,j,k,l).
\end{multline*}
We represent $\tK_{a,b;\xi,\zeta}$ and $\omega_{a;\xi,\zeta}$ diagrammatically as follows:
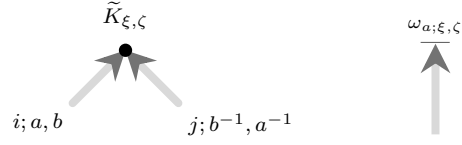
\begin{center}
	\begin{adjustbox}{max width=\textwidth}
		\begin{tikzpicture}[x = 7 mm, y = 7 mm]
			\draw[reflection-arrow] (-1,0) -- (0,1);
            \draw[reflection-arrow] (1,0) -- (0,1);
			\fill[black] (0,1) circle[radius=2.5pt];
			\node[font=\scriptsize, above] at (0,1.2) {$\tK_{\xi,\zeta}$};
			\node[font=\scriptsize, below left] at (-1,0) {$i;a,b$};
            \node[font=\scriptsize, below right] at (1,0) {$j;b^{-1},a^{-1}$};
		\end{tikzpicture}
        \hspace{1cm}
    \begin{tikzpicture}[scale=1.2, baseline=-4pt]
		\draw[fused] (0,0) -- (0,1);
        \node at (0,1) {$\text{---}$};
        \node[above] at (0,1) {\tiny $\omega_{a;\xi,\zeta}$};
\end{tikzpicture} 
	\end{adjustbox}
\end{center}
Using these boundary and reflection weights we find the following two-parametric \emph{boundary relations} for $\tW^a, \tW^b$ and $\tR$:
\begin{theoremAlph}[Boundary relations, Theorems \ref{reflection-theo} and \ref{reflection2-theo} in the text]\label{thma} The following equations hold for generic $a,b,c, \xi,\zeta$ and every $i,j,l\in\mathbb Z_{\geq 0}$:
    $$
\begin{tikzpicture}[scale=1.1, baseline=-4pt]
		\draw[fused] (-1,0) -- (1,0) -- (2,1);
		\draw[fused] (0,-1) -- (0,1);
        \draw[fused] (3,0) -- (2,1);
        \node at (2,1) {$\bullet$};
        \node at (0,1) {$\text{---}$};
	   \node[below] at (-1,-0.1) {\tiny \shortstack[c]{$i; a,b$}};
        \node[above] at (2,1) {\tiny \shortstack[c]{$\widetilde{K}_{\zeta,\xi}$}};
        \node[below] at (3,-0.1) {\tiny \shortstack[c]{$l; b^{-1},c^{-1}$}};
		\node[below] at (0,-1) {\tiny $j; c,a^{-1}$};
        \node[above] at (0,1) {\tiny $\omega_{a;\xi,\zeta}$};
		\node[above right] at (0,0) {\tiny $\widetilde{W}^b$};
\end{tikzpicture} 
\quad
=
\quad
\begin{tikzpicture}[scale=1.1,baseline=-4pt]
		\draw[fused] (1,0) -- (-1,0) -- (-2,1);
		\draw[fused] (0,-1) -- (0,1);
        \draw[fused] (-3,0) -- (-2,1);
        \node at (-2,1) {$\bullet$};
        \node at (0,1) {$\text{---}$};
	   \node[below] at (-3,-0.1) {\tiny \shortstack[c]{$i; a,b$}};
        \node[above] at (-2,1) {\tiny \shortstack[c]{$\widetilde{K}_{\xi,\zeta}$}};
        \node[below] at (1,-0.1) {\tiny \shortstack[c]{$l; b^{-1},c^{-1}$}};
		\node[below] at (0,-1) {\tiny $j; c,a^{-1}$};
        \node[above] at (0,1) {\tiny $\omega_{c;\xi,\zeta}$};
		\node[above right] at (0,0) {\tiny $\widetilde{W}^a$};
\end{tikzpicture} 
$$
$$
\begin{tikzpicture}[scale=1.1, baseline=-4pt]
		\draw[fused] (-1,0) -- (1,0) -- (2,1);
		\draw[fused] (0,-1) -- (0,1);
        \draw[fused] (3,0) -- (2,1);
        \node at (2,1) {$\bullet$};
        \node at (0,1) {$\text{---}$};
	   \node[below] at (-1,-0.1) {\tiny \shortstack[c]{$i; a,b$}};
        \node[above] at (2,1) {\tiny \shortstack[c]{$\widetilde{K}_{\zeta,\xi}$}};
        \node[below] at (3,-0.1) {\tiny \shortstack[c]{$l; b^{-1},a^{-1}$}};
		\node[below] at (0,-1) {\tiny $j; c,c^{-1}$};
        \node[above] at (0,1) {\tiny $\omega_{c;\xi,\zeta}$};
		\node[above right] at (0,0) {\tiny $\tR$};
\end{tikzpicture} 
\quad
=
\quad
\begin{tikzpicture}[scale=1.1,baseline=-4pt]
		\draw[fused] (1,0) -- (-1,0) -- (-2,1);
		\draw[fused] (0,-1) -- (0,1);
        \draw[fused] (-3,0) -- (-2,1);
        \node at (-2,1) {$\bullet$};
        \node at (0,1) {$\text{---}$};
	   \node[below] at (-3,-0.1) {\tiny \shortstack[c]{$i; a,b$}};
        \node[above] at (-2,1) {\tiny \shortstack[c]{$\widetilde{K}_{\xi,\zeta}$}};
        \node[below] at (1,-0.1) {\tiny \shortstack[c]{$l; b^{-1},a^{-1}$}};
		\node[below] at (0,-1) {\tiny $j; c,c^{-1}$};
        \node[above] at (0,1) {\tiny $\omega_{c;\xi,\zeta}$};
		\node[above right] at (0,0) {\tiny $\tR$};
\end{tikzpicture} 
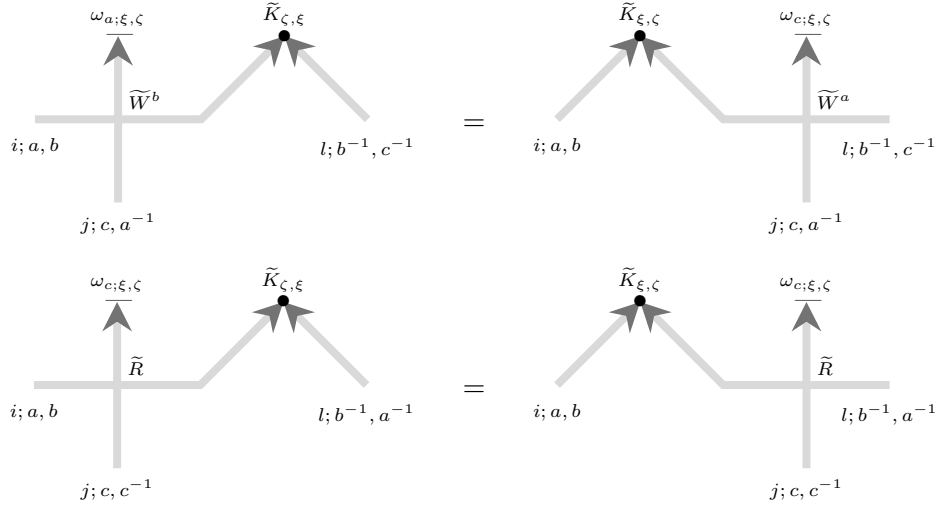
$$
\end{theoremAlph}
The diagrams above indicate how we multiply matrices $\tR,\tW^{a},\tW^b,\tK,\omega$. 
See Section \ref{sec: reflectioneq} for precise descriptions how such diagrams should be understood.

In addition to the boundary relations above, we are able to prove a reflection equation for the fully fused six-vertex model:
\begin{theoremAlph}[Reflection equation, Theorem \ref{thm: reflection3} in the text] \label{thmb} The following equation holds for generic $a,b,c, d, \xi,\zeta$ and every $i,j,k,l\in\mathbb Z_{\geq 0}$:
$$
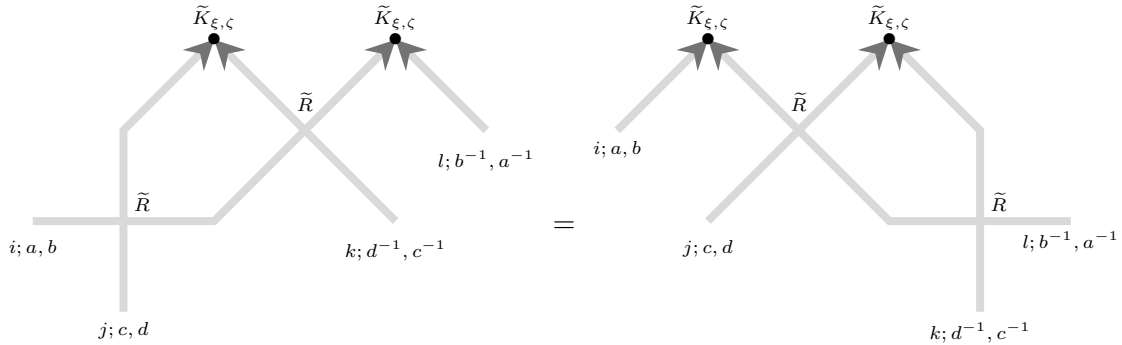
\begin{tikzpicture}[scale=1.2, baseline=-4pt]
		\draw[fused] (-3,0) -- (-1,0) -- (1,2);
		\draw[fused] (1,0) -- (-1,2);
		\draw[fused] (-2,-1) -- (-2,1) -- (-1,2);
        \draw[fused] (2,1) -- (1,2);
        \node at (1,2) {$\bullet$};
        \node at (-1,2) {$\bullet$};
	   \node[below] at (-3,-0.1) {\tiny \shortstack[c]{$i; a,b$}};
       	\node[below] at (-2,-1) {\tiny $j; c,d$};
        \node[below] at (1,-0.1) {\tiny \shortstack[c]{$k; d^{-1},c^{-1}$}};
        \node[below] at (2,0.9) {\tiny \shortstack[c]{$l; b^{-1},a^{-1}$}};
        \node[above] at (1,2) {\tiny \shortstack[c]{$\widetilde{K}_{\xi,\zeta}$}};
        \node[above] at (-1,2) {\tiny \shortstack[c]{$\widetilde{K}_{\xi,\zeta}$}};
		\node[above right] at (-2,0) {\tiny $\tR$};
        \node[above] at (0,1.1) {\tiny $\tR$};
\end{tikzpicture} 
=
\begin{tikzpicture}[scale=1.2, baseline=-4pt]
		\draw[fused] (-1,0) -- (1,2);
		\draw[fused] (3,0) -- (1,0) -- (-1,2);
		\draw[fused] (-2,1) -- (-1,2);
        \draw[fused] (2,-1) -- (2,1) -- (1,2);
        \node at (1,2) {$\bullet$};
        \node at (-1,2) {$\bullet$};
	   \node[below] at (-2,1) {\tiny \shortstack[c]{$i; a,b$}};
       	\node[below] at (-1,-0.1) {\tiny $j; c,d$};
        \node[below] at (2,-1) {\tiny \shortstack[c]{$k; d^{-1},c^{-1}$}};
        \node[below] at (3,0) {\tiny \shortstack[c]{$l; b^{-1},a^{-1}$}};
        \node[above] at (1,2) {\tiny \shortstack[c]{$\widetilde{K}_{\xi,\zeta}$}};
        \node[above] at (-1,2) {\tiny \shortstack[c]{$\widetilde{K}_{\xi,\zeta}$}};
		\node[above right] at (2,0) {\tiny $\tR$};
        \node[above] at (0,1.1) {\tiny $\tR$};
\end{tikzpicture} 
$$
\end{theoremAlph}

When $b/a=d/c=q$, the weights $\tR$ in Theorem \ref{thmb} degenerate to the \emph{six-vertex weights}. In this situation, our result degenerates to the reflection equation found in \cite{dVGR93}, see Remark \ref{rem: degen}. More generally, setting $b/a=q^I$, $d/c=q^J$ for $I,J\in\mathbb Z_{> 0}$ and applying a crossing symmetry, we get the reflection equation for the fused six-vertex $R$ matrix with finite spins. Such equations were studied in \cite{ML19}. However, the reflection weights found there look different from our weights $\tK_{z,zq^I;\xi,\zeta}$. In particular, it is not clear if the expressions in \cite{ML19} can be analytically continued to non-integer spins. It would be interesting to compare our results, however, we do not pursue this direction here.

While the three identities above do not follow directly from each other, they are proved using the same method. Namely, both sides of each identity correspond to different constructions of a certain $L(\xi,\zeta)_z$-invariant linear functional on tensor products of evaluation modules $V^{\sqrt{a}}_{\sqrt{b}}$. From this point of view, our reflection weights $\tK_{su,s/u;\xi,\zeta}$ correspond to an $L(\xi,\zeta)_z$-invariant pairing between $V(s)_{zu}$ and $V(s)_{z/u}$. This is different from the usual reflection equations, where $K$ corresponds to an \emph{operator}, rather than a pairing.

\subsection{Littlewood identities} Using the boundary relations, we prove Littlewood identities for inhomogeneous stable spin Hall-Littlewood functions $\sfF_{\lambda}(x_1, \dots, x_n\mid\calA,\calB)$ and inhomogeneous spin $q$-Whittaker functions $\F_{\lambda}(x_1,\dots, x_n\mid \calA,\calB)$. These functions depend on two families of parameters $\calA=(a_0, a_1, \dots)$ and $\calB=(b_0, b_1, \dots)$, and, after a certain normalization,
when $a_i=0, b_i\to\infty$ for all $i\geq0$, these functions reduce to the classical Hall-Littlewood and $q$-Whittaker symmetric functions. Both $\sfF_\lambda$ and $\F_\lambda$ are defined using vertex models, see Sections 4 and 6 respectively.

To describe our Littlewood identities, we use the following reparameterization of Al-Salam-Chihara polynomials:
$$
h_n(u,v\mid a)=\sum_{i+j=n} u^iv^{i}\frac{(a/v;q)_i(av;q)_j}{(a^2;q)_n}\frac{(q;q)_n}{(q;q)_i(q;q)_j} = (uv)^{\frac{n}2}\, \omega_{a;\sqrt{u/v},\sqrt{uv}}(n).
$$
	\begin{theoremAlph}[Theorem \ref{thm:sHLlittlewood} in the text]\label{thmc}
		For a suitable choice of parameters $\calA$ we have
        $$
			\sum_{\lambda} h_{\lambda}(u,v\mid\calA)\overline{\calA}^{\lambda}\,
			\sfF_{\lambda}(x_1, \dots, x_{n}\mid\calA,\overline{\calA})
			\\
            =\prod_{i=1}^n\frac{(1-ux_i)(1-vx_i)}{1-x^2_i}\prod_{1\leq i<j\leq n}\frac{1-qx_ix_j}{1-x_ix_j},
        $$
        where $\overline{\calA}^{\lambda}=\prod_{j\geq1} a_{j-1}^{-\lambda'_{j}}$ and
        $$
        h_{\lambda}(u,v\mid \calA)=\prod_{\substack{t\geq2\\t-even}}h_{m_t(\lambda)}(u,v\mid a_t)\prod_{\substack{t\geq1\\t-odd}}v^{m_t(\lambda)}h_{m_t(\lambda)}(u,v^{-1}\mid a_t).
        $$
        
		\end{theoremAlph}
        
\begin{theoremAlph}[Theorem \ref{thm:littlewood} in the text]\label{thmd}
		We have 
		\begin{multline*}
			\sum_{\lambda} h_{\lambda'}(u,v\mid\calA)\overline\calA^{\lambda'}\,
			\F_{\lambda}(x_1, \dots, x_{n}\mid\calA,\overline{\calA})
			\\
            =\prod_{i=1}^n\frac{(a_iu, a_iv;q)_\infty}{(a_i^2, x_iu, x_iv;q)_\infty}\cdot\frac{\prod_{i=1}^n\prod_{j=1}^n(x_ia_j;q)_\infty}{\prod_{1\leq i<j\leq n}(x_ix_j, a_ia_j;q)_\infty},
		\end{multline*}
		where $h_\lambda$ and $\overline\calA^{\lambda}$ are the same as in Theorem \ref{thmc}. The identity holds both formally in the space of formal power series in $x_1, \dots, x_n, q$ and numerically for a suitable choice of parameters. 
\end{theoremAlph}

We also provide skew versions of these identities in the main text, as well as degenerations of the form similar to \eqref{introeq: Schura}-\eqref{introeq: Schurc}.

Following the duality between Hall-Littlewood and $q$-Whittaker functions, we interpret the Littlewood identities in Theorms \ref{thmc} and \ref{thmd} as dual versions of each other. However, we do not yet have a rigorous framework for the ``spin" version of this duality, so we have to provide two independent proofs for our identities. It is worth noting that, in the homogeneous case when the parameters $a_i$ in $\calA$ are equal to each other, the spin $q$-Whittaker functions $\F_\lambda$ can be constructed from the stable spin Hall-Littlewood functions $\sfF_\lambda$ using \emph{fusion} \cite{BW21}. In this case Theorem \ref{thmd} follows from Theorem \ref{thmc}. However, we do not expect the fusion argument to work in the fully inhomogeneous case due to differences in how parameters $a_i$ propagate through the vertex models: in the higher spin model the parameters $a_i$ are attached to columns of the lattice, while for the $q$-Hahn model the parameters are attached both to columns and diagonals of the lattice.

\subsection{Layout of the text} The text is organized as follows. In Section 2 we describe the quantum loop algebra $U_q(L\sl_2)$ and its modules $V^a_b$. Using isomorphisms between tensor products $V^{a_1}_{b_1}\otimes V^{a_2}_{b_2}$, in the same section we define the fully fused vertex models used throughout the paper. In Section 3 we use the dual vectors invariant under the action of $q$-Onsager algebra to prove our first main result: the reflection and boundary equations. The remaining sections are dedicated to the Littlewood identities. First, in Section 4 (Section 6) we remind the definition and properties of inhomogeneous stable spin Hall-Littlewood functions $\sfF_\lambda$ (inhomogeneous spin $q$-Whittaker functions $\F_\lambda$ respectively). Then, in Section 5 (Section 7) we use the reflection equation to prove the Littlewood identity for $\sfF_\lambda$ (for $\F_\lambda$ respectively).

	\subsection{Notation}

	Throughout the paper we treat $q^{\frac12}$ either as a formal parameter or as a transcendental complex number satisfying $|q^{\frac12}|<1$. In either case we set $\calF:=\mathbb C(q^{\frac12})$ and $\calF^{\times}=\calF\backslash\{0\}$. We use the standard $q$-Pochhammer symbol
	\[
	(a;q)_n =\begin{cases} \prod_{k=0}^{n-1}(1-aq^k)\qquad &\text {if\ } n\in\Z_{\geq0}, \\
    \prod_{k=1}^{-n}(1-aq^{-k})^{-1}\qquad &\text {if\ } n\in\Z_{\leq0}
    \end{cases}
	\]
	with the convention $(a;q)_0 = 1$. We also set $(a;q)_\infty = \prod_{k=0}^\infty (1-aq^k)$. We also use the $q$-binomial coefficients, which are defined as follows:
    $$
    [n]_q=\frac{q^{\frac n2}-q^{-\frac n2}}{q^{\frac 12} - q^{-\frac 12}},\qquad [n]_q!=[1]_q[2]_q\dots [n]_q,
    $$
    $$
    {\, n\,\brack m}_q=\frac{[n]_q!}{[m]_q![n-m]_q!}=q^{-\frac{m(n-m)}{2}}\frac{(q;q)_n}{(q;q)_m(q;q)_{n-m}}.
    $$

    A \emph{partition} is an infinite integer sequence $\lambda=(\lambda_1\geq\lambda_2\geq\dots\geq 0)$ with all but finitely many entries equal to zero. The \emph{length} of a partition $\lambda$ is denoted by $l(\lambda)$ and is equal to the number of nonzero entries $\lambda_i$. For $k\geq1$ we use $m_k(\lambda)$ to denote the number of indices $i$ such that $\lambda_i=k$. We set $|\lambda|=\lambda_1+\lambda_2+\dots$. We say that $\mu$ \emph{interlaces} $\lambda$ if
	$
	\lambda_1\geq\mu_1\geq\lambda_2\geq\mu_2\geq\lambda_3\geq\cdots.
	$
    In this case we write $\mu\prec\lambda$.

For a sequence $\mathcal X=(\chi_0,\chi_1,\ldots)$ we write
	$$
	\tau^m\mathcal X=(\chi_{m},\chi_{m+1},\ldots),
	\qquad
	\overline{\mathcal X}=(\chi_0^{-1},\chi_1^{-1},\chi_2^{-1},\ldots).
	$$

\subsection{Acknowledgments} The authors are deeply grateful to Alexei Borodin and Jimmy He for their constant support and productive discussions.

The first author was partially supported by the grant SFI-MPS-SIM-00017348 of the Simons Foundation and by the NSF grant DMS-2450323.

\section{Quantum loop algebra $U_q(L\mathfrak{sl}_2)$ and fully fused vertex models} \label{quantum-sect}

In this section we describe the necessary background on the quantum loop algebra $U_q(L\sl_{2})$ and define the representations $V^a_b$. We also describe the isomorphisms $\tW^a$, $\tW^b$, $\tR$ and use them to define vertex models, following \cite{Kor24}. 

\subsection{Quantum loop algebra $U_q(L\sl_{2} )$}

The \emph{quantum loop algebra} $U_q(L\sl_{2})$ is the unital associative algebra over $\calF$ with the generators $k_1^{\pm 1}$, $k_0^{\pm 1}$, $e_1$, $e_0$, $f_1$, $f_0$ with the following relations:
$$
k_ik_i^{-1}=k_i^{-1}k_i=1, \qquad k_0k_1=1,\qquad [e_i,f_j]=\delta_{i,j}\frac{k_i-k_i^{-1}}{q^{\frac12}-q^{-\frac12}},
$$ 
$$
k_ie_i=qe_ik_i,\qquad k_ie_{1-i}=q^{-1}e_{1-i}k_i,\qquad k_if_i=q^{-1} f_ik_i,\qquad k_if_{1-i}=qf_{1-i}k_i,
$$
$$
\sum_{r=0}^{3}(-1)^r{\, 3\,\brack r}_qe_i^re_{1-i}e_{i}^{3-r}=0,\qquad \sum_{r=0}^{3}(-1)^r{\, 3\,\brack r}_qf_i^rf_{1-i}f_{i}^{3-r}=0,
$$
where in all relations $i,j\in\{0,1\}$. The algebra $U_q(L\sl_{2})$ has a Hopf algebra structure, which is given by
$$
\Delta k_i=k_i\otimes k_i, \qquad \Delta e_i=e_i\otimes k_i+1\otimes e_i,\qquad \Delta f_i=f_i\otimes 1 + k_i^{-1}\otimes f_i,
$$
$$
S(e_i)=-e_ik_i^{-1}, \qquad S(f_i)=-k_if_i,\qquad S(k_i^{\pm 1})=k_i^{\mp 1},
$$
$$
\epsilon(k_i)=1,\qquad \epsilon(e_i)=\epsilon(f_i)=0.
$$

There is another presentation of $U_q(L\sl_2)$, introduced by Drinfeld in \cite{Dri87}. In this presentation, the algebra is generated by $k^{\pm1}$, $\{h_{m}\}_{m\in \mathbb Z\backslash\{0\}}$, $\{x^{\pm}_{m}\}_{m\in\mathbb Z}$, with the relations
$$
kk^{-1}=k^{-1}k=1, \qquad kh_{m}=h_{m}k, \qquad kx^{\pm}_{m}=q^{\pm 1}x^{\pm}_{m}k, 
$$ 
$$
 [h_{m}, h_{l}]=0, \qquad [h_{m},x^{\pm}_{l}]=\pm\frac{1}{m}[2m]_q\,x^{\pm}_{m+l}, 
$$
\begin{equation}\label{newrealisation}
x^{\pm}_{m+1}x^{\pm}_{l}-q^{\pm1}x^{\pm}_{l}x^{\pm}_{m+1}=q^{\pm1}x^{\pm}_{m}x^{\pm}_{l+1}-x^{\pm}_{l+1}x^{\pm}_{m},
\end{equation}
$$
\qquad [x_{m}^+,x_{l}^-]=\frac{\phi^+_{m+l}-\phi^-_{m+l}}{q^{\frac12}-q^{-\frac12}}.
$$
Here $\phi^{\pm}_{m}$ are defined as the coefficients in the formal power series
$$
\phi^{\pm}(u)=\sum_{m\geq 0}\phi^{\pm}_{\pm m}u^{\pm m}=k^{\pm 1}\exp\left(\pm(q^{\frac12}-q^{-\frac12})\sum_{m>0} h_{\pm m} u^{\pm m}\right).
$$
We set $\phi^{\pm}_{\pm m}=0$ for $m<0$. The two presentations of $U_q(L\sl_{2})$ are related as follows: 
$$
k_1=k,\quad k_0=k^{-1},\quad e_{1}=x^{+}_0,\quad f_1=x_0^-,\quad e_0=x_{1}^- k^{-1},\quad f_0= k x_{-1}^+.
$$

The quantum algebra $U_q(\sl_2)$ is isomorphic to the Hopf subalgebra of $U_q(L\sl_2)$ generated by $e_1$, $k_1$, $f_1$. Let $\hat U^\pm$ be the subalgebras generated by $\{x^\pm_{m}\}_{m\in\mathbb Z}$, and $\hat U^0$ be the subalgebra generated by $k^{\pm1}$ and $\{h_m\}_{m\in\mathbb Z\backslash\{0\}}$. Then we have the following triangular decomposition \cite{CP94}:
$$
U_q(L\sl_2)=\hat U^-.\hat U^0.\hat U^+.
$$

\subsection{Representations of $U_q(L\sl_2)$} The representations of $U_q(\sl_{2})$ behave in a way similar to the classical $\sl_2$. For a $U_q(\sl_{2})$-module $V$ we say that a nonzero vector $v\in V$ is a weight vector with weight $\rho\in\calF^{\times}$ if $k_1.v=\rho v$. We call a $U_q(\sl_{2})$-module $V$ a \emph{weight module} if 
$$
V=\bigoplus_{\rho\in\mathbb \calF^\times}V_{\rho},\qquad V_\rho:=\{v\in V \mid k_1.v=\rho v\}.
$$
The spaces $V_\rho$ are called the \emph{weight spaces} of $V$, and $\rho$ is called a weight of $V$ if $V_\rho\neq 0$. We say that a $U_q(\sl_{2})$-module $V$ is in the category $\mathcal O$ if 
\begin{itemize}
\item $V$ is a weight module and all its weight spaces are finite-dimensional;
\item All weights of $V$ are contained in $\bigcup_{i=1}^{m}\{\rho_i q^{-\frac{l}{2}}\mid l\in \Z^+\}$ for some collection of weights $\rho_1, \dots, \rho_m\in\calF^{\times}$.
\end{itemize}
For each weight $\rho$ there exists a unique irreducible $U_q(\sl_{2})$-module $V(\rho)\in \mathcal O$ with the highest weight $\rho$ \cite[Chapter 10]{CP94}.

Following \cite{MY14}, we can repeat the formalism above for the loop algebra $U_q(L\sl_{2})$. Let $V$ be a $U_q(L\sl_{2})$-module. We call a nonzero vector $v\in V$ an \emph{$l$-weight vector} if 
$$
\phi^{\pm}_{m}.v=\gamma^{\pm}_{m}v\qquad \text{for\ all\ }m\in\Z,$$
where $\gamma=(\gamma_{m}^\pm)_{m\in \pm\mathbb{Z}_{\geq 0}}$ is a collection of elements in $\calF$ satisfying $\gamma_{0}^+\gamma_{0}^-=1$. Such collections of numbers $\gamma$ are called \emph{$l$-weights}. We say that an $l$-weight $\gamma$ is \emph{rational} if for some rational function $f\in\calF(u)$ the expansions of $f(u)$ around $0$ and $\infty$ are
$$
f(u)=\sum_{m\geq 0} \gamma_{m}^+ u^m, \qquad f(u)=\sum_{m\geq 0} \gamma_{-m}^- u^{-m}. 
$$
In this situation we denote this $l$-weight by $f$. Note that a rational function $f(u)$ defines a rational $l$-weight as long as  $f(u)$ is regular at $0,\infty$ and $f(0)f(\infty)=1$. 

We call an $l$-weight vector $v$ \emph{singular} if $x^+_{m}.v=0$ for all $m$.  A $U_q(L\sl_{2})$-module $V$ is called \emph{a highest $l$-weight module} if $V=U_q(L\sl_{2}).v=\widehat{U}^-.v$ for a singular $l$-weight vector $v$. In this case, the $l$-weight of $v$ is called the \emph{highest $l$-weight} of $V$.

\begin{proposition}[\cite{MY14}]\label{Lmod} Let $f$, $g$ be rational $l$-weights.
\begin{enumerate}
\item There exists a unique up to isomorphism irreducible representation of $U_q(L\sl_{2})$ with the highest $l$-weight $f$, which we denote by $L(f)$.
\item $L(f)\in\mathcal O$ as a $U_q(\sl_{2})$-module.
\item If $L(f)\otimes L(g)$ is an irreducible $U_q(L\sl_{2})$-module, then $L(f)\otimes L(g)\cong L({fg})\cong L(g)\otimes L(f)$.
\end{enumerate}
\end{proposition}

An important example of the irreducible $U_q(L\sl_2)$-modules is given by the family of evaluation modules $V(s)_z$. They depend on a pair of parameters $s,z$ and can be defined by
$$
V(s)_z=L\left(s^{-1}\frac{1-szu}{1-s^{-1}zu}\right).
$$
It is also convenient to use an alternative notation for these modules, where we set 
$$
V^a_b=V(a/b)_{ab}=L\left(\frac{a^{-1}-au}{b^{-1}-bu}\right).
$$
Below we give an explicit description of these modules:

\begin{proposition}[{\cite[Proposition 2.3]{Kor24}}]\label{explicitAction} The representation $V(s)_z$ is infinite-dimensional and has a basis $\{v_{i}\}_{i\in\Z_{\geq 0}}$ if $s^2\notin q^{-\mathbb{Z}_{\geq 0}}$, and is finite-dimensional with a basis $\{v_i\}_{i=0,\dots, m}$ if $s^2=q^{-m}$ for $m\in\Z_{\geq0}$. The action of $U_q(L\sl_2)$ is given explicitly by
\begin{align}
k_1. v_{i}&=s^{-1}q^{-i}\, v_{i},\quad& e_1.v_{i}&=s^{-1}q^{\frac12(-i+1)}[i]_q\, v_{i-1},\quad& f_1.v_{i}&=\frac{1-s^2q^{i}}{q^{\frac12}-q^{-\frac12}}\, v_{i+1},  \label{slaction} \\
k_0. v_{i}&= k_1^{-1}. v_{i}, \quad& e_0.v_i&=z\, f_1.v_i, \quad &f_0. v_{i}&=z^{-1}\, e_1.v_i. \nonumber
\end{align}
Moreover, viewed as a $U_q(\sl_{2})$-module, $V(s)_z$ is irreducible with the highest weight $s^{-1}$.
\end{proposition}

It turns out that a large class of irreducible representations of $U_q(L\sl_2)$ can be obtained by taking tensor products of $V^a_b$. In particular, we have the following statement:

\begin{proposition}
\label{irred}
Let $a_1, \dots, a_m, b_1, \dots, b_m\in\calF^{\times}$ be a collection of parameters satisfying $\frac{a^2_i}{b^2_j}\notin q^{-\mathbb Z_{\geq0}}$ for all $i\neq j$. Then 
$
V^{a_1}_{b_1}\otimes V^{a_2}_{b_2}\otimes \dots \otimes V^{a_m}_{b_m}
$
is irreducible and isomorphic to $L\left(\prod_{i=1}^m\frac{(a_i^{-1}-a_iu)}{(b_i^{-1}-b_iu)}\right)$.
\end{proposition}
\begin{proof}
For the finite-dimensional representations this was proved in \cite{CP91}. The irreducibility in the general claim follows from \cite[Theorem 4.6]{MY14}. Here is the correspondence with our notation:
\begin{itemize}
    \item $q$ in \cite{MY14} corresponds to $q^{\frac12}$ here;
    \item The module $V(\mu)_z$ from \cite{MY14} is the module $V^a_b$ in our notation with $b/a=q^\mu$ and $qab=z$.
\end{itemize}
The second claim follows from the last part of Proposition \ref{Lmod}.
\end{proof}

For a $U_q(L\sl_2)$-module $V$, let $V^*$ denote the \emph{full} dual of $V$. Throughout the text, we treat $V^*$ as a \emph{right} $U_q(L\sl_2)$-module, and we use the notation $\langle v|$ for elements of $V^*$. We also occasionally use $|v\rangle$ to denote vectors $v \in V$.
If $V$ has a countable or finite basis $\{v_i\}_{i \in I}$, then every element of $V^*$ can be written as 
$$
\langle v|=\sum_{i\in I} c_i \langle v_i|,
$$
where $\{\langle v_i|\}_{i\in I}$ are the elements dual to $\{v_i\}_{i\in I}$. Note that the expansion above gives a well-defined element of $V^*$ even when infinitely many $c_i$ are non-zero. 

\subsection{Isomorphisms of tensor products} Proposition \ref{irred} implies that for generic $a_1, a_2, b_1, b_2$ the tensor products $\{V^{a_i}_{b_j}\otimes V^{a_{3-i}}_{b_{3-j}}\}_{i,j=1,2}$ are all isomorphic to each other. In \cite{Kor24} we have found an explicit description for these isomorphisms, which is provided below.

For a pair of parameters $a,b$ and integers $i,j\in\mathbb{Z}$ set
$$
\Phi(i,j,a,b)=b^i\frac{(a;q)_{i}(b;q)_j}{(ab;q)_{i+j}}\frac{(q;q)_{i+j}}{(q;q)_i(q;q)_j}.
$$
In particular, $\Phi(i,j,a,b)=0$ unless $i,j\geq 0$.

\begin{theorem}[{\cite[Theorem 3.1]{Kor24}}] \label{WaWb-theo}
    Let $\sqrt{a_1}, \sqrt{a_2}, \sqrt{b_1}, \sqrt{b_2}\in\calF^{\times}$ be parameters\footnote{Here and throughout the text we often use square roots of certain parameters. To avoid working with the algebraic closure of $\calF$, in such situations we fix a choice of the roots $\sqrt{a},\sqrt{b},\dots\in\calF$ and set $a=(\sqrt{a})^2, b=(\sqrt{b})^2,\sqrt{ab}=\sqrt{a}\sqrt{b}$, etc.   } satisfying $\frac{a_i}{b_j}\notin q^{-\mathbb{Z}_{\geq 0}}$ for all $i,j\in\{1,2\}$.

    (1) The $U_q(L\sl_2)$-modules $V^{\sqrt{a_1}}_{\sqrt{b_1}}\otimes V^{\sqrt{a_2}}_{\sqrt{b_2}}$ and $V^{\sqrt{a_2}}_{\sqrt{b_1}}\otimes V^{\sqrt{a_1}}_{\sqrt{b_2}}$  are irreducible and isomorphic, with an isomorphism given by
    $$
    \tW^a:V^{\sqrt{a_1}}_{\sqrt{b_1}}\otimes V^{\sqrt{a_2}}_{\sqrt{b_2}}\to V^{\sqrt{a_2}}_{\sqrt{b_1}}\otimes V^{\sqrt{a_1}}_{\sqrt{b_2}}
    $$
    $$
    \langle v_k\otimes v_l|\ \tW^a\ |v_i\otimes v_j\rangle=\tW^a_{a_1,b_1,a_2,b_2}(i,j,k,l):=\delta_{i+j=k+l}\frac{(b_1/a_2)^{l/2}}{(b_1/a_1)^{j/2}}\Phi(i-k, k, a_1/a_2, a_2/b_1).
    $$

    (2) The $U_q(L\sl_2)$-modules $V^{\sqrt{a_1}}_{\sqrt{b_1}}\otimes V^{\sqrt{a_2}}_{\sqrt{b_2}}$ and $V^{\sqrt{a_1}}_{\sqrt{b_2}}\otimes V^{\sqrt{a_2}}_{\sqrt{b_1}}$  are irreducible and isomorphic, with an isomorphism given by
    $$
    \tW^b:V^{\sqrt{a_1}}_{\sqrt{b_1}}\otimes V^{\sqrt{a_2}}_{\sqrt{b_2}}\to V^{\sqrt{a_1}}_{\sqrt{b_2}}\otimes V^{\sqrt{a_2}}_{\sqrt{b_1}}
    $$
    $$
    \langle v_k\otimes v_l|\ \tW^b\ |v_i\otimes v_j\rangle=\tW^b_{a_1,b_1,a_2,b_2}(i,j,k,l):=\delta_{i+j=k+l}\frac{(b_2/a_1)^{l/2}}{(b_1/a_1)^{j/2}}\Phi(l, j-l, a_2/b_1, b_1/b_2).
    $$    
\end{theorem}
\begin{proof}
For generic parameters $a_1$, $b_1$, $a_2$, $b_2$ the claim is proved in \cite[Theorem 3.1]{Kor24}. Here we outline that argument to explain why it is still valid in the current setting.

Due to our assumptions on $a_i, b_i$ all the modules $V^{\sqrt{a_i}}_{\sqrt{b_j}}$ are infinite dimensional. Moreover, $\tW^a_{a_1,b_1,a_2,b_2}(i,j,k,l)$ and $\tW^b_{a_1,b_1,a_2,b_2}(i,j,k,l)$ are both well-defined since the symbols $(a_i/b_j;q)_n$ appearing in the denominators are nonzero. By Proposition \ref{irred}, all the tensor products are irreducible and isomorphic. So, by Schur's lemma, it is enough to show that $\tW^a$ and $\tW^b$ commute with the action of $U_q(L\sl_2)$. This is verified directly in \cite[Theorem 3.1]{Kor24}.
\end{proof} 

Using Proposition \ref{irred} and Theorem \ref{WaWb-theo} we can construct isomorphisms for larger tensor products as well. Namely, let $\pi, \rho\in S_m$ be a pair of permutations and $\sqrt{a_1}, \dots, \sqrt{a_m}$, $\sqrt{b_1}, \dots, \sqrt{b_m}$ be parameters such that $\frac{a_i}{b_j}\notin q^{-\mathbb Z_{\geq 0}}$.  By Proposition \ref{irred} the tensor products
$$
V^{\sqrt{a_1}}_{\sqrt{b_1}}\otimes V^{\sqrt{a_2}}_{\sqrt{b_2}}\otimes \dots \otimes V^{\sqrt{a_m}}_{\sqrt{b_m}},\qquad V^{\sqrt{a_{\pi(1)}}}_{\sqrt{b_{\rho(1)}}}\otimes V^{\sqrt{a_{\pi(2)}}}_{\sqrt{b_{\rho(2)}}}\otimes \dots \otimes V^{\sqrt{a_{\pi(m)}}}_{\sqrt{b_{\rho(m)}}}
$$
are both irreducible and isomorphic. So, there is a unique up to a scalar isomorphism between these products, and it can be constructed by decomposing $\pi$ and $\rho$ into simple transpositions and taking an appropriate product of $\tW^a$ and $\tW^b$ isomorphisms. Moreover, different combinations of $\tW^a$ and $\tW^b$ corresponding to $\pi$ and $\rho$ result in the same isomorphism due to Schur's lemma. The last observation is the main reason for developing this formalism, and various identities can be derived using this argument. Below we provide several notable examples:

\begin{theorem}\label{tRdef}
Let $\sqrt{a_1}, \sqrt{a_2}, \sqrt{b_1}, \sqrt{b_2}\in\mathbb \calF^{\times}$ satisfy $\frac{a_i}{b_j}\notin q^{-\mathbb{Z}_{\geq 0}}$ for all $i,j\in\{1,2\}$. Then the $U_q(L\sl_2)$-modules $V^{\sqrt{a_1}}_{\sqrt{b_1}}\otimes V^{\sqrt{a_2}}_{\sqrt{b_2}}$ and $V^{\sqrt{a_2}}_{\sqrt{b_2}}\otimes V^{\sqrt{a_1}}_{\sqrt{b_1}}$  are irreducible and isomorphic. The isomorphism 
$$
\tR: V^{\sqrt{a_1}}_{\sqrt{b_1}}\otimes V^{\sqrt{a_2}}_{\sqrt{b_2}}\to V^{\sqrt{a_2}}_{\sqrt{b_2}}\otimes V^{\sqrt{a_1}}_{\sqrt{b_1}}
$$ 
is explicitly described by 
\begin{multline*}
\tR_{a_1,b_1,a_2,b_2}(i,j,k,l):=\langle v_k\otimes v_l|\ \tR\ |v_i\otimes v_j\rangle\\
=\delta_{i+j=k+l}\frac{(b_2/a_2)^{l/2}}{(b_1/a_1)^{j/2}}\sum_{p=0}^{\min(i,k)} \Phi(i-p, p, a_1/a_2, a_2/b_1)\Phi(l, k-p, a_1/b_1, b_1/b_2)\\
=\delta_{i+j=k+l}\frac{(b_2/a_2)^{l/2}}{(b_1/a_1)^{j/2}}\sum_{p=0}^{\min(j,l)} \Phi(l-p, k, a_1/a_2, a_2/b_2)\Phi(p, j-p, a_2/b_1, b_1/b_2).
\end{multline*}
\end{theorem}
\begin{proof}
    Both $\tW^a\circ \tW^b$ and $\tW^b\circ\tW^a$ define an isomorphism 
    $$
    V^{\sqrt{a_1}}_{\sqrt{b_1}}\otimes V^{\sqrt{a_2}}_{\sqrt{b_2}}\to V^{\sqrt{a_2}}_{\sqrt{b_2}}\otimes V^{\sqrt{a_1}}_{\sqrt{b_1}}.
    $$
    Since the tensor products are irreducible, the isomorphism between them is unique up to a scalar. Comparing the action on $v_0\otimes v_0$, we get $\tW^a\circ \tW^b=\tW^b\circ \tW^a$. The claim follows by setting $\tR=\tW^a\circ \tW^b=\tW^b\circ \tW^a$ and using the explicit expressions from Theorem \ref{WaWb-theo}.
\end{proof}

\begin{theorem}\label{YB} Assume that $\sqrt{a_1}, \sqrt{a_2}, \sqrt{a_3}, \sqrt{b_1}, \sqrt{b_2}, \sqrt{b_3}\in\calF^{\times}$ satisfy $\frac{a_i}{b_j}\notin q^{-\mathbb{Z}_{\geq 0}}$ for all $i,j\in\{1,2,3\}$. The following identity holds for every $i_1, i_2,i_3, j_1,j_2,j_3\in\mathbb Z_{\geq0}$:
\begin{multline*}
\sum_{p_1,p_2,p_3\geq 0}\tR_{a_1,b_1,a_2,b_2}(i_1,i_2,p_2,p_1)\tR_{a_1,b_1,a_3,b_3}(p_1,i_3,p_3,j_1)\tR_{a_2,b_2,a_3,b_3}(p_2,p_3,j_3,j_2)\\
=\sum_{p_1,p_2,p_3\geq 0}\tR_{a_2,b_2,a_3,b_3}(i_2,i_3,p_3,p_2)\tR_{a_1,b_1,a_3,b_3}(i_1,p_3,j_3,p_1)\tR_{a_1,b_1,a_2,b_2}(p_1,p_2,j_2,j_1).  
\end{multline*}
Equivalently, 
$$
(\tR \otimes 1)(1\otimes\tR)(\tR\otimes 1)=(1\otimes\tR)(\tR \otimes 1)(1\otimes\tR),
$$
where both sides are maps
$$
    V^{\sqrt{a_1}}_{\sqrt{b_1}}\otimes V^{\sqrt{a_2}}_{\sqrt{b_2}}\otimes V^{\sqrt{a_3}}_{\sqrt{b_3}}\to V^{\sqrt{a_3}}_{\sqrt{b_3}}\otimes V^{\sqrt{a_2}}_{\sqrt{b_2}}\otimes V^{\sqrt{a_1}}_{\sqrt{b_1}}.
$$
\end{theorem}
\begin{proof}
Both tensor products  in the claim are irreducible by Proposition \ref{irred}. Since $(\tR \otimes 1)(1\otimes\tR)(\tR\otimes 1)$ and $(1\otimes\tR)(\tR \otimes 1)(1\otimes\tR)$ have the same action on $v_0\otimes v_0\otimes v_0$, the claim follows from Schur's Lemma.
\end{proof}

\begin{theorem}\label{WRW-YB} Assume that $\sqrt{a_1}, \sqrt{a_2}, \sqrt{a_3}, \sqrt{b_1}, \sqrt{b_2}, \sqrt{b_3}\in\calF^{\times}$ satisfy $\frac{a_i}{b_j}\notin q^{-\mathbb{Z}_{\geq 0}}$ for all $i,j\in\{1,2,3\}$. The following identity holds for every $i_1, i_2,i_3, j_1,j_2,j_3\in\mathbb Z_{\geq0}$:
\begin{multline*}
\sum_{p_1,p_2,p_3\geq 0}\tW^b_{a_1,b_1,a_2,b_2}(i_1,i_2,p_2,p_1)\tR_{a_2,b_1,a_3,b_3}(p_1,i_3,p_3,j_1)\tW^a_{a_1,b_2,a_3,b_3}(p_2,p_3,j_3,j_2)\\
=\sum_{p_1,p_2,p_3\geq 0}\tW^a_{a_2,b_2,a_3,b_3}(i_2,i_3,p_3,p_2)\tR_{a_1,b_1,a_3,b_2}(i_1,p_3,j_3,p_1)\tW^b_{a_1,b_1,a_2,b_3}(p_1,p_2,j_2,j_1).  
\end{multline*}
Equivalently, 
$$
(\tW^a \otimes 1)(1\otimes\tR)(\tW^b\otimes 1)=(1\otimes\tW^b)(\tR \otimes 1)(1\otimes\tW^a),
$$
where both sides are maps
$$
    V^{\sqrt{a_1}}_{\sqrt{b_1}}\otimes V^{\sqrt{a_2}}_{\sqrt{b_2}}\otimes V^{\sqrt{a_3}}_{\sqrt{b_3}}\to V^{\sqrt{a_3}}_{\sqrt{b_2}}\otimes V^{\sqrt{a_1}}_{\sqrt{b_3}}\otimes V^{\sqrt{a_2}}_{\sqrt{b_1}}.
$$
\end{theorem}
\begin{proof}
Follows from Proposition \ref{irred} and Schur's Lemma.
\end{proof}

\subsection{Vertex models}\label{sec: vertex-models} It is convenient to present identities like Theorem \ref{YB} and \ref{WRW-YB} using certain diagrams, called \emph{vertex models}. We give a description of these diagrams below.

A vertex model consists of a collection of oriented lines with simple intersections. Each intersection is called a \emph{vertex}, and the line segments between these intersections are called \emph{edges}. An edge is \emph{internal} if both of its endpoints are vertices, and \emph{boundary} otherwise. To each edge we attach a pair of parameters, often denoted by letters $a$ or $b$. 

A \emph{configuration} of a vertex model is an assignment of a nonnegative integer label to each edge of the model. It is convenient to interpret vertex model configurations as ensembles of lattice paths oriented along the edges. In our setting multiple paths may share an edge, and each integer label assigned by a configuration tracks the number of paths going through the corresponding edge, thus we also call this integer a \emph{(path) occupation number}. See Figure \ref{fig: path-ens} for a depiction of this correspondence. A \emph{boundary condition}
specifies these occupation numbers on the boundary edges. We usually use letters $i,j,k,l$ to denote the path occupation numbers.\footnote{In different parts of the text we use $k$ as either a generator of $U_q(L\sl_2)$ or as an edge label. In each instance the meaning should be clear from the context.} As a result, in our diagrams the edges are often labeled by triples like $(i;a,b)$, where the first number reflects the configuration of the edge and the latter two are the edge parameters. 

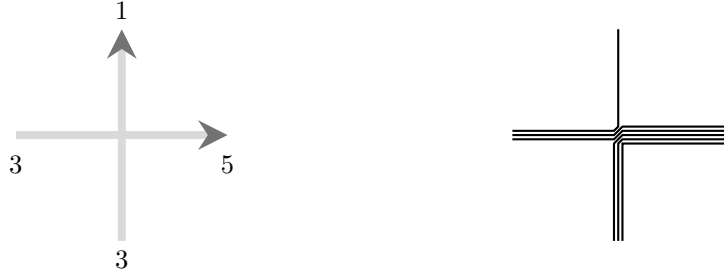
\begin{figure}
\centering
\begin{tikzpicture}[scale=1.4]
				\draw[fused] (-1,0) -- (1,0);
				\draw[fused] (0,-1) -- (0,1);
				\node[below] at (-1,-0.1) {\small \shortstack[c]{$3$}};
				\node[below] at (1,-0.1) {\small \shortstack[c]{$5$}};
				\node[below] at (0,-1) {\small $3$};
				\node[above] at (0,1) {\small $1$};
			\end{tikzpicture}
\hspace{8em}
            \begin{tikzpicture}[scale=1.4]
                \draw[black, line width=0.8pt] (-1,0.04) -- (-0.04, 0.04) -- (0, 0.08) -- (0,1);
                \draw[black, line width=0.8pt] (-1,0) -- (-0.04, 0) -- (0.04, 0.08) -- (1,0.08);
                \draw[black, line width=0.8pt] (-1,-0.04) -- (-0.04, -0.04) -- (0.04, 0.04) -- (1,0.04);
                \draw[black, line width=0.8pt] (-0.04,-1) -- (-0.04, -0.08) -- (0.04, 0) -- (1,0);
                \draw[black, line width=0.8pt] (-0,-1) -- (-0, -0.08) -- (0.04, -0.04) -- (1,-0.04);
                \draw[black, line width=0.8pt] (0.04,-1) -- (0.04, -0.08) -- (0.04, -0.08) -- (1,-0.08);
				\node[below] at (-1,-0.1) {\phantom{\small \shortstack[c]{$3$}}};
				\node[below] at (1,-0.1) {\phantom{\small \shortstack[c]{$5$}}};
				\node[below] at (0,-1) {\phantom{\small $3$}};
				\node[above] at (0,1) {\phantom{\small $1$}};
			\end{tikzpicture}
    \caption{Correspondence between integer labels and path ensembles.}
    \label{fig: path-ens}
\end{figure}

To each vertex we assign a family of \emph{vertex weights} $W(i,j,k,l)$, which depend on the configurations of $4$ edges surrounding the vertex. The weight of the configuration is obtained by multiplying all corresponding vertex weights. Given a vertex model with a boundary condition, a \emph{partition function} is the sum of weights of all configurations which satisfy the boundary condition. 

We mainly use three types of vertex weights in this work, which are inspired by the isomorphisms from the previous section. Namely, we set
\begin{align*}
		\begin{gathered}
			\begin{tikzpicture}[scale=1.3]
				\draw[fused] (-1,0) -- (1,0);
				\draw[fused] (0,-1) -- (0,1);
				\node[below] at (-1,-0.1) {\tiny \shortstack[c]{$i; a_1,b_1$}};
				\node[below] at (1,-0.1) {\tiny \shortstack[c]{$l; a_1,b_2$}};
				\node[below] at (0,-1) {\tiny $j; a_2,b_2$};
				\node[above] at (0,1) {\tiny $k; a_2,b_1$};
				\node[above right] at (0,0) {\tiny $\tW^a$};
			\end{tikzpicture}
		\end{gathered}
		&=\tW^a_{a_1,b_1,a_2,b_2}(i,j,k,l).
	\end{align*}
	\begin{align*}
		\begin{gathered}
			\begin{tikzpicture}[scale=1.3]
				\draw[fused] (-1,0) -- (1,0);
				\draw[fused] (0,-1) -- (0,1);
				\node[below] at (-1,-0.1) {\tiny \shortstack[c]{$i; a_1,b_1$}};
				\node[below] at (1,-0.1) {\tiny \shortstack[c]{$l; a_2,b_1$}};
				\node[below] at (0,-1) {\tiny $j; a_2,b_2$};
				\node[above] at (0,1) {\tiny $k; a_1,b_2$};
				\node[above right] at (0,0) {\tiny $\tW^b$};
			\end{tikzpicture}
		\end{gathered}
		&=\tW^b_{a_1,b_1,a_2,b_2}(i,j,k,l).
	\end{align*}
    \begin{align*}
		\begin{gathered}
			\begin{tikzpicture}[scale=1.3]
				\draw[fused] (-1,0) -- (1,0);
				\draw[fused] (0,-1) -- (0,1);
				\node[below] at (-1,-0.1) {\tiny \shortstack[c]{$i; a_1,b_1$}};
				\node[below] at (1,-0.1) {\tiny \shortstack[c]{$l; a_1,b_1$}};
				\node[below] at (0,-1) {\tiny $j; a_2,b_2$};
				\node[above] at (0,1) {\tiny $k; a_2,b_2$};
				\node[above right] at (0,0) {\tiny $\tR$};
			\end{tikzpicture}
		\end{gathered}
		&=\tR_{a_1,b_1,a_2,b_2}(i,j,k,l).
	\end{align*}
    All these weights vanish unless $i+j=k+l$, so all configurations with non-zero weight satisfy \emph{the conservation law}: the number of paths entering a vertex is equal to the number of paths leaving it. Note that each weight also enforces a certain relation between the edge parameters $a_1$, $a_2$, $b_1$, $b_2$ around the vertex. Since the weights $\tW^a$, $\tW^b$ are identical to the orthogonality weights of $q$-Hahn polynomials, we sometimes refer to these weights as \emph{$q$-Hahn vertex weights}.

    Using the notation above we can give a visual presentation of the relations from the previous subsection. For example, the equality from Theorem \ref{tRdef} is equivalent to
    $$
    
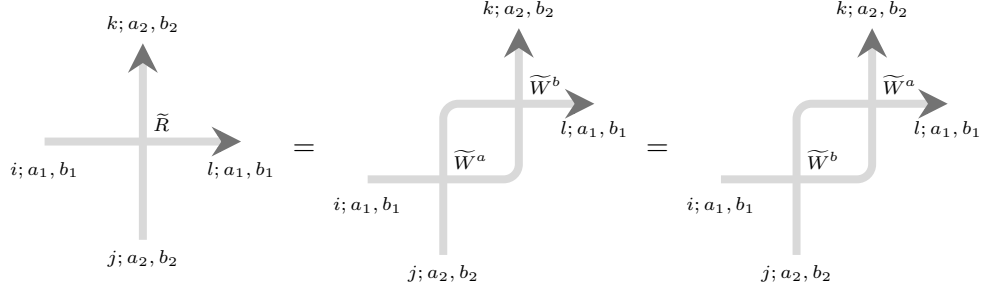
\begin{gathered}
    \begin{tikzpicture}[scale=1.3]
				\draw[fused] (-1,0) -- (1,0);
				\draw[fused] (0,-1) -- (0,1);
				\node[below] at (-1,-0.1) {\tiny \shortstack[c]{$i; a_1,b_1$}};
				\node[below] at (1,-0.1) {\tiny \shortstack[c]{$l; a_1,b_1$}};
				\node[below] at (0,-1) {\tiny $j; a_2,b_2$};
				\node[above] at (0,1) {\tiny $k; a_2,b_2$};
				\node[above right] at (0,0) {\tiny $\tR$};
    \end{tikzpicture}
    \end{gathered}
            =
    \begin{gathered}
    \begin{tikzpicture}[scale=1]
				\draw[fused] (-1,0) -- (0.8,0) arc (-90:0:0.2) -- (1, 2);
				\draw[fused] (0,-1) -- (0,0.8) arc (180:90:0.2) -- (2,1);
				\node[below] at (-1,-0.1) {\tiny \shortstack[c]{$i; a_1,b_1$}};
				\node[below] at (2,0.9) {\tiny \shortstack[c]{$l; a_1,b_1$}};
				\node[below] at (0,-1) {\tiny $j; a_2,b_2$};
				\node[above] at (1,2) {\tiny $k; a_2,b_2$};
				\node[above right] at (0,0) {\tiny $\tW^a$};
                \node[above right] at (1,1) {\tiny $\tW^b$};
    \end{tikzpicture}
    \end{gathered}
    =
                \begin{gathered}
    \begin{tikzpicture}[scale=1]
				\draw[fused] (-1,0) -- (0.8,0) arc (-90:0:0.2) -- (1, 2);
				\draw[fused] (0,-1) -- (0,0.8) arc (180:90:0.2) -- (2,1);
				\node[below] at (-1,-0.1) {\tiny \shortstack[c]{$i; a_1,b_1$}};
				\node[below] at (2,0.9) {\tiny \shortstack[c]{$l; a_1,b_1$}};
				\node[below] at (0,-1) {\tiny $j; a_2,b_2$};
				\node[above] at (1,2) {\tiny $k; a_2,b_2$};
				\node[above right] at (0,0) {\tiny $\tW^b$};
                \node[above right] at (1,1) {\tiny $\tW^a$};
    \end{tikzpicture}
    \end{gathered}
    $$
    Here each diagram denotes the corresponding partition function, where the sum is taken over all possible configurations of the internal edges. Note that we do not label the internal edges in such diagrams, their parameters can be recovered from the boundary conditions and the choice of the vertex weights. Theorem \ref{YB} is equivalent to
    \[
	\stikzbase{1.3}{0}{
		\tikzset{every node/.append style={text=black}}
		\draw[fused]
		(-2,0.5) node[above,scale=0.62] {$i_1; a_1,b_1$} -- (-1,-0.5) -- (1,-0.5) node[right,scale=0.62] {$j_1; a_1,b_1$};
		\draw[fused]
		(-2,-0.5) node[below,scale=0.62] {$i_2; a_2,b_2$}
		-- (-1,0.5) -- (1,0.5) node[right,scale=0.62] {$j_2; a_2,b_2$};
		\draw[fused]
		(0,-1.5) node[below,scale=0.62] {$i_3; a_3,b_3$}
		-- (0,0) -- (0,1.5) node[above,scale=0.62] {$j_3; a_3,b_3$};
		\node[above right] at (0,-0.5) {\tiny $\tR$};
		\node[above right] at (0,0.5) {\tiny $\tR$};
		\node[right] at (-1.4,0) {\tiny $\tR$};
	}
	\qquad=\qquad
    \stikzbase{1.3}{0}{
		\tikzset{every node/.append style={text=black}}
		\draw[fused]
        (-1,1) node[left,scale=0.62] {$i_1; a_1,b_1$}-- (1,1) -- (2,0) node[below,scale=0.62] {$j_1; a_1,b_1$};
		\draw[fused]
		(-1,0) node[left,scale=0.62] {$i_2; a_2,b_2$}
		-- (1,0) -- (2,1) node[above,scale=0.62] {$j_2; a_2,b_2$};
		\draw[fused]
		(0,-1) node[below,scale=0.62] {$i_3; a_3,b_3$}
		-- (0,0.5) -- (0,2) node[above,scale=0.62] {$j_3; a_3,b_3$};
		\node[above right] at (0,0) {\tiny $\tR$};
		\node[above right] at (0,1) {\tiny $\tR$};
		\node[right] at (1.6,0.5) {\tiny $\tR$};
	}
	\]
    Here the two tilted vertices are obtained by rotating the standard vertex while preserving the orientation. Theorem \ref{WRW-YB} can be presented as
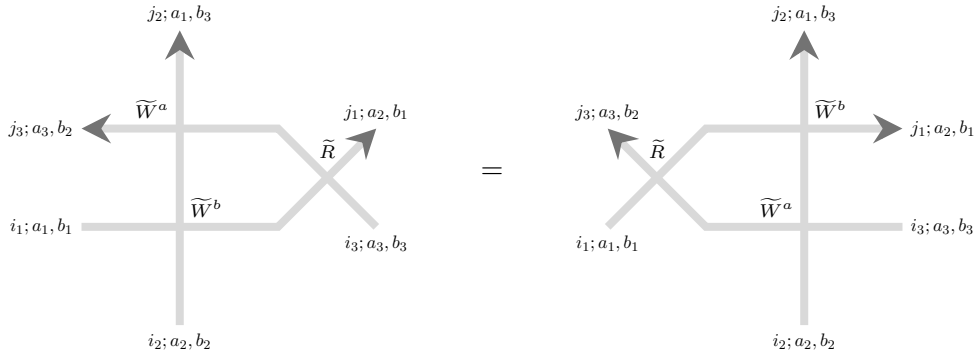
\[
	\stikzbase{1.3}{0}{
		\tikzset{every node/.append style={text=black}}
		\draw[fused]
		(2,0) node[below,scale=0.62] {$i_3; a_3,b_3$}
		-- (1,1) -- (-1,1) node[left,scale=0.62] {$j_3; a_3,b_2$};
		\draw[fused]
		(-1,0) node[left,scale=0.62] {$i_1; a_1,b_1$}
		-- (1,0) -- (2,1) node[above,scale=0.62] {$j_1; a_2,b_1$};
		\draw[fused]
		(0,-1) node[below,scale=0.62] {$i_2; a_2,b_2$}
		-- (0,0.5) -- (0,2) node[above,scale=0.62] {$j_2; a_1,b_3$};
		\node[above right] at (0,0) {\tiny $\tW^b$};
		\node[above left] at (0,1) {\tiny $\tW^a$};
		\node[above] at (1.5,0.6) {\tiny $\tR$};
	}
	\qquad=\qquad
	\stikzbase{1.3}{0}{
		\tikzset{every node/.append style={text=black}}
		\draw[fused]
		(1,-0.5) node[right,scale=0.62] {$i_3; a_3,b_3$}
		-- (-1,-0.5) -- (-2,0.5) node[above,scale=0.62] {$j_3; a_3,b_2$};
		\draw[fused]
		(-2,-0.5) node[below,scale=0.62] {$i_1; a_1,b_1$}
		-- (-1,0.5) -- (1,0.5) node[right,scale=0.62] {$j_1; a_2,b_1$};
		\draw[fused]
		(0,-1.5) node[below,scale=0.62] {$i_2; a_2,b_2$}
		-- (0,0) -- (0,1.5) node[above,scale=0.62] {$j_2; a_1,b_3$};
		\node[above left] at (0,-0.5) {\tiny $\tW^a$};
		\node[above right] at (0,0.5) {\tiny $\tW^b$};
		\node[above] at (-1.5,0.1) {\tiny $\tR$};
	}
	\]

    \begin{remark} One can think of vertex models as a way to visualize products of $\tW^a$, $\tW^b$, $\tR$ acting on tensor products of the evaluation modules $V^{\sqrt{a}}_{\sqrt{b}}$. Namely, an edge with parameters $(a,b)$ corresponds to the module $V^{\sqrt{a}}_{\sqrt{b}}$, while a vertex corresponds to an application of $\tW^a$, $\tW^b$ or $\tR$ to a tensor product of modules corresponding to the incoming edges. Then the partition function computes a matrix coefficient of the appropriate matrix product.
    \end{remark}
    
\subsection{Stochastic weights}\label{stochW} It is convenient to introduce another version of the weights $\tW^a$, $\tW^b$, $\tR$, which we denote by $W^a, W^b, R$ and define by
\begin{multline*}
    W^a_{a_1,b_1,a_2,b_2}(i,j,k,l):=\frac{(b_1/a_1)^{j/2}}{(b_1/a_2)^{l/2}}\tW^a_{a_1,b_1,a_2,b_2}(i,j,k,l)
    =\delta_{i+j=k+l}\Phi(i-k, k, a_1/a_2, a_2/b_1),
\end{multline*}
\begin{multline*}
    W^b_{a_1,b_1,a_2,b_2}(i,j,k,l):=\frac{(b_1/a_1)^{j/2}}{(b_2/a_1)^{l/2}}\tW^b_{a_1,b_1,a_2,b_2}(i,j,k,l)
    =\delta_{i+j=k+l}\Phi(l, j-l, a_2/b_1, b_1/b_2).
\end{multline*}
\begin{multline}\label{eq: stochR}
    R_{a_1,b_1,a_2,b_2}(i,j,k,l):=\frac{(b_1/a_1)^{j/2}}{(b_2/a_2)^{l/2}}\tR_{a_1,b_1,a_2,b_2}(i,j,k,l)\\
=\delta_{i+j=k+l}\sum_{p} \Phi(i-p, p, a_1/a_2, a_2/b_1)\Phi(l, k-p, a_1/b_1, b_1/b_2)\\
=\delta_{i+j=k+l}\sum_{p} \Phi(l-p, k, a_1/a_2, a_2/b_2)\Phi(p, j-p, a_2/b_1, b_1/b_2).
\end{multline}
Note that these weights do not involve the roots $\sqrt{a_i}, \sqrt{b_i}$.

The main reason to introduce these weights is the following stochasticity relation: one can prove that for every $i,j\in\mathbb Z_{\geq0}$ we have
$$
\sum_{k,l\geq0}W^a_{a_1,b_1,a_2,b_2}(i,j,k,l)=\sum_{k,l\geq0}W^b_{a_1,b_1,a_2,b_2}(i,j,k,l)=\sum_{k,l\geq0}R_{a_1,b_1,a_2,b_2}(i,j,k,l)=1.
$$
Because of this property, this normalization is standard in the context of integrable probability, and many objects of our interest are defined in terms of these weights.

To work with the two normalizations of our weights, we define the following diagonal operator $\Theta$ on $V^{\sqrt{a_1}}_{\sqrt{b_1}}\otimes V^{\sqrt{a_2}}_{\sqrt{b_2}}$:
$$
\Theta.(v_i\otimes v_j) =(b_1/a_1)^{j/2}\ v_i\otimes v_j.
$$
One can immediately verify that 
$$
\tW^a=\Theta W^a\Theta^{-1},\quad \tW^b=\Theta W^b\Theta^{-1},\quad \tR=\Theta R \Theta^{-1}.
$$
More generally, consider $V^{\sqrt{a_1}}_{\sqrt{b_1}}\otimes \dots\otimes V^{\sqrt{a_n}}_{\sqrt{b_n}}$ with the parameters $a_i,b_j$ satisfying the condition of Proposition \ref{irred}. Let $\Theta_{i,j}$ denote the operator $\Theta$ acting on $i$th and $j$th factors of this tensor product, and set 
$$
\Theta_{[1,n]}=\prod_{1\leq i<j\leq n}\Theta_{i,j}.
$$
For $X$ being one of the isomorphisms $W^a$, $\tW^a$, $W^b$, $\tW^b$, $R$, $\tR$, let $X_p$ denote the isomorphism $\dots\otimes 1\otimes X\otimes 1\otimes \dots$ acting on the $p$th and $p+1$st factors. 

\begin{proposition}\label{rescale} Let $X$ be either an isomorphism $W^a$, $W^b$ or $R$ and consider a tensor product $V^{\sqrt{a_1}}_{\sqrt{b_1}}\otimes \dots\otimes V^{\sqrt{a_n}}_{\sqrt{b_n}}$ satisfying the conditions of Proposition \ref{irred}. Then for every $p=1,\dots, n-1$ we have
$$
\widetilde{X}_p=\Theta_{[1,n]} X_p \Theta_{[1,n]}^{-1}.
$$
\end{proposition}
\begin{proof} We have $\widetilde{X}_p=\Theta_{p,p+1}\ X_p\ \Theta_{p,p+1}^{-1}$, so we only need to show that the contributions of the other $\Theta_{s,t}^{\pm1}$ cancel out. We have three cases:

{\bfseries Case 1: $s,t\notin\{p,p+1\}$} We have $\Theta_{s,t}\ X_p\ \Theta^{-1}_{s,t}=X_p$ since $\Theta_{s,t}$ and $X_p$ act on different tensor factors. 

{\bfseries Case 2: $s<p$, $t\in\{p,p+1\}$} Note that $\Theta_{1,2}\Theta_{1,3}$ acts on $v_i\otimes v_j\otimes v_k$ by $(b_1/a_1)^{\frac{j+k}2}$. Hence, the conservation law implies 
$$
\Theta_{1,2}\Theta_{1,3} (1\otimes X)=(1\otimes X) \Theta_{1,2}\Theta_{1,3}.
$$
Replacing $1,2,3$ by $s, p, p+1$, we get the cancellation of $\Theta_{s,p}\Theta_{s,p+1}$.

{\bfseries Case 3: $p+1<t$, $s\in\{p,p+1\}$} Note that $\Theta_{2,3}\Theta_{1,3}$ acts on $v_i\otimes v_j\otimes v_k$ by $(\frac{b_1b_2}{a_1a_2})^{\frac{k}2}$. Since the application of $X$ only permutes $a_1,a_2$ and $b_1,b_2$, we get
$$
\Theta_{2,3}\Theta_{1,3} (X\otimes 1)=(X\otimes 1) \Theta_{2,3}\Theta_{1,3}.
$$
Replacing $1,2,3$ by  $p, p+1, t$, we get the cancellation of $\Theta_{p,t}\Theta_{p+1,t}$.
\end{proof}

Here is the practical application of the computation above. Assume we are given a partition function of a vertex model with vertex weights $\tW^a$, $\tW^b$, $\tR$. Then, in order to change all the weights to the stochastic weights $W^a$, $W^b$, $R$, we need to apply $\Theta_{[1,n]}$ to the vector corresponding to the incoming boundary condition and $\Theta_{[1,n]}^{-1}$ to the outgoing boundary condition.

\section{Reflection equation} \label{sec: reflectioneq}

Here we prove the reflection equations involving the vertex weights $\tW^a$, $\tW^b$ and $\tR$. To do so, we first describe one-dimensional subspaces of the dual modules $(V^{\sqrt{a}}_{\sqrt{b}})^*$, which are invariant under a $q$-Onsager coideal subalgebra of $U_q(L\sl_2)$.

\subsection{$q$-Onsager invariant subspaces} Let $\gamma,\sigma,\tau,z$ be parameters such that $\gamma,z\neq 0$. Set
$$
B_i(\gamma,\sigma)=\left(q^{\frac12}-q^{-\frac12}\right)f_i-\gamma (1-q^{-1}) e_ik_i^{-1}+\sigma k_i^{-1}\in U_q(L\sl_2).
$$
Let $L(\xi,\zeta)_z$ be the subalgebra of $U_q(L\sl_2)$ generated by 
$$
B_1(1, \xi+\xi^{-1}),\qquad B_0\left(\frac{q}{z^2},\frac{q^{\frac12}(\zeta+\zeta^{-1})}{z}\right).
$$
This subalgebra is isomorphic to the \emph{$q$-Onsager algebra}, see \cite[Example 7.6]{Kol14}.
The goal of this subsection is to describe one-dimensional subspaces of $\left(V(s)_z\otimes V(t)_z\right)^*$ which are preserved by the action of $L(\xi,\zeta)_z$. We start from describing the eigenvectors of $B_1$ and $B_0$ in $V(s)_z^*$.

\begin{lemma}\label{B1omega} Assume that $\gamma,\sigma, \tau, s, z\in\calF$ satisfy $s, z,\gamma \neq 0$, $s^2\notin q^{-\mathbb{Z}_{\geq 0}}$. Let $\langle\omega|\in V(s)_z^*$ be a nonzero dual vector satisfying  
$$
\langle\omega|\ B_1(\gamma, \sigma)=\tau \langle\omega|.
$$
Then the generating function
$$
F(T)=\sum_{n\geq 0}\langle\omega\mid v_n\rangle \frac{(s^2;q)_n}{(q;q)_n} T^n
$$ 
satisfies
$$
\frac{F(T)}{F(qT)}=\frac{1-\sigma sT+\gamma s^2T^2}{1-\tau T+\gamma T^2}.
$$
\end{lemma}
\begin{proof}
Since $s^2\notin q^{-\mathbb{Z}_{\geq 0}}$, the module $V(s)_z$ is infinite-dimensional. Using Proposition \ref{explicitAction}, for every $n\geq0$ we have
$$
\langle\omega|\left(B_1(\gamma,\sigma)-\tau\right)|v_n\rangle=(1-s^2q^n)c_{n+1}-\gamma(q^n-1)c_{n-1}+\sigma sq^nc_n-\tau c_n=0,
$$
where $c_n:=\langle\omega\mid v_n\rangle$ and we set $c_{-n}=0$ for $n\in\mathbb Z_{> 0}$. Setting $d_n=c_n \frac{(s^2;q)_n}{(q;q)_n}$, the recurrence relation above can be rewritten as 
\begin{equation}
\label{omegarec}
(1-q^{n+1})d_{n+1}+\gamma(1-s^2q^{n-1})d_{n-1}+\sigma sq^nd_n-\tau d_n=0.
\end{equation}
Let $T$ be a formal variable and set $F(T)=\sum_{n\geq 0}d_n T^n$. Note that \eqref{omegarec} holds for all $n\geq -1$, where $n=-1$ case follows from $(1-q^{n+1})=0$. So, multiplying \eqref{omegarec} by $T^n$ and taking the sum over $n\geq -1$, we get
$$
T^{-1}F(T)-T^{-1}F(qT)+\gamma T F(T)-\gamma T s^2 F(qT)+\sigma s F(qT)-\tau F(T)=0.
$$
Rearranging the terms yields the claim. Note that the division by $F(qT)$ is well-defined since $\langle\omega|\neq 0$.
\end{proof}

\begin{lemma}\label{B0flip} For $s,z\in\calF^\times$ let $\pi_{V(s)_z}:U_q(L\sl_2)\to \End(V(s)_z)$ denote the action of $U_q(L\sl_2)$ on $V(s)_z$. Assume that $\gamma,\sigma,\tau\in\calF$ and $\gamma\neq 0$. Then
$$
\pi_{V(s)_z}\left(B_0\left(\frac{q}{\gamma z^2},\frac{q^{\frac12}\tau}{\gamma z}\right)-\frac{q^{\frac12}\sigma}{\gamma z}\right)=-\frac{q^{\frac12}}{\gamma z} \pi_{V(s)_z}\left(B_1(\gamma, \sigma)k_1-\tau k_1\right).
$$
\end{lemma}
\begin{proof} Follows directly from 
$$
k_0=k_1^{-1},\qquad\pi_{V(s)_z}(f_0)=z^{-1} \pi_{V(s)_z}(e_1),\qquad \pi_{V(s)_z}(e_0)=z \pi_{V(s)_z}(f_1).
$$
\end{proof}

\begin{proposition}\label{Bomega1} Let $s,\xi,\zeta,z\in\calF^\times$ and assume $s^2\notin q^{-\mathbb{Z}_{\geq 0}}$. There exists a unique up to a scalar nonzero covector $\langle\omega_{s;\xi,\zeta}|\in V(s)_z^*$ such that 
$$
\langle\omega_{s;\xi,\zeta}|\ L(\xi,\zeta)_z\subset\calF\langle\omega_{s;\xi,\zeta}|.
$$
This covector satisfies
$$
\langle\omega_{s;\xi,\zeta}|\ B_1(1,\xi+\xi^{-1})=(\zeta+\zeta^{-1})\langle\omega_{s;\xi,\zeta}|,
$$
$$
\langle\omega_{s;\xi,\zeta}|\ B_0\left(\frac{q}{z^2},\frac{q^{\frac12}(\zeta+\zeta^{-1})}{z}\right)=\frac{q^{\frac12}(\xi+\xi^{-1})}{z}\langle\omega_{s;\xi,\zeta}|,
$$
and it is determined by
$$
\langle\omega_{s;\xi,\zeta}|=\sum_{n\geq 0} \omega_{s;\xi,\zeta}(n) \langle v_n |,
$$ 
$$
\sum_{n\geq 0}\omega_{s;\xi,\zeta}(n)\frac{(s^2;q)_n}{(q;q)_n}T^n=\frac{(s\xi T;q)_\infty(s\xi^{-1} T;q)_\infty}{(\zeta T;q)_\infty(\zeta^{-1} T;q)_\infty},
$$
where $T$ is a formal variable. Equivalently,
\begin{equation}\label{omegaex}
\omega_{s;\xi,\zeta}(n)=\zeta^{-n}\sum_{i+j=n}(\xi\zeta/s)^i\Phi(i,j,s\xi/\zeta, s\zeta/\xi).
\end{equation}
\end{proposition}
\begin{proof}
For now let $\sigma,\tau$ be arbitrary parameters. Assume that for $\langle \omega|\in V(s)_z^*$ and $\lambda,\mu\in\calF$ we have 
$$
\langle\omega|\ B_1(1,\sigma)=\lambda\langle\omega|,
$$
$$
\langle\omega|\ B_0\left(\frac{q}{z^2},\frac{q^{\frac12}\tau}{z}\right)=\frac{q^{\frac12}\mu}{z}\langle\omega|.
$$
Using Lemma \ref{B0flip}, we can rewrite the last equation as
$$
\langle\omega|\ B_1(1,\mu)=\tau\langle\omega|.
$$
Let $T$ be a formal variable and set $F(T)=\sum_{n\geq 0}\langle \omega\mid v_n\rangle \frac{(s^2;q)_n}{(q;q)_n}T^n$. Lemma \ref{B1omega} implies
$$
\frac{F(T)}{F(qT)}=\frac{1-\sigma sT+ s^2T^2}{1-\lambda T+T^2}=\frac{1-\mu sT+ s^2T^2}{1-\tau T+ T^2}.
$$
Taking the common denominator and comparing the coefficients of $T$, $T^2$ and $T^3$, we get
$$
(\sigma-\mu)s=\lambda-\tau,\qquad \sigma\tau s=\mu\lambda s,\qquad (\sigma-\mu)s=(\lambda-\tau)s^2.
$$
Since $s^2\neq1$ and $s\neq 0$, this implies $\lambda=\tau$, $\mu=\sigma$. 

Now set $\sigma=\xi+\xi^{-1}$ and $\tau=\zeta+\zeta^{-1}$. The elements $B_1(1,\sigma)$ and $B_0\left(\frac{q}{ z^2},\frac{q^{\frac12}\tau}{z}\right)$ become the generators of $L(\xi,\zeta)_z$, and the generating function $F(T)$ becomes
$$
\frac{F(T)}{F(qT)}=\frac{1-\sigma sT+ s^2T^2}{1-\tau T+ T^2}=\frac{(1-s\xi T)(1-s\xi^{-1} T)}{(1-\zeta T)(1-\zeta^{-1} T)}.
$$
Let 
$$
G(T)=\frac{(\zeta T;q)_\infty(\zeta^{-1} T;q)_\infty}{(s\xi T;q)_\infty(s\xi^{-1} T;q)_\infty}F(T).
$$
Then $G(T)=G(qT)$, which forces $G(T)=G(0)=F(0)=\langle\omega\mid v_0\rangle$. Hence
$$
\sum_{n\geq 0}\langle \omega\mid v_n\rangle \frac{(s^2;q)_n}{(q;q)_n}T^n=F(T)=\langle\omega\mid v_0\rangle\frac{(s\xi T;q)_\infty(s\xi^{-1} T;q)_\infty}{(\zeta T;q)_\infty(\zeta^{-1} T;q)_\infty}.
$$
This determines $\langle \omega|$ uniquely up to a scalar, and by rescaling $\omega$ we can set $\langle\omega\mid v_0\rangle=1$. 

To get the explicit expression \eqref{omegaex} we use the $q$-binomial theorem:
\begin{equation}\label{eq: qbinom}
\sum_{n\geq 0} \frac{(a;q)_n}{(q;q)_n}T^n=\frac{(aT;q)_\infty}{(T;q)_\infty}.
\end{equation}
It implies that
$$
\frac{(s\xi T;q)_\infty(s T/\xi;q)_\infty}{(\zeta T;q)_\infty( T/\zeta;q)_\infty}=\left(\sum_{i\geq 0}\frac{(s\xi/\zeta;q)_i}{(q;q)_i}(\zeta T)^i\right)\left(\sum_{j\geq 0}\frac{(s\zeta/\xi;q)_j}{(q;q)_j}(T/\zeta)^j\right).
$$
Hence
\begin{multline*}
\langle\omega\mid v_n\rangle=\sum_{i+j=n}\frac{(s\xi/\zeta;q)_i(s\zeta/\xi;q)_j}{(s^2;q)_n}\frac{(q;q)_n}{(q;q)_i(q;q)_j}\zeta^{i-j}\\
=\zeta^{-n}\sum_{i+j=n}(\xi\zeta/s)^i\Phi(i,j,s\xi/\zeta, s\zeta/\xi). \qedhere
\end{multline*}
\end{proof}

Now we are ready for our key statement.

\begin{theorem}\label{Bdoubleomega} Let $\xi, \zeta, s, t, z\in \calF^\times$ be parameters satisfying 
\begin{equation*}
s^2,t^2, st\notin q^{-\mathbb{Z}_{\geq 0}}. 
\end{equation*}
Then there is a unique up to a scalar element $\langle\omega|\in\left(V(s)_z\otimes V(t)_z\right)^*$ satisfying
$$
\langle \omega |\ B_1(1, \xi+\xi^{-1})=(\xi+\xi^{-1})\langle \omega |,\quad \langle \omega |\ B_0\left(\frac{q}{z^2}, \frac{q^{\frac12}(\zeta+\zeta^{-1})}{z}\right)=\frac{q^{\frac12}(\zeta+\zeta^{-1})}{z}\langle \omega |.
$$
This element $\langle \omega |$ is proportional to $\langle\omega_{s;\zeta,\xi}|\otimes \langle \omega_{t;\xi,\zeta}|$.
\end{theorem}
Note that $\xi$ and $\zeta$ alternate in the expression for $\langle\omega|$.
\begin{proof} Assume that $\langle\omega|\in \left(V(s)_z\otimes V(t)_z\right)^*$ satisfies the conditions from the claim. Note that for every $i\in\{0,1\}$ we have
$$
\Delta B_i(\gamma,\sigma)= k_i^{-1}\otimes (B_i(\gamma,\sigma)-\tau)+B_i(\gamma,\tau)\otimes 1.
$$
Setting $\gamma=1$, $\sigma=\xi+\xi^{-1}$, $\tau=\zeta+\zeta^{-1}$ we get
\begin{equation}\label{B1relation}
\langle\omega|\ \left(k_1^{-1}\otimes (B_1(1,\sigma)-\tau)+B_1(1,\tau)\otimes 1\right) =\langle\omega|\ B_1(1,\sigma)=\sigma \langle\omega|. 
\end{equation}
A similar computation for $B_0$ yields
\begin{multline*}
\langle\omega|\left(k_0^{-1}\otimes \left[B_0\left(\frac{q}{z^2},\frac{q^{\frac12}\tau}{z}\right)-\frac{q^{\frac12}\sigma}{ z}\right]+\left[B_0\left(\frac{q}{z^2},\frac{q^{\frac12}\sigma}{z}\right)-\frac{q^{\frac12}\tau}{z}\right]\otimes 1\right)\\
=\langle\omega|\left(B_0\left(\frac{q}{z^2},\frac{q^{\frac12}\tau}{z}\right)-\frac{q^{\frac12}\tau}{z}\right)=0.
\end{multline*}
Applying Lemma \ref{B0flip}, we get
\begin{equation}\label{B0relation}
\langle\omega|\ \left(k_1\otimes \left(B_1\left(1,\sigma\right)k_1-\tau k_1\right)+(B_1\left(1,\tau\right)k_1-\sigma k_1)\otimes 1\right)=0.
\end{equation}
Multiplying \eqref{B1relation} by $k_1\otimes 1$ and \eqref{B0relation} by $k_1^{-1}\otimes k_1^{-1}$ results in the following system of relations:
\begin{equation}\label{Bsystemrelation}
\begin{split}
\langle\omega|\ \left(1\otimes [B_1(1,\sigma)-\tau]+[B_1(1,\tau)k_1-\sigma k_1]\otimes 1\right) =0,\\
\langle\omega|\ \left(1\otimes \left[B_1\left(1,\sigma\right)-\tau\right]+[B_1\left(1,\tau\right)-\sigma]\otimes k_1^{-1}\right)=0.
\end{split}
\end{equation}
Taking the difference of the two expressions above, we arrive at
$$
\langle\omega|\left(\left[B_1(1,\tau)-\sigma\right]\otimes 1\right)\left(k_1\otimes 1-1\otimes k_1^{-1}\right)=0.
$$
The element $k_1\otimes 1-1\otimes k_1^{-1}$ acts on $V(s)_z\otimes V(t)_z$ as a diagonal matrix with entries $\{s^{-1}q^{-n}-tq^{m}\}_{n,m\in\mathbb Z_{\geq 0}}$. Since $st\notin q^{-\mathbb{Z}_{\geq0}}$, all these entries are nonzero, hence the action of $k_1\otimes 1-1\otimes k_1^{-1}$ is invertible and we get
$$
\langle\omega|\left[B_1(1,\tau)-\sigma\right]\otimes 1=0.
$$
Together with \eqref{Bsystemrelation}, this implies
\begin{equation}
\label{splitB}
\langle\omega|\left[B_1(1,\tau)-\sigma\right]\otimes 1=\langle\omega|\, 1\otimes\left[B_1(1,\sigma)-\tau\right]=0.
\end{equation}

Consider the expansion $\langle\omega|=\sum_{n\geq 0} \langle v_n|\otimes \langle w_n|$, where $\langle v_n|\in V(t)_z^*$ are dual to the standard basis $\{v_n\}_{n\geq 0}$ of $V(s)_z$ and $\langle w_n|\in V(t)_z^*$. Then, for all $n\geq0$ we have 
$$
\langle w_n| B_1(1,\sigma)=\tau \langle w_n|.
$$ 
Remembering the values of $\sigma,\tau$ and applying Lemma \ref{B1omega} together with Proposition \ref{Bomega1} leads to $\langle w_n|=c_n\langle \omega_{t;\xi,\zeta}|$ for some $c_n\in\calF$. Hence, $\langle\omega|=\langle w| \otimes \langle \omega_{t;\xi,\zeta}|$ for $\langle w|=\sum_{n\geq 0} c_n\langle v_n|\in V(s)_z^*$. From \eqref{splitB} we have 
$$
\langle w|\, B_1(1,\tau)=\sigma\langle w|,
$$
which implies $\langle w|=c\langle \omega_{s;\zeta,\xi}|$ for $c\in\calF$. 
\end{proof}

Theorem \ref{Bdoubleomega} allows us to characterize $\langle\omega_{s;\zeta,\xi}|\otimes \langle \omega_{t;\xi,\zeta}|$ using the action of $L(\xi,\zeta)_z$. Since this action commutes with the isomorphisms $\tR$, $\tW^a$, $\tW^b$, we get a natural collection of covectors, consistent with these isomorphisms. We summarize this observation below.    

\begin{proposition}\label{dualeigenR}
    (1) Let $s,t,z,\xi,\zeta\in\calF$ be parameters satisfying the conditions of Theorem \ref{Bdoubleomega}. We have
    $$
    \langle\omega_{t;\xi,\zeta}|\,\otimes \langle \omega_{s;\zeta,\xi}|\, \tR=\langle\omega_{s;\xi,\zeta}|\,\otimes \langle \omega_{t;\zeta,\xi}|,
    $$
    where $\widetilde{R}:V(s)_z\otimes V(t)_z\to V(t)_z\otimes V(s)_z$ is defined in Theorem \ref{tRdef}. In other words, for every $i,j\in\mathbb Z_{\geq 0}$ we have
    $$
    \sum_{k,l\geq 0}\omega_{t;\xi,\zeta}(k)\omega_{s;\zeta,\xi}(l) \tR_{sz,s^{-1}z, tz, t^{-1}z}(i,j,k,l)=\omega_{s;\xi,\zeta}(i)\omega_{t;\zeta,\xi}(j).
    $$

    (2) Let $\sqrt{a}, \sqrt{b}, \xi, \zeta\in\calF^\times$ be parameters satisfying $a^2, b^{-2}, a/b\notin q^{-\mathbb Z_{\geq0}}$ . Then
    $$
        \langle\omega_{b^{-1};\xi,\zeta}|\,\otimes \langle \omega_{a;\zeta,\xi}|\, \tW^a=\langle\omega_{a;\xi,\zeta}|\,\otimes \langle \omega_{b^{-1};\zeta,\xi}|\, \tW^b,
    $$
    where 
    $$
    \tW^a:V^{\sqrt{a}}_{\sqrt{b}}\otimes V^{\sqrt{b}^{-1}}_{\sqrt{a}^{-1}}\to V^{\sqrt{b}^{-1}}_{\sqrt{b}}\otimes V^{\sqrt{a}}_{\sqrt{a}^{-1}}, \quad  \tW^b:V^{\sqrt{a}}_{\sqrt{b}}\otimes V^{\sqrt{b}^{-1}}_{\sqrt{a}^{-1}}\to V^{\sqrt{a}}_{\sqrt{a}^{-1}}\otimes V^{\sqrt{b}^{-1}}_{\sqrt{b}}
    $$ 
    are defined in Theorem \ref{WaWb-theo}. Equivalently, for every $i,j\in\mathbb Z_{\geq 0}$ we have
    \begin{multline}\label{omegaW}
    \sum_{k,l\geq 0}\omega_{b^{-1};\xi,\zeta}(k)\omega_{a;\zeta,\xi}(l) \tW^a_{a,b,b^{-1},a^{-1}}(i,j,k,l)\\
    =\sum_{k,l\geq 0}\omega_{a;\xi,\zeta}(k)\omega_{b^{-1};\zeta,\xi}(l) \tW^b_{a,b,b^{-1},a^{-1}}(i,j,k,l).
    \end{multline}
\end{proposition}
\begin{proof}
(1) Recall that $\tR$ commutes with the action of $U_q(L\sl_2)$. Theorem \ref{Bdoubleomega} yields
$$
\langle\omega_{t;\xi,\zeta}|\,\otimes \langle \omega_{s;\zeta,\xi}|\, \tR B_1(1, \zeta+\zeta^{-1})=(\zeta+\zeta^{-1})\langle\omega_{t;\xi,\zeta}|\,\otimes \langle \omega_{s;\zeta,\xi}|\,\tR,
$$
$$
\langle\omega_{t;\xi,\zeta}|\,\otimes \langle \omega_{s;\zeta,\xi}|\, \tR B_0\left(\frac{q}{z^2}, \frac{q^{\frac12}(\xi+\xi^{-1})}{z}\right)= \frac{q^{\frac12}(\xi+\xi^{-1})}{z} \langle\omega_{t;\xi,\zeta}|\,\otimes \langle \omega_{s;\zeta,\xi}|\,\tR.
$$
So, $\langle\omega_{t;\xi,\zeta}|\,\otimes \langle \omega_{s;\zeta,\xi}|\, \tR$ is a vector in $(V(s)_z\otimes V(t)_z)^*$ satisfying the conditions of Theorem \ref{Bdoubleomega} with $\xi$ and $\zeta$ swapped. Hence
$$
\langle\omega_{t;\xi,\zeta}|\,\otimes \langle \omega_{s;\zeta,\xi}|\, \tR= c\langle\omega_{s;\xi,\zeta}|\,\otimes \langle \omega_{t;\zeta,\xi}|,
$$
where $c\in\calF$. Since $\langle v_0|\otimes\langle v_0|\, \tR=\langle v_0|\otimes\langle v_0|$, we get $c=1$. 

(2) Applying Schur's Lemma and Proposition \ref{irred}, we see that $\tR\circ \tW^b=\tW^a$: both are isomorphisms of irreducible modules $V^{\sqrt{a}}_{\sqrt{b}}\otimes V^{\sqrt{b}^{-1}}_{\sqrt{a}^{-1}}\to V^{\sqrt{b}^{-1}}_{\sqrt{b}}\otimes V^{\sqrt{a}}_{\sqrt{a}^{-1}}$ sending $v_0\otimes v_0$ to $v_0\otimes v_0$. Setting $s=a$, $t=b^{-1}$ and $z=1$, part (1) implies
$$
\langle\omega_{b^{-1};\xi,\zeta}|\,\otimes \langle \omega_{a;\zeta,\xi}|\, \tW^a=\langle\omega_{b^{-1};\xi,\zeta}|\,\otimes \langle \omega_{a;\zeta,\xi}|\, \tR\tW^b=\langle\omega_{a;\xi,\zeta}|\,\otimes \langle \omega_{b^{-1};\zeta,\xi}|\,\tW^b.
$$
\end{proof}
\subsection{Reflection weights $K_{\xi,\zeta}$} 
The covector 
$$
\langle\omega_{b^{-1};\xi,\zeta}|\,\otimes \langle \omega_{a;\zeta,\xi}|\, \tW^a=\langle\omega_{a;\xi,\zeta}|\,\otimes \langle \omega_{b^{-1};\zeta,\xi}|\,\tW^b
$$
from the second part of Proposition \ref{dualeigenR} is the bilinear pairing defining our reflection matrix $\tK$. Namely, we set
\begin{multline}\label{tKdef}
\tK_{a,b;\xi,\zeta}(i,j):=\sum_{k,l\geq 0}\omega_{b^{-1};\xi,\zeta}(k)\omega_{a;\zeta,\xi}(l) \tW^a_{a,b,b^{-1},a^{-1}}(i,j,k,l)\\
=\sum_{k,l\geq 0}\omega_{a;\xi,\zeta}(k)\omega_{b^{-1};\zeta,\xi}(l) \tW^b_{a,b,b^{-1},a^{-1}}(i,j,k,l).
\end{multline}
Recalling the expressions for $\tW^a,\tW^b$ and $\omega$ from Theorem \ref{WaWb-theo} and Proposition \ref{Bomega1}, this definition is equivalent to
\begin{align*}
\tK_{a,b;\xi,\zeta}(i,j)&=\begin{aligned}[t]
(b/a)^{-\frac j2}\sum_{\substack{k,l,p,r\geq 0\\k+l+p+r=i+j}}&\xi^{k-r}\zeta^{p-l}b^{k+p+r}a^{-p}\ \Phi\left(k,l,\frac{\xi}{\zeta b}, \frac{\zeta}{\xi b}\right)\\
&\times\Phi\left(p,r,\frac{a\zeta}{\xi}, \frac{a\xi}{\zeta}\right)\Phi\left(i-k-l, k+l, ab, \frac1{b^2}\right)
\end{aligned}\\
&=\begin{aligned}[t]
(b/a)^{-\frac j2}\sum_{\substack{k,l,p,r\geq 0\\k+l+p+r=i+j}}&\xi^{k-r}\zeta^{p-l}a^{-k-p-r}b^p\ \Phi\left(k,l,\frac{a\xi}{\zeta}, \frac{a\zeta}{\xi}\right)\\
&\times \Phi\left(p,r,\frac{\zeta}{\xi b}, \frac{\xi}{\zeta b}\right)\Phi\left(p+r, j-p-r, \frac1{b^2}, ab\right).
\end{aligned}
\end{align*}

In terms of vertex models, we use the following notation for the reflection weights:
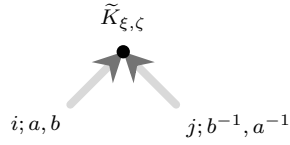
\begin{center}
	\begin{adjustbox}{max width=\textwidth}
		\begin{tikzpicture}[x = 7 mm, y = 7 mm]
			\draw[reflection-arrow] (-1,0) -- (0,1);
            \draw[reflection-arrow] (1,0) -- (0,1);
			\fill[black] (0,1) circle[radius=2.5pt];
			\node[font=\scriptsize, above] at (0,1.2) {$\tK_{\xi,\zeta}$};
			\node[font=\scriptsize, below left] at (-1,0) {$i;a,b$};
            \node[font=\scriptsize, below right] at (1,0) {$j;b^{-1},a^{-1}$};
		\end{tikzpicture}
	\end{adjustbox}
\end{center}
It is also convenient to introduce a boundary condition defined by the weights $\omega_{a;\xi,\zeta}$. It sets the parameters of the corresponding boundary edge to $(a,a^{-1})$, and takes the sum over all non-negative path occupation numbers $i\in\mathbb Z_{\geq 0}$ with weights $\omega_{a;\xi,\zeta}(i)$. We denote this boundary condition by 
$$
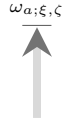
\begin{tikzpicture}[scale=1.2, baseline=-4pt]
		\draw[fused] (0,0) -- (0,1);
        \node at (0,1) {$\text{---}$};
        \node[above] at (0,1) {\tiny $\omega_{a;\xi,\zeta}$};
\end{tikzpicture} 
$$
With this notation, \eqref{tKdef} can be depicted by
$$
\begin{gathered}
		\begin{tikzpicture}[scale= 1.2, x = 7 mm, y = 7 mm]
			\draw[reflection-arrow] (-1,0) -- (0,1);
            \draw[reflection-arrow] (1,0) -- (0,1);
			\fill[black] (0,1) circle[radius=2.5pt];
			\node[font=\scriptsize, above] at (0,1.2) {$\tK_{\xi,\zeta}$};
			\node[font=\scriptsize, below left] at (-1,0) {$i;a,b$};
            \node[font=\scriptsize, below right] at (1,0) {$j;b^{-1},a^{-1}$};
		\end{tikzpicture}
\end{gathered}
=
\begin{gathered}
		\begin{tikzpicture}[scale = 1.2, x = 7 mm, y = 7 mm]
			\draw[reflection-arrow] (-1,0) -- (1,2);
            \draw[reflection-arrow] (1,0) -- (-1,2);
			\node[font=\scriptsize, above] at (0,1.2) {$\tW^a$};
			\node[font=\scriptsize, below left] at (-1,0) {$i;a,b$};
            \node[font=\scriptsize, below right] at (1,0) {$j;b^{-1},a^{-1}$};
            \node at (1,2) {$\text{---}$};
            \node[above] at (-1,2) {\tiny $\omega_{b^{-1};\xi,\zeta}$};
            \node at (-1,2) {$\text{---}$};
            \node[above] at (1,2) {\tiny $\omega_{a;\zeta,\xi}$};
		\end{tikzpicture}
\end{gathered}
=
\begin{gathered}
		\begin{tikzpicture}[scale = 1.2, x = 7 mm, y = 7 mm]
			\draw[reflection-arrow] (-1,0) -- (1,2);
            \draw[reflection-arrow] (1,0) -- (-1,2);
			\node[font=\scriptsize, above] at (0,1.2) {$\tW^b$};
			\node[font=\scriptsize, below left] at (-1,0) {$i;a,b$};
            \node[font=\scriptsize, below right] at (1,0) {$j;b^{-1},a^{-1}$};
            \node at (1,2) {$\text{---}$};
            \node[above] at (-1,2) {\tiny $\omega_{a;\xi,\zeta}$};
            \node at (-1,2) {$\text{---}$};
            \node[above] at (1,2) {\tiny $\omega_{b^{-1};\zeta,\xi}$};
		\end{tikzpicture}
\end{gathered}
$$

Our main interest in the reflection weights $\tK_{\xi,\zeta}$ comes from the following set of reflection and boundary equations.

\begin{theorem}\label{reflection-theo} (W-boundary relation) Let $\sqrt{a},\sqrt{b},\sqrt{c},\xi,\zeta\in\calF^\times$ be parameters satisfying
$$
\frac ab, \frac cb, ac, a^2, \frac 1{b^{2}}, c^2\notin q^{-\mathbb Z_{\geq0}}.
$$ 
Then the following equation holds for every $i,j,l\in\mathbb Z_{\geq 0}$:
    \begin{multline*}
    \sum_{k,p\geq 0}\widetilde{W}^b_{a,b,c,a^{-1}}(i,j,k,p)\omega_{a;\xi,\zeta}(k)\widetilde{K}_{c,b;\zeta,\xi}(p,l)\\
    =\sum_{k,p\geq 0}\widetilde{K}_{a,b;\xi,\zeta}(i,p)\widetilde{W}^a_{c,a^{-1},b^{-1},c^{-1}}(j,l,p,k)\omega_{c;\xi,\zeta}(k).
    \end{multline*}
Graphically, we write this equation as
$$
\begin{tikzpicture}[scale=1.2, baseline=-4pt]
		\draw[fused] (-1,0) -- (1,0) -- (2,1);
		\draw[fused] (0,-1) -- (0,1);
        \draw[fused] (3,0) -- (2,1);
        \node at (2,1) {$\bullet$};
        \node at (0,1) {$\text{---}$};
	   \node[below] at (-1,-0.1) {\tiny \shortstack[c]{$i; a,b$}};
        \node[above] at (2,1) {\tiny \shortstack[c]{$\widetilde{K}_{\zeta,\xi}$}};
        \node[below] at (3,-0.1) {\tiny \shortstack[c]{$l; b^{-1},c^{-1}$}};
		\node[below] at (0,-1) {\tiny $j; c,a^{-1}$};
        \node[above] at (0,1) {\tiny $\omega_{a;\xi,\zeta}$};
		\node[above right] at (0,0) {\tiny $\widetilde{W}^b$};
\end{tikzpicture} 
\quad
=
\quad
\begin{tikzpicture}[scale=1.2,baseline=-4pt]
		\draw[fused] (1,0) -- (-1,0) -- (-2,1);
		\draw[fused] (0,-1) -- (0,1);
        \draw[fused] (-3,0) -- (-2,1);
        \node at (-2,1) {$\bullet$};
        \node at (0,1) {$\text{---}$};
	   \node[below] at (-3,-0.1) {\tiny \shortstack[c]{$i; a,b$}};
        \node[above] at (-2,1) {\tiny \shortstack[c]{$\widetilde{K}_{\xi,\zeta}$}};
        \node[below] at (1,-0.1) {\tiny \shortstack[c]{$l; b^{-1},c^{-1}$}};
		\node[below] at (0,-1) {\tiny $j; c,a^{-1}$};
        \node[above] at (0,1) {\tiny $\omega_{c;\xi,\zeta}$};
		\node[above right] at (0,0) {\tiny $\widetilde{W}^a$};
\end{tikzpicture} 
$$
\end{theorem}
\begin{proof} Note that both $(1\otimes \tW^a)(\tW^b\otimes 1)$ and $(\tW^b\otimes 1)(1\otimes \tW^a)$ define an isomorphism
$$
V_{\sqrt{b}}^{\sqrt{a}}\otimes V_{\sqrt{a}^{-1}}^{\sqrt{c}}\otimes V_{\sqrt{c}^{-1}}^{\sqrt{b}^{-1}} \to V_{\sqrt{a}^{-1}}^{\sqrt{a}}\otimes V_{\sqrt{b}}^{\sqrt{b}^{-1}}\otimes V_{\sqrt{c}^{-1}}^{\sqrt{c}}.
$$
By Proposition \ref{irred} both tensor products are irreducible. Applying Schur's Lemma and comparing the action on $v_0\otimes v_0\otimes v_0$, we get
$$
(1\otimes \tW^a)(\tW^b\otimes 1)=(\tW^b\otimes 1)(1\otimes \tW^a).
$$
Graphically this identity has the form
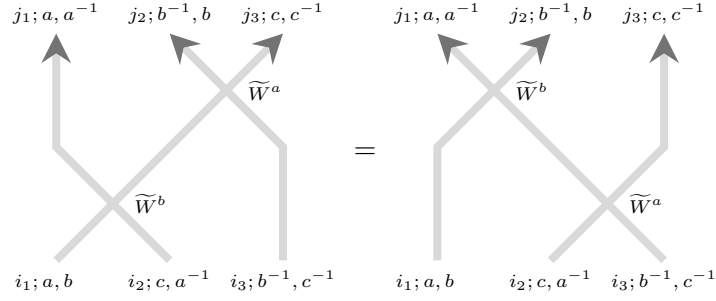
\begin{align*}
	\begin{gathered}
	\begin{tikzpicture}[scale=1.5]
		\draw[fused] (-1,-1) -- (0,0) -- (1,1);
		\draw[fused] (0,-1) -- (-1,0) -- (-1,1);
        \draw[fused] (1,-1) -- (1,0) -- (0,1);
	\node[below] at (-1,-1) {\tiny \shortstack[c]{$i_1; a,b^{\phantom{-1}}$}};
	\node[below] at (0,-1) {\tiny \shortstack[c]{$i_2; c,a^{-1}$}};
	\node[below] at (1,-1) {\tiny \shortstack[c]{$i_3; b^{-1},c^{-1}$}};
	\node[above] at (-1,1) {\tiny $j_1; a,a^{-1}$};
	\node[above] at (0,1) {\tiny $j_2; b^{-1},b$};
	\node[above] at (1,1) {\tiny $j_3; c,c^{-1}$};    
	\node[right] at (-0.4,-0.5) {\tiny $\widetilde{W}^b$};
	\node[right] at (0.6,0.5) {\tiny $\widetilde{W}^a$};
	\end{tikzpicture} 
	\end{gathered}
	=
	\begin{gathered}
	\begin{tikzpicture}[scale=1.5]
		\draw[fused] (-1,-1) -- (-1,0) -- (0,1);
		\draw[fused] (0,-1) -- (1,0) -- (1,1);
        \draw[fused] (1,-1) -- (0,0) -- (-1,1);
	\node[below] at (-1,-1) {\tiny \shortstack[c]{$i_1; a,b^{\phantom{-1}}$}};
	\node[below] at (0,-1) {\tiny \shortstack[c]{$i_2; c,a^{-1}$}};
	\node[below] at (1,-1) {\tiny \shortstack[c]{$i_3; b^{-1},c^{-1}$}};
	\node[above] at (-1,1) {\tiny $j_1; a,a^{-1}$};
	\node[above] at (0,1) {\tiny $j_2; b^{-1},b$};
	\node[above] at (1,1) {\tiny $j_3; c,c^{-1}$};    
	\node[right] at (-0.4,0.5) {\tiny $\widetilde{W}^b$};
	\node[right] at (0.6,-0.5) {\tiny $\widetilde{W}^a$};
	\end{tikzpicture} 
	\end{gathered}
\end{align*}
for every $i_1, i_2, i_3, j_1, j_2,j_3\in\mathbb Z_{\geq 0}$. In particular, we get
\begin{multline*}
\langle\omega_{a;\xi,\zeta}|\otimes \langle\omega_{b^{-1};\zeta,\xi}|\otimes \langle\omega_{c;\xi,\zeta}|\ (1\otimes \tW^a)(\tW^b\otimes 1)\\
=\langle\omega_{a;\xi,\zeta}|\otimes \langle\omega_{b^{-1};\zeta,\xi}|\otimes \langle\omega_{c;\xi,\zeta}|\ (\tW^b\otimes 1)(1\otimes \tW^a).
\end{multline*}
Equivalently, for every $i,j,l\in\mathbb Z_{\geq 0}$, we have
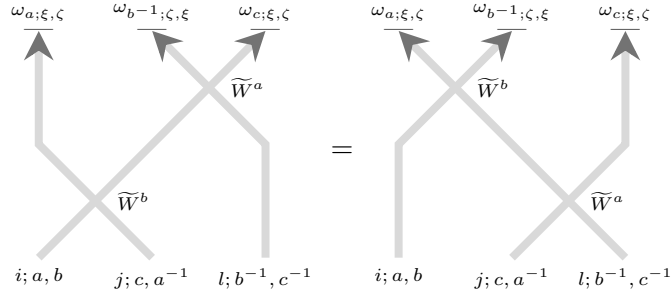
\begin{align*}
	\begin{gathered}
	\begin{tikzpicture}[scale=1.5]
		\draw[fused] (-1,-1) -- (0,0) -- (1,1);
		\draw[fused] (0,-1) -- (-1,0) -- (-1,1);
        \draw[fused] (1,-1) -- (1,0) -- (0,1);
	\node[below] at (-1,-1) {\tiny \shortstack[c]{$i; a,b$}};
	\node[below] at (0,-1) {\tiny \shortstack[c]{$j; c,a^{-1}$}};
	\node[below] at (1,-1) {\tiny \shortstack[c]{$l; b^{-1},c^{-1}$}};
	\node[above] at (-1,1) {\tiny $\omega_{a;\xi,\zeta}$};
    \node[] at (-1,1) {\text{---}};
	\node[above] at (0,1) {\tiny $\omega_{b^{-1};\zeta,\xi}$};
    \node[] at (0,1) {\text{---}};
	\node[above] at (1,1) {\tiny $\omega_{c;\xi,\zeta}$};    
    \node[] at (1,1) {\text{---}};
	\node[right] at (-0.4,-0.5) {\tiny $\widetilde{W}^b$};
	\node[right] at (0.6,0.5) {\tiny $\widetilde{W}^a$};
	\end{tikzpicture} 
	\end{gathered}
	=
	\begin{gathered}
	\begin{tikzpicture}[scale=1.5]
		\draw[fused] (-1,-1) -- (-1,0) -- (0,1);
		\draw[fused] (0,-1) -- (1,0) -- (1,1);
        \draw[fused] (1,-1) -- (0,0) -- (-1,1);
	\node[below] at (-1,-1) {\tiny \shortstack[c]{$i; a,b$}};
	\node[below] at (0,-1) {\tiny \shortstack[c]{$j; c,a^{-1}$}};
	\node[below] at (1,-1) {\tiny \shortstack[c]{$l; b^{-1},c^{-1}$}};
	\node[above] at (-1,1) {\tiny $\omega_{a;\xi,\zeta}$};
    \node[] at (-1,1) {\text{---}};
	\node[above] at (0,1) {\tiny $\omega_{b^{-1};\zeta,\xi}$};
    \node[] at (0,1) {\text{---}};
	\node[above] at (1,1) {\tiny $\omega_{c;\xi,\zeta}$};    
    \node[] at (1,1) {\text{---}};    
	\node[right] at (-0.4,0.5) {\tiny $\widetilde{W}^b$};
	\node[right] at (0.6,-0.5) {\tiny $\widetilde{W}^a$};
	\end{tikzpicture} 
	\end{gathered}
\end{align*}
The reflection equation follows by using \eqref{tKdef} to replace the upper vertices with $\widetilde{K}$.
\end{proof}

\begin{theorem}\label{reflection2-theo} (R-boundary relation) Let $\sqrt{a},\sqrt{b},\sqrt{c},\xi,\zeta\in\calF^\times$ be parameters satisfying
$$
\frac ab, \frac cb, ac, a^2, \frac 1{b^{2}}, c^2\notin q^{-\mathbb Z_{\geq0}}.
$$ 
The following equation holds for every $i,j,l\in\mathbb Z_{\geq 0}$:
    \begin{multline*}
    \sum_{k,p\geq 0}\tR_{a,b,c,c^{-1}}(i,j,k,p)\omega_{c;\xi,\zeta}(k)\widetilde{K}_{a,b;\zeta,\xi}(p,l)\\
    =\sum_{k,p\geq 0}\widetilde{K}_{a,b;\xi,\zeta}(i,p)\tR_{c,c^{-1},b^{-1},a^{-1}}(j,l,p,k)\omega_{c;\xi,\zeta}(k).
    \end{multline*}
Graphically, we write this equation as
$$
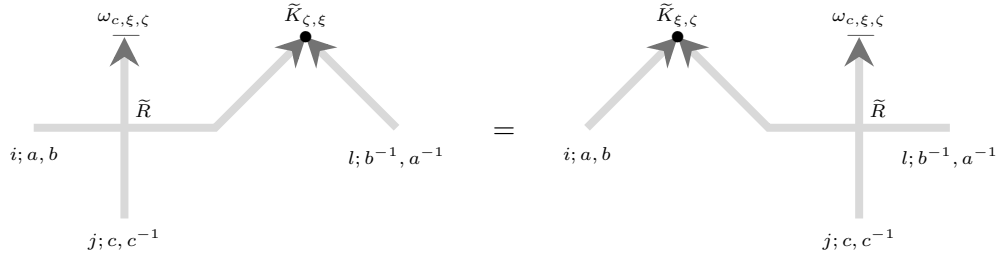
\begin{tikzpicture}[scale=1.2, baseline=-4pt]
		\draw[fused] (-1,0) -- (1,0) -- (2,1);
		\draw[fused] (0,-1) -- (0,1);
        \draw[fused] (3,0) -- (2,1);
        \node at (2,1) {$\bullet$};
        \node at (0,1) {$\text{---}$};
	   \node[below] at (-1,-0.1) {\tiny \shortstack[c]{$i; a,b$}};
        \node[above] at (2,1) {\tiny \shortstack[c]{$\widetilde{K}_{\zeta,\xi}$}};
        \node[below] at (3,-0.1) {\tiny \shortstack[c]{$l; b^{-1},a^{-1}$}};
		\node[below] at (0,-1) {\tiny $j; c,c^{-1}$};
        \node[above] at (0,1) {\tiny $\omega_{c,\xi,\zeta}$};
		\node[above right] at (0,0) {\tiny $\tR$};
\end{tikzpicture} 
\quad
=
\quad
\begin{tikzpicture}[scale=1.2,baseline=-4pt]
		\draw[fused] (1,0) -- (-1,0) -- (-2,1);
		\draw[fused] (0,-1) -- (0,1);
        \draw[fused] (-3,0) -- (-2,1);
        \node at (-2,1) {$\bullet$};
        \node at (0,1) {$\text{---}$};
	   \node[below] at (-3,-0.1) {\tiny \shortstack[c]{$i; a,b$}};
        \node[above] at (-2,1) {\tiny \shortstack[c]{$\widetilde{K}_{\xi,\zeta}$}};
        \node[below] at (1,-0.1) {\tiny \shortstack[c]{$l; b^{-1},a^{-1}$}};
		\node[below] at (0,-1) {\tiny $j; c,c^{-1}$};
        \node[above] at (0,1) {\tiny $\omega_{c,\xi,\zeta}$};
		\node[above right] at (0,0) {\tiny $\tR$};
\end{tikzpicture} 
$$
\end{theorem}
\begin{proof} The claim is equivalent to
\begin{multline*}
\langle \omega_{c;\xi,\zeta}|\otimes\langle \omega_{b^{-1};\zeta,\xi}|\otimes\langle \omega_{a;\xi,\zeta}|\ (1\otimes \tW^a)(\tR\otimes 1)\\
=\langle \omega_{a;\xi,\zeta}|\otimes\langle \omega_{b^{-1};\zeta,\xi}|\otimes\langle \omega_{c;\xi,\zeta}|\ (\tW^b\otimes 1)(1\otimes\tR),
\end{multline*}
where both sides are elements of
$
\left(V^{\sqrt{a}}_{\sqrt{b}}\otimes V^{\sqrt{c}}_{\sqrt{c}^{-1}}\otimes V^{\sqrt{b}^{-1}}_{\sqrt{a}^{-1}}\right)^*.
$
Proposition \ref{dualeigenR} implies
\begin{multline*}
\langle \omega_{c;\xi,\zeta}|\otimes\langle \omega_{b^{-1};\zeta,\xi}|\otimes\langle \omega_{a;\xi,\zeta}|\  (1\otimes \tW^a)(\tR\otimes 1)\\
=\langle \omega_{a;\xi,\zeta}|\otimes\langle \omega_{b^{-1};\zeta,\xi}|\otimes\langle \omega_{c;\xi,\zeta}|\ (1\otimes\tR)(\tR\otimes 1)(1\otimes\tR)  (1\otimes \tW^a)(\tR\otimes 1).
\end{multline*}
Both $(1\otimes\tR)(\tR\otimes 1)(1\otimes\tR)  (1\otimes \tW^a)(\tR\otimes 1)$ and $(\tW^b\otimes 1)(1\otimes\tR)$ are isomorphisms of irreducible modules
$$
V^{\sqrt{a}}_{\sqrt{b}}\otimes V^{\sqrt{c}}_{\sqrt{c}^{-1}}\otimes V^{\sqrt{b}^{-1}}_{\sqrt{a}^{-1}}\to V^{\sqrt{a}}_{\sqrt{a}^{-1}}\otimes V^{\sqrt{b}^{-1}}_{\sqrt{b}}\otimes V^{\sqrt{c}}_{\sqrt{c}^{-1}}
$$
with the same action on $v_0^{\otimes3}$. Hence
$$
(1\otimes\tR)(\tR\otimes 1)(1\otimes\tR)  (1\otimes \tW^a)(\tR\otimes 1)=(\tW^b\otimes 1)(1\otimes\tR).
$$
\end{proof}

\begin{theorem} \label{thm: reflection3}(Reflection equation)
Let $\sqrt{a},\sqrt{b},\sqrt{c},\sqrt{d},\xi,\zeta\in\calF^\times$ be parameters satisfying
$$
\frac ab, \frac ad, \frac cb, \frac cd, ac, \frac 1{bd}, a^2, \frac 1{b^{2}}, c^2, \frac1{d^{2}}\notin q^{-\mathbb Z_{\geq0}}.
$$
The following equation holds for every $i,j,k,l \in\mathbb Z_{\geq 0}$:
    \begin{multline*}
    \sum_{p,r,s,t\geq 0}\tR_{a,b,c,d}(i,j,p,r)\widetilde{K}_{c,d;\xi,\zeta}(p,s)\tR_{a,b,d^{-1},c^{-1}}(r,k,s,t)\tK_{a,b;\xi,\zeta}(t,l)\\
    =\sum_{p,r,s,t\geq 0}\widetilde{K}_{a,b;\xi,\zeta}(i,s)\tR_{c,d,b^{-1},a^{-1}}(j,p,s,t)\tK_{c,d;\xi,\zeta}(t,r)\tR_{d^{-1},c^{-1},b^{-1},a^{-1}}(k,l,p,r).
    \end{multline*}
Graphically, we write this equation as
$$
\begin{tikzpicture}[scale=1.2, baseline=-4pt]
		\draw[fused] (-3,0) -- (-1,0) -- (1,2);
		\draw[fused] (1,0) -- (-1,2);
		\draw[fused] (-2,-1) -- (-2,1) -- (-1,2);
        \draw[fused] (2,1) -- (1,2);
        \node at (1,2) {$\bullet$};
        \node at (-1,2) {$\bullet$};
	   \node[below] at (-3,-0.1) {\tiny \shortstack[c]{$i; a,b$}};
       	\node[below] at (-2,-1) {\tiny $j; c,d$};
        \node[below] at (1,-0.1) {\tiny \shortstack[c]{$k; d^{-1},c^{-1}$}};
        \node[below] at (2,0.9) {\tiny \shortstack[c]{$l; b^{-1},a^{-1}$}};
        \node[above] at (1,2) {\tiny \shortstack[c]{$\widetilde{K}_{\xi,\zeta}$}};
        \node[above] at (-1,2) {\tiny \shortstack[c]{$\widetilde{K}_{\xi,\zeta}$}};
		\node[above right] at (-2,0) {\tiny $\tR$};
        \node[above] at (0,1.1) {\tiny $\tR$};
\end{tikzpicture} 
=
\begin{tikzpicture}[scale=1.2, baseline=-4pt]
		\draw[fused] (-1,0) -- (1,2);
		\draw[fused] (3,0) -- (1,0) -- (-1,2);
		\draw[fused] (-2,1) -- (-1,2);
        \draw[fused] (2,-1) -- (2,1) -- (1,2);
        \node at (1,2) {$\bullet$};
        \node at (-1,2) {$\bullet$};
	   \node[below] at (-2,1) {\tiny \shortstack[c]{$i; a,b$}};
       	\node[below] at (-1,-0.1) {\tiny $j; c,d$};
        \node[below] at (2,-1) {\tiny \shortstack[c]{$k; d^{-1},c^{-1}$}};
        \node[below] at (3,0) {\tiny \shortstack[c]{$l; b^{-1},a^{-1}$}};
        \node[above] at (1,2) {\tiny \shortstack[c]{$\widetilde{K}_{\xi,\zeta}$}};
        \node[above] at (-1,2) {\tiny \shortstack[c]{$\widetilde{K}_{\xi,\zeta}$}};
		\node[above right] at (2,0) {\tiny $\tR$};
        \node[above] at (0,1.1) {\tiny $\tR$};
\end{tikzpicture} 
$$
\end{theorem}
\begin{proof} Applying \eqref{tKdef}, the claim is equivalent to
\begin{multline}\label{reflectiongoal}
\langle \omega_{c;\xi,\zeta}|\otimes\langle \omega_{d^{-1};\zeta,\xi}|\otimes\langle \omega_{a;\xi,\zeta}|\otimes\langle \omega_{b^{-1};\zeta,\xi}|\ (\tW^b\otimes \tW^b)(1\otimes\tR\otimes 1)(\tR\otimes 1\otimes 1)\\
=\langle \omega_{a;\xi,\zeta}|\otimes\langle \omega_{b^{-1};\zeta,\xi}|\otimes\langle \omega_{c;\xi,\zeta}|\otimes\langle \omega_{d^{-1};\zeta,\xi}|\ (\tW^b\otimes \tW^b)(1\otimes\tR\otimes 1)(1\otimes 1\otimes \tR),
\end{multline}
where both sides are elements of
$$
\left(V^{\sqrt{a}}_{\sqrt{b}}\otimes V^{\sqrt{c}}_{\sqrt{d}}\otimes V^{\sqrt{d}^{-1}}_{\sqrt{c}^{-1}}\otimes V^{\sqrt{b}^{-1}}_{\sqrt{a}^{-1}}\right)^*.
$$
To simplify the expressions, let
$$
F_1=(\tW^b\otimes \tW^b)(1\otimes\tR\otimes 1)(\tR\otimes 1\otimes 1),\qquad F_2=(\tW^b\otimes \tW^b)(1\otimes\tR\otimes 1)(1\otimes 1\otimes\tR),
$$
$$
G=(1\otimes \tR\otimes 1)(\tR\otimes\tR)(1\otimes\tR\otimes1).
$$
Note that both $G\circ F_1$ and $F_2$ are isomorphisms 
$$
V^{\sqrt{a}}_{\sqrt{b}}\otimes V^{\sqrt{c}}_{\sqrt{d}}\otimes V^{\sqrt{d}^{-1}}_{\sqrt{c}^{-1}}\otimes V^{\sqrt{b}^{-1}}_{\sqrt{a}^{-1}}\to V^{\sqrt{a}}_{\sqrt{a}^{-1}}\otimes V^{\sqrt{b}^{-1}}_{\sqrt{b}}\otimes V^{\sqrt{c}}_{\sqrt{c}^{-1}}\otimes V^{\sqrt{d}^{-1}}_{\sqrt{d}},
$$
where both sides are irreducible by Proposition \ref{irred}. Comparing the action on $v_0^{\otimes 4}$ and using the Schur's lemma, we get that  $G\circ F_1=F_2$. So
\begin{multline*}
\langle \omega_{a;\xi,\zeta}|\otimes\langle \omega_{b^{-1};\zeta,\xi}|\otimes\langle \omega_{c;\xi,\zeta}|\otimes\langle \omega_{d^{-1};\zeta,\xi}|\ F_2\\
=\langle \omega_{a;\xi,\zeta}|\otimes\langle \omega_{b^{-1};\zeta,\xi}|\otimes\langle \omega_{c;\xi,\zeta}|\otimes\langle \omega_{d^{-1};\zeta,\xi}|\ G\circ F_1.
\end{multline*}
From Proposition \ref{dualeigenR}, we have
\begin{multline*}
\langle \omega_{a;\xi,\zeta}|\otimes\langle \omega_{b^{-1};\zeta,\xi}|\otimes\langle \omega_{c;\xi,\zeta}|\otimes\langle \omega_{d^{-1};\zeta,\xi}|\ G\circ F_1\\
=\langle \omega_{a;\xi,\zeta}|\otimes\langle \omega_{b^{-1};\zeta,\xi}|\otimes\langle \omega_{c;\xi,\zeta}|\otimes\langle \omega_{d^{-1};\zeta,\xi}|\ (1\otimes \tR\otimes 1)(\tR\otimes\tR)(1\otimes\tR\otimes1)F_1\\
=\langle \omega_{c;\xi,\zeta}|\otimes\langle \omega_{d^{-1};\zeta,\xi}|\otimes\langle \omega_{a;\xi,\zeta}|\otimes\langle \omega_{b^{-1};\zeta,\xi}|\, F_1,
\end{multline*}
which implies \eqref{reflectiongoal}.
\end{proof}

\section{Higher-spin six vertex model and spin Hall-Littlewood functions}

In this section, we provide the notation and vertex models related to inhomogeneous stable spin Hall-Littlewood functions.

\subsection{Higher-spin six vertex model} Here we consider vertex models where some edges have parameters $(z, zq)$. We call these edges \emph{unfused}, we parametrize them using a single parameter $z$ and they are visualized by thin lines in our diagrams. These edges correspond to the two-dimensional modules $V(q^{-\frac12})_{zq^{1/2}}$, so such edges are only allowed to have values $0$ or $1$. Below we describe the corresponding degeneration of the vertex weights.
\begin{proposition}\label{prop: hsw} Let $a,b,z$ be generic parameters. 

(1) The vertex weights 
$$
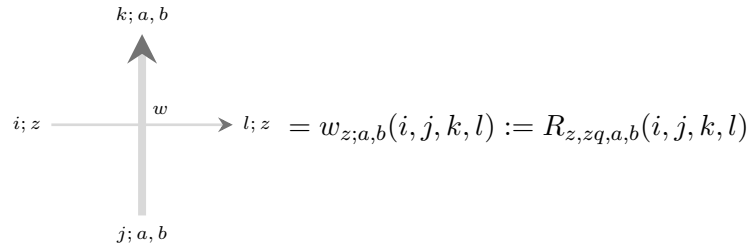
\begin{tikzpicture}[scale=1.2, baseline=-4pt]
	\draw[unfused] (-1,0) -- (1,0);
	\draw[fused] (0,-1) -- (0,1);
	\node[left] at (-1,0) {\tiny $i;z$};\node[right] at (1,0) {\tiny $l;z$};
	\node[below] at (0,-1) {\tiny $j;a,b$};\node[above] at (0,1) {\tiny $k;a,b$};
	\node[above right] at (0,0) {\tiny $w$};
\end{tikzpicture}=w_{z;a,b}(i,j,k,l):=R_{z, zq,a,b}(i,j,k,l)
$$
are well-defined for all $i,j,k,l\in\mathbb Z_{\geq0}$. When $i\in\{0,1\}$ they are described by the following table:
$$
w_{z,a,b}(i,j,k,l)=0\qquad \text{unless}\ i+j=k+l\ \text{and}\ l\in\{0,1\},
$$

\begin{align*}
\begin{array}{cccc}
\begin{tikzpicture}[scale=0.9, baseline=-4pt]
	\draw[unfused] (-1,0) -- (1,0);
	\draw[fused] (0,-1) -- (0,1);
	\node[left] at (-1,0) {\tiny $0;z$};\node[right] at (1,0) {\tiny $0;z$};
	\node[below] at (0,-1) {\tiny $g;a,b$};\node[above] at (0,1) {\tiny $g;a,b$};
	\node[above right] at (0,0) {\tiny $w$};
\end{tikzpicture}
&
\begin{tikzpicture}[scale=0.9, baseline=-4pt]
	\draw[unfused] (-1,0) -- (1,0);
	\draw[fused] (0,-1) -- (0,1);
	\node[left] at (-1,0) {\tiny $0;z$};\node[right] at (1,0) {\tiny $1;z$};
	\node[below] at (0,-1) {\tiny $g;a,b$};\node[above] at (0,1) {\tiny $g-1;a,b$};
	\node[above right] at (0,0) {\tiny $w$};
\end{tikzpicture}
&
\begin{tikzpicture}[scale=0.9, baseline=-4pt]
	\draw[unfused] (-1,0) -- (1,0);
	\draw[fused] (0,-1) -- (0,1);
	\node[left] at (-1,0) {\tiny $1;z$};\node[right] at (1,0) {\tiny $0;z$};
	\node[below] at (0,-1) {\tiny $g;a,b$};\node[above] at (0,1) {\tiny $g+1;a,b$};
	\node[above right] at (0,0) {\tiny $w$};
\end{tikzpicture}
&
\begin{tikzpicture}[scale=0.9, baseline=-4pt]
	\draw[unfused] (-1,0) -- (1,0);
	\draw[fused] (0,-1) -- (0,1);
	\node[left] at (-1,0) {\tiny $1;z$};\node[right] at (1,0) {\tiny $1;z$};
	\node[below] at (0,-1) {\tiny $g;a,b$};\node[above] at (0,1) {\tiny $g;a,b$};
	\node[above right] at (0,0) {\tiny $w$};
\end{tikzpicture}
\\
\\
\dfrac{1-q^gz/b}{1-z/b}
\qquad
&
\dfrac{(q^g-1)z/b}{1-z/b}
\qquad
&
\dfrac{1-q^ga/b}{1-z/b}
\qquad
&
\dfrac{q^ga/b-z/b}{1-z/b}
\vspace{1ex}
\end{array}
\end{align*}

(2) The vertex weights 
$$
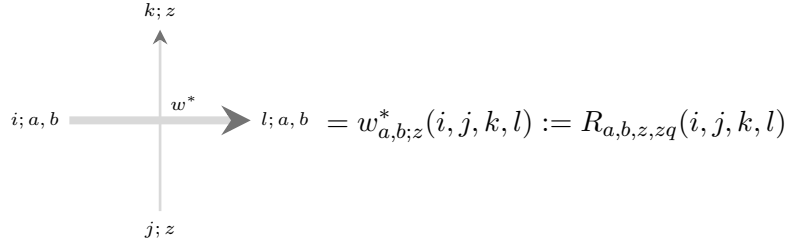
\begin{tikzpicture}[scale=1.2, baseline=-4pt]
	\draw[fused] (-1,0) -- (1,0);
	\draw[unfused] (0,-1) -- (0,1);
	\node[left] at (-1,0) {\tiny $i;a,b$};\node[right] at (1,0) {\tiny $l;a,b$};
	\node[below] at (0,-1) {\tiny $j;z$};\node[above] at (0,1) {\tiny $k;z$};
	\node[above right] at (0,0) {\tiny $w^*$};
\end{tikzpicture}=w^*_{a,b;z}(i,j,k,l):=R_{a,b, z, zq}(i,j,k,l)
$$
are well-defined for all $i,j,k,l\in\mathbb Z_{\geq0}$. When $j\in\{0,1\}$ they are described by the following table:
$$
w^*_{a,b;z}(i,j,k,l)=0\qquad \text{unless}\ i+j=k+l\ \text{and}\ k\in\{0,1\},
$$

\begin{align*}
\begin{array}{cccc}
\begin{tikzpicture}[scale=0.9, baseline=-4pt]
	\draw[fused] (-1,0) -- (1,0);
	\draw[unfused] (0,-1) -- (0,1);
	\node[below] at (-1,-0.1) {\tiny $g;a,b$};\node[below] at (1,-0.1) {\tiny $g;a,b$};
	\node[below] at (0,-1) {\tiny $0;z$};\node[above] at (0,1) {\tiny $0;z$};
	\node[above right] at (0,0) {\tiny $w^*$};
\end{tikzpicture}
\ \
&
\begin{tikzpicture}[scale=0.9, baseline=-4pt]
	\draw[fused] (-1,0) -- (1,0);
	\draw[unfused] (0,-1) -- (0,1);
	\node[below] at (-1,-0.1) {\tiny $g;a,b$};\node[below] at (1,-0.1) {\tiny $g-1;a,b$};
	\node[below] at (0,-1) {\tiny $0;z$};\node[above] at (0,1) {\tiny $1;z$};
	\node[above right] at (0,0) {\tiny $w^*$};
\end{tikzpicture}
\ \
&
\begin{tikzpicture}[scale=0.9, baseline=-4pt]
	\draw[fused] (-1,0) -- (1,0);
	\draw[unfused] (0,-1) -- (0,1);
	\node[below] at (-1,-0.1) {\tiny $g;a,b$};\node[below] at (1,-0.1) {\tiny $g+1;a,b$};
	\node[below] at (0,-1) {\tiny $1;z$};\node[above] at (0,1) {\tiny $0;z$};
	\node[above right] at (0,0) {\tiny $w^*$};
\end{tikzpicture}
\ \
&
\begin{tikzpicture}[scale=0.9, baseline=-4pt]
	\draw[fused] (-1,0) -- (1,0);
	\draw[unfused] (0,-1) -- (0,1);
	\node[below] at (-1,-0.1) {\tiny $g;a,b$};\node[below] at (1,-0.1) {\tiny $g;a,b$};
	\node[below] at (0,-1) {\tiny $1;z$};\node[above] at (0,1) {\tiny $1;z$};
	\node[above right] at (0,0) {\tiny $w^*$};
\end{tikzpicture}
\\
\\
\dfrac{q^{-g}-\frac{a}{zq}}{1-\frac{a}{zq}}
\qquad
&
\dfrac{1-q^{-g}}{1-\frac{a}{zq}}
\qquad
&
\dfrac{q^{-g}\frac{b}{qz}-\frac{a}{zq}}{1-\frac{a}{zq}}
\qquad
&
\dfrac{1-q^{-g}\frac{b}{zq}}{1-\frac{a}{zq}}
\vspace{1ex}
\end{array}
\end{align*}
\end{proposition}
\begin{proof} Follows directly from \eqref{eq: stochR}. More precisely,
$$
R_{z,zq,a,b}(i,j,k,l)=\delta_{i+j=k+l}\sum_{p\geq 0} \Phi(l-p, k, z/a, a/b)\Phi(p, j-p, a/(zq), zq/b)
$$
implies that $w_{z;a,b}$ is always well-defined, while 
$$
R_{z,zq,a,b}(i,j,k,l)=\delta_{i+j=k+l}\sum_{p=0}^{\max(i,k)} \Phi(i-p, p, z/a, a/(zq))\Phi(l, k-p, q^{-1}, zq/b)
$$
can be easily computed when $i\in\{0,1\}$ since there are at most two terms. The argument for $w^*$ is identical.
\end{proof}

The weights $w$ and $w^*$ are called \emph{stochastic higher spin six-vertex weights} \cite{Bor14}, \cite{BP18}. We set 
$$
\tw_{z;a,b}(i,j,k,l):=\tR_{z, zq,a,b}(i,j,k,l), \qquad \tw^*_{a,b;z}(i,j,k,l):=\tR_{a,b, z, zq}(i,j,k,l).
$$
Note that, as long as the unfused edges are adjacent to $w$ or $w^*$ vertices, restricting the values of unfused boundary edges to $\{0,1\}$ forces all internal unfused edges to have values $0$ or $1$ as well. This ensures that unfused edges are consistent with our notation for partition functions, and we can get identities for the weights $w$, $w^*$ by setting some of the edge parameters $(a,b)$ to $(z,zq)$. For instance, we have the following reflection equation for the higher spin six-vertex weights.

\begin{proposition}\label{prop: hsref} For generic parameters $\sqrt{z},\sqrt{c},\zeta,\xi$ and $i,l\in\{0,1\}$, $j\in\mathbb Z_{\geq0}$ we have
    $$
\begin{tikzpicture}[scale=1.2, baseline=-4pt]
		\draw[unfused] (-1,0) -- (1,0) -- (2,1);
		\draw[fused] (0,-1) -- (0,1);
        \draw[unfused] (3,0) -- (2,1);
        \node at (2,1) {$\bullet$};
        \node at (0,1) {$\text{---}$};
	   \node[below] at (-1,-0.1) {\tiny \shortstack[c]{$i; z$}};
        \node[above] at (2,1) {\tiny \shortstack[c]{$\widetilde{K}_{\zeta,\xi}$}};
        \node[below] at (3,-0.1) {\tiny \shortstack[c]{$l; (qz)^{-1}$}};
		\node[below] at (0,-1) {\tiny $j; c,c^{-1}$};
        \node[above] at (0,1) {\tiny $\omega_{c;\xi,\zeta}$};
		\node[above right] at (0,0) {\tiny $\tw$};
\end{tikzpicture} 
\quad
=
\quad
\begin{tikzpicture}[scale=1.2,baseline=-4pt]
		\draw[unfused] (1,0) -- (-1,0) -- (-2,1);
		\draw[fused] (0,-1) -- (0,1);
        \draw[unfused] (-3,0) -- (-2,1);
        \node at (-2,1) {$\bullet$};
        \node at (0,1) {$\text{---}$};
	   \node[below] at (-3,-0.1) {\tiny \shortstack[c]{$i; z$}};
        \node[above] at (-2,1) {\tiny \shortstack[c]{$\widetilde{K}_{\xi,\zeta}$}};
        \node[below] at (1,-0.1) {\tiny \shortstack[c]{$l; (qz)^{-1}$}};
		\node[below] at (0,-1) {\tiny $j; c,c^{-1}$};
        \node[above] at (0,1) {\tiny $\omega_{c;\xi,\zeta}$};
		\node[above right] at (0,0) {\tiny $\tw^*$};
\end{tikzpicture} 
$$
The reflection weights $\tK_{z,qz;\xi,\zeta}(i,j)$ are given by

\begin{align*}
\begin{array}{cccc}
\begin{tikzpicture}[scale=0.9, baseline=-4pt]
	\draw[unfused] (-1,0) -- (0,1);
	\draw[unfused] (1,0) -- (0,1);
	\node[below] at (-1,0) {\tiny $0;z$};\node[below] at (1,0) {\tiny $0;z^{-1}q^{-1}$};
	\node[above] at (0,1) {\tiny $\tK_{\xi,\zeta}$};
    \node at (0,1) {$\bullet$};
\end{tikzpicture}
\ \
&
\begin{tikzpicture}[scale=0.9, baseline=-4pt]
	\draw[unfused] (-1,0) -- (0,1);
	\draw[unfused] (1,0) -- (0,1);
	\node[below] at (-1,0) {\tiny $0;z$};\node[below] at (1,0) {\tiny $1;z^{-1}q^{-1}$};
	\node[above] at (0,1) {\tiny $\tK_{\xi,\zeta}$};
    \node at (0,1) {$\bullet$};
\end{tikzpicture}
\ \
&
\begin{tikzpicture}[scale=0.9, baseline=-4pt]
	\draw[unfused] (-1,0) -- (0,1);
	\draw[unfused] (1,0) -- (0,1);
	\node[below] at (-1,0) {\tiny $1;z$};\node[below] at (1,0) {\tiny $0;z^{-1}q^{-1}$};
	\node[above] at (0,1) {\tiny $\tK_{\xi,\zeta}$};
    \node at (0,1) {$\bullet$};
\end{tikzpicture}
\ \
&
\begin{tikzpicture}[scale=0.9, baseline=-4pt]
	\draw[unfused] (-1,0) -- (0,1);
	\draw[unfused] (1,0) -- (0,1);
	\node[below] at (-1,0) {\tiny $1;z$};\node[below] at (1,0) {\tiny $1;z^{-1}q^{-1}$};
	\node[above] at (0,1) {\tiny $\tK_{\xi,\zeta}$};
    \node at (0,1) {$\bullet$};
\end{tikzpicture}
\\
\\
1
\qquad
&
\frac{zq^{\frac12}(\xi+\xi^{-1}-z(\zeta+\zeta^{-1}))}{1-z^2}
\qquad
&
\frac{\zeta+\zeta^{-1}-z(\xi+\xi^{-1})}{1-z^2}
\qquad
&
q^{\frac12}
\vspace{1ex}
\end{array}
\end{align*}
\end{proposition}
\begin{proof} Note that Theorem \ref{reflection2-theo} can be viewed as an identity of rational functions in $\sqrt{a}$, $\sqrt{b}$, $\sqrt{c}$, $\zeta$, $\xi$, so we can set $a=z$ and $b=qz$ as long as all the weights are well-defined. Proposition \ref{prop: hsw} ensures that the weights $\tR$ on both sides have no poles at $a/b=q^{-1}$. For the reflection weights we use \eqref{tKdef} to get
\begin{multline*}
\tK_{a,b;\xi,\zeta}(i,j)=\sum_{k= 0}^{i}\omega_{b^{-1};\xi,\zeta}(k)\omega_{a;\zeta,\xi}(i+j-k)\frac{b^{j+k-i}}{(b/a)^{j/2}}\frac{(ab;q)_{i-k}(b^{-2};q)_k(q;q)_i}{(a/b;q)_{i}(q;q)_{i-k}(q;q)_k}
\\=\sum_{l= 0}^j\omega_{a;\xi,\zeta}(i+j-l)\omega_{b^{-1};\zeta,\xi}(l)\frac{b^l}{(b/a)^{j/2}}\frac{(b^{-2};q)_l(ab;q)_{j-l}(q;q)_j}{(a/b;q)_j(q;q)_l(q;q)_{j-l}}.
\end{multline*}
The first line has no pole at $b/a=q$ if $i\in\{0,1\}$, ensuring that the right-hand side of Theorem \ref{reflection2-theo} is well-defined. The second line has no pole if $j\in\{0,1\}$, which handles the other side. The values for the reflection weights follow by a direct computation. 
\end{proof}

\begin{remark}\label{rem: degen} Setting $(a,b)=(z,qz)$ and $(c,d)=(u,qu)$ in Theorem \ref{thm: reflection3} turns all $\tR$ weights into \emph{six-vertex} weights. By swapping zeroes and ones on one side of the reflection weights $\tK_{z,qz;\xi,\zeta}(i,j)$, using cross symmetry of six-vertex weights and applying a gauge transformation, Theorem \ref{thm: reflection3} reduces to the solution of the reflection equation found in \cite{dVGR93}. In particular, 
    $$
    \left(\frac{z^{\frac{i+j-1}{2}}q^{\frac{j-1}2}(-1)^{j}(z-z^{-1})\tK_{z,qz;\xi,\zeta}(i,1-j)}{2c}\right)_{i,j\in\{0,1\}}=\begin{pmatrix}
    \sinh(\varepsilon+\theta)&-c^{-1}\sinh(\theta)\\
    c^{-1}\sinh(\theta)&\sinh(\varepsilon-\theta)\\
    \end{pmatrix}
    $$
    where
    $$
    c=(\xi+\xi^{-1})^{\frac12}(\zeta+\zeta^{-1})^{\frac12},\qquad e^{\theta}=\sqrt{z}, \qquad e^{\varepsilon}=\left(\frac{\zeta+\zeta^{-1}}{\xi+\xi^{-1}}\right)^{\frac12}.
    $$
    These reflection weights have also been used in the half-space six-vertex model from \cite{He24}, \cite{GdGMW}.
\end{remark}

\subsection{Inhomogeneous stable spin Hall-Littlewood functions}\label{sec: sHL-def}  Let $\calA=(a_0,a_1,a_2, \dots )$, $\calB=(b_0,b_1,b_2,\dots)$ be two sequences of generic complex parameters, and $\lambda=(\lambda_1\geq\lambda_2\geq\dots\geq 0)$, $\mu=(\mu_1\geq\mu_2\geq\dots\geq 0)$ be a pair of partitions. For each $N>\lambda_1,\mu_1$ and $(r_1,\dots,r_n)\in\{0,1\}^n$ we consider the partition function $Zw^{r_1,\dots,r_n}_{\lambda/\mu,N}(x_1,\dots, x_n\mid \calA,\calB)$ consisting of the higher spin six-vertex weights $w_{z;a,b}$:

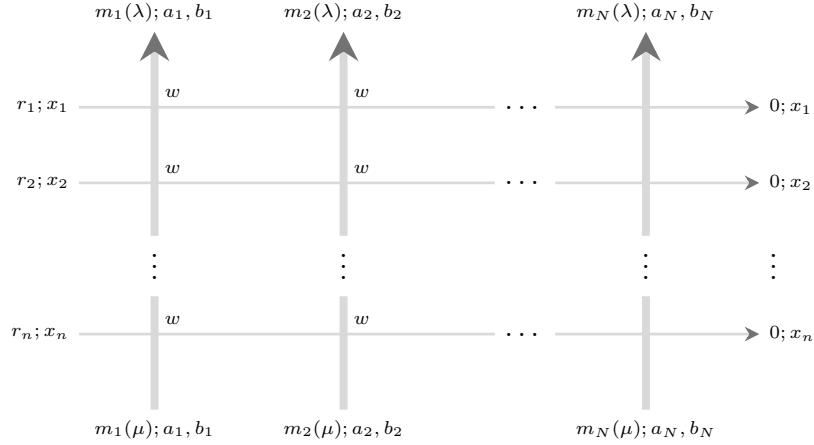
	\begin{center}
		\begin{tikzpicture}
			\foreach\y in {4,...,5}{
				\draw[unfused*] (1,\y) -- (6.5,\y);
				\node[right] at (6.5,\y) {$\dots$};
				\draw[unfused] (7.3,\y) -- (10,\y);
			}
			\draw[unfused*] (1,2) -- (6.5,2);
			\node[right] at (6.5,2) {$\dots$};
			\draw[unfused] (7.3,2) -- (10,2);
			\foreach\x in {1, 2}{
				\draw[fused] (2.5*\x-0.5,3.3) -- (2.5*\x-0.5,6);
				\draw[fused*] (2.5*\x-0.5,1) -- (2.5*\x-0.5,2.5);
				\node at (2.5*\x-0.5, 3) {$\vdots$};
			}
            \draw[fused] (8.5,3.3) -- (8.5,6);
			\draw[fused*] (8.5,1) -- (8.5,2.5);
			\node at (8.5, 3) {$\vdots$};
			\node[above right] at (2,5) {\tiny $w$};
			\node[above right] at (4.5,5) {\tiny $w$};
			\node[above right] at (2,4) {\tiny $w$};
			\node[above right] at (4.5,4) {\tiny $w$};
			\node[above right] at (2,2) {\tiny $w$};
			\node[above right] at (4.5,2) {\tiny $w$};
			\node[above] at (8.5,6) {\tiny $m_N(\lambda); a_N, b_N$};
			\node[above] at (4.5,6) {\tiny $m_2(\lambda); a_2, b_2$};
			\node[above] at (2,6) {\tiny $m_1(\lambda); a_1, b_1$};
			\node[left] at (1,2) {\tiny $r_n; x_n$};
			\node[left] at (1,4) {\tiny $r_2; x_2$};
			\node[left] at (1,5) {\tiny $r_1; x_1$};
			\node[below] at (8.5,1) {\tiny $m_N(\mu); a_{N}, b_N$};
			\node[below] at (4.5,1) {\tiny $m_2(\mu); a_{2}, b_2$};
			\node[below] at (2,1) {\tiny $m_1(\mu); a_{1}, b_1$};
			\node[right] at (10,2) {\tiny $0;x_n$};
			\node[right] at (10,3) {$\vdots$};
			\node[right] at (10,4) {\tiny $0; x_2$};
			\node[right] at (10,5) {\tiny $0; x_1$};
		\end{tikzpicture}
	\end{center}
	The configuration occupies a strip with $n$ horizontal rows and $N$ columns. Every internal vertex carries the stochastic $w$ weight and the boundary conditions are
	\begin{itemize}
		\item the top edge of the $j$-th column carries the integer $m_j(\lambda)$ and the parameter pair $(a_j,b_j)$;
		\item the bottom edge of the $j$-th column carries the integer $m_j(\mu)$ and the parameter pair $(a_{j},b_j)$;
		\item the left unfused edge of the $i$-th row carries the integer $r_i$ and the parameter $x_i$;
		\item the right unfused edge of the $i$-th row carries the value $0$ and the parameter $x_i$.
	\end{itemize}
	The partition function $Zw^{r_1,\dots,r_n}_{\lambda/\mu,N}$ is the sum over all admissible configurations of the products of vertex weights. 
    
    Since $N>\lambda_1,\mu_1$, the conservation law forces all vertices in the last column to be surrounded by edges with $0$ labels. Hence the corresponding vertex weights are  all equal to $1$ and these vertices can be removed without changing the partition function. In other words, $Zw^{r_1,\dots,r_n}_{\lambda/\mu,N}$ does not depend on the number of columns $N$, as long as $\lambda,\mu$ are fixed and $N$ is sufficiently large. For this reason we omit $N$ from our notation, writing $Zw^{r_1,\dots,r_n}_{\lambda/\mu}$. Equivalently, one can think of $Zw^{r_1,\dots,r_n}_{\lambda/\mu}$ as a partition function of a semi-infinite strip with  finitely many non-trivial columns.

	\begin{definition}
		The \emph{inhomogeneous stable spin Hall-Littlewood function} $\sfF_{\lambda/\mu}(x_1,\dots,x_n\mid\calA,\calB)$ is defined by
		\begin{equation}\label{eq: def-sHL}
			\sfF_{\lambda/\mu}(x_1, \dots, x_n\mid\calA,\calB)
			=\sum_{r_1, \dots, r_n\in\{0,1\}}
			\prod_{i=1}^n(-x_i/b_0)^{r_i}\,
			Zw^{r_1, \dots, r_n}_{\lambda/\mu}(x_1, \dots, x_n\mid \calA,\calB).
		\end{equation}
	\end{definition}

    It is convenient to construct these functions row-by-row, using row operators. Let $V$ denote the vector space over $\calF$ with the basis vectors $\{|\lambda\rangle\}_{\lambda}$ enumerated by partitions. We use $\langle\lambda|\in V^*$ to denote the vectors dual to $\{|\lambda\rangle\}_{\lambda}$. For $i\in\{0,1\}$ define the row operator $T_i(x\mid\calA,\calB)$ as the matrix
    $$
    \langle\lambda|\ T_i(x\mid\calA,\calB)\ |\mu\rangle=		
    \begin{gathered}
    \begin{tikzpicture}
			\draw[unfused*] (1,2) -- (8.1,2);
			\node[right] at (8.1,2) {$\dots$};
			\draw[unfused] (8.8,2) -- (10,2);
			\foreach\x in {1,...,3}{
				\draw[fused] (2.5*\x-0.5,1) -- (2.5*\x-0.5,3);
			}
			\node[above right] at (2,2) {\tiny $w$};
			\node[above right] at (4.5,2) {\tiny $w$};
            \node[above right] at (7,2) {\tiny $w$};
			\node[above] at (7,3) {\tiny $m_3(\lambda); a_3, b_3$};
			\node[above] at (4.5,3) {\tiny $m_2(\lambda); a_2, b_2$};
			\node[above] at (2,3) {\tiny $m_1(\lambda); a_1, b_1$};
			\node[left] at (1,2) {\tiny $i; x$};
			\node[below] at (7,1) {\tiny $m_3(\mu); a_{3}, b_3$};
			\node[below] at (4.5,1) {\tiny $m_2(\mu); a_{2}, b_2$};
			\node[below] at (2,1) {\tiny $m_1(\mu); a_{1}, b_1$};
			\node[right] at (10,2) {\tiny $0;x$};
		\end{tikzpicture}
        \end{gathered}
    $$
    Note that the partition function on the right-hand side has at most one configuration with non-zero weight, due to the conservation law. In this configuration the horizontal edge entering column $j$ is forced to carry the label $\lambda'_{j}-\mu'_{j}$, which implies $\lambda'_j-\mu'_j\in\{0,1\}$. In other words, $\langle\lambda|\ T_i(x\mid\calA,\calB)\ |\mu\rangle=0$ unless $\lambda/\mu$ is a horizontal strip or, equivalently, unless $\mu\prec\lambda$. To get $\sfF_{\lambda/\mu}$, define
    $$
    B(x\mid\calA,\calB)= T_0(x\mid\calA,\calB)-xb_0^{-1} T_1(x\mid\calA,\calB).
    $$
    Then the vertex model definition of $\sfF_{\lambda/\mu}$ is equivalent to
    $$
    \sfF_{\lambda/\mu}(x_1,\dots,x_n\mid\calA,\calB)=\langle\lambda|\ B(x_1\mid\calA,\calB)B(x_2\mid\calA,\calB)\dots B(x_n\mid\calA,\calB)\ |\mu\rangle.
    $$
    From the interlacing, we have $\sfF_{\lambda/\mu}(x_1,\dots,x_n\mid\calA,\calB)=0$ unless $l(\lambda)\leq l(\mu)+n$.

    There is also an alternative description, in terms of the weights $w^*$. Define
    $T^*_i(y\mid\calA,\calB)$ as the matrix
    $$
    \langle\lambda|\ T_i^*(y\mid\calA,\calB)\ |\mu\rangle=q^{l(\mu)}\left(		
    \begin{gathered}
    \begin{tikzpicture}
			\draw[unfused] (7.9,2) -- (1,2);
			\node[right] at (7.9,2) {$\dots$};
			\draw[unfused*] (8.6,2) -- (9.3,2);
			\foreach\x in {1,...,3}{
				\draw[fused] (2.5*\x-0.5,1) -- (2.5*\x-0.5,3);
			}
			\node[above right] at (2,2) {\tiny $w^*$};
			\node[above right] at (4.5,2) {\tiny $w^*$};
            \node[above right] at (7,2) {\tiny $w^*$};
			\node[above] at (7,3) {\tiny $m_3(\lambda); a_3, b_3$};
			\node[above] at (4.5,3) {\tiny $m_2(\lambda); a_2, b_2$};
			\node[above] at (2,3) {\tiny $m_1(\lambda); a_1, b_1$};
			\node[below] at (1,1.9) {\tiny $i; q^{-1}y$};
			\node[below] at (7,1) {\tiny $m_3(\mu); a_{3}, b_3$};
			\node[below] at (4.5,1) {\tiny $m_2(\mu); a_{2}, b_2$};
			\node[below] at (2,1) {\tiny $m_1(\mu); a_{1}, b_1$};
			\node[below] at (9.3,2) {\tiny $0$};
		\end{tikzpicture}
        \end{gathered}\right)
    $$
    and set
    $$
    B^*(y\mid\calA,\calB)=T^*_0(y\mid\calA,\calB)+b_0y^{-1}T^*_1(y\mid\calA,\calB).
    $$
    \begin{proposition}\label{SHLrenorm}
        For every pair of partitions $\lambda,\mu$ we have
        $$
        \langle\mu| B^*(y\mid\calA,\calB)|\lambda\rangle=\frac{\psi_{\lambda}(\calA,\calB)}{\psi_\mu(\calA,\calB)}\langle\lambda| B(y^{-1}\mid\overline{\calB},\overline{\calA})|\mu\rangle,
        $$
        where 
    	\begin{equation}\label{eq:def-psi}
	       \psi_\lambda(\calA,\calB)=\psi_\lambda(\overline{\calB},\overline{\calA})
	        =
	       \prod_{r\geq 1}\left(\frac{b_{r-1}}{a_{r-1}}\right)^{\lambda'_{r}}
	       \frac{(q;q)_{m_r(\lambda)}}
	       {(a_r/b_r;q)_{m_r(\lambda)}}.
	\end{equation}
    \end{proposition}
    \begin{proof} Note that
     $$
     q^i\
    \begin{gathered}
            \begin{tikzpicture}[scale=1.3]
				\draw[unfused] (1,0) -- (-1,0);
				\draw[fused] (0,-1) -- (0,1);
				\node[below] at (-0.9,-0.1) {\tiny \shortstack[c]{$k;y/q$}};
				\node[below] at (0.9,-0.1) {\tiny \shortstack[c]{$j; y/q$}};
				\node[below] at (0,-1) {\tiny $i; a,b$};
				\node[above] at (0,1) {\tiny $l; a,b$};
				\node[above right] at (0,0) {\tiny $w^*$};
			\end{tikzpicture}
		\end{gathered}
        = (-1)^{k-j}(b/a)^{j}\frac{(q;q)_i(a/b;q)_l}{(a/b;q)_i(q;q)_l}\ 
        \begin{gathered}
			\begin{tikzpicture}[scale=1.3]
				\draw[unfused] (-1,0) -- (1,0);
				\draw[fused] (0,-1) -- (0,1);
				\node[below] at (-0.9,-0.1) {\tiny \shortstack[c]{$k;y^{-1}$}};
				\node[below] at (0.9,-0.1) {\tiny \shortstack[c]{$j; y^{-1}$}};
				\node[below] at (0,-1) {\tiny $l; b^{-1},a^{-1}$};
				\node[above] at (0,1) {\tiny $i; b^{-1},a^{-1}$};
				\node[above right] at (0,0) {\tiny $w$};
			\end{tikzpicture}
		\end{gathered}
    $$
    Setting $(a,b)=(a_r, b_r)$ and taking the product over all vertices, we get
    \begin{multline*}
    \langle\mu|\ T_k^*(y\mid\calA,\calB)\ |\lambda\rangle\\
    =(-1)^k\prod_{r\geq 1}\left(\frac{b_r}{a_r}\right)^{\lambda'_{r+1}-\mu'_{r+1}}\frac{(q;q)_{m_r(\lambda)}}
	       {(a_r/b_r;q)_{m_r(\lambda)}}\frac{(a_r/b_r;q)_{m_r(\mu)}}{(q;q)_{m_r(\mu)}}\langle\lambda|\ T_k(y^{-1}\mid\overline{\calB},\overline{\calA})\ |\mu\rangle.
    \end{multline*}
    To finish the proof we multiply both sides by $(b_0/y)^{k}$ and take the sum over $k\in\{0,1\}$. Note that for fixed $\lambda,\mu$ the sum has at most one nonzero term, when $k=\lambda'_1-\mu'_1$.
    \end{proof}

    We define the dual stable spin Hall-Littlewood functions by
    $$
    \sfF^*_{\lambda/\mu}(y_1,\dots, y_m\mid\calA, \calB)=\langle\mu| B^*(y^{-1}_m\mid\overline{\calB},\overline{\calA})\dots B^*(y^{-1}_2\mid\overline{\calB},\overline{\calA}) B^*(y^{-1}_1\mid\overline{\calB},\overline{\calA})|\lambda\rangle.
    $$
    Then
	\begin{equation}\label{eq:def-dual-f}
	\sfF^*_{\lambda/\mu}(y_1,\ldots,y_m\mid\calA,\calB)
	=
	\frac{\psi_\lambda(\calA,\calB)}
	{\psi_\mu(\calA,\calB)}
	\sfF_{\lambda/\mu}(y_1,\ldots,y_m\mid\calA,\calB).
	\end{equation}

	\subsection{Properties of $\sfF_{\lambda/\mu}$}\label{sec: sHLbasic-properties} Here we list some of the known properties of the stable spin Hall-Littlewood functions. These properties follow from \cite{BP18}, since our functions $\sfF_{\lambda/\mu}$ are a particular case of the functions studied there. See Appendix \ref{sHLmatch} for the exact match with the results of \cite{BP18}.

    \begin{proposition}[{\cite[Proposition 4.5]{BP18}}] The function $\sfF_{\lambda/\mu}(x_1,\dots,x_n\mid\calA,\calB)$ is a symmetric rational function in $x_1,\dots, x_n$.
    \end{proposition}

    \begin{proposition}[{\cite[Proposition 4.6]{BP18}}]  For every pair of partitions $\lambda,\nu$ we have	the following branching rule:
    \begin{equation}\label{eq:branching-rulesHL}
	\sfF_{\lambda/\nu}(x_1,x_2,\ldots,x_n\mid\calA,\calB)
	=
	\sum_{\mu\prec\lambda}
	\sfF_{\lambda/\mu}(x_1\mid\calA,\calB)\,
	\sfF_{\mu/\nu}(x_2,\ldots,x_n\mid\calA,\calB),
	\end{equation}
    where the sum is over partitions $\mu$ interlacing $\lambda$.
    \end{proposition}

    \begin{proposition}[{\cite[Corollary 4.10]{BP18}}]\label{thm:sHLskew-cauhy} Assume that $\calA=(a_0,a_1,a_2,\dots)$, $\calB=(b_0,b_1, b_2,\dots)$, $x_1, \dots,x_n$, $y_1, \dots, y_m, q$ are complex parameters satisfying 
    $$
    \prod_{k=M}^\infty \left|\frac{a_k/b_k-x_i/b_k}{1-x_i/b_k}\cdot\frac{1-b_ky_j}{1-a_ky_j}\right|=0
    $$
    for every $i,j,M$. Let $\mu,\nu$ be a pair of partitions. Then
	\begin{multline}\label{eq:sHLskew-cauchy}
	\sum_{\lambda}
	\sfF_{\lambda/\mu}(x_1,\ldots,x_n\mid\calA,\calB)\,
	\sfF^*_{\lambda/\nu}(y_1,\ldots,y_m\mid\overline{\calB},\overline{\calA})
	\\
	=
	\prod_{i=1}^{n}\prod_{j=1}^{m}
	\frac{1-qx_iy_j}{1-x_iy_j}
	\sum_{\rho}
	\sfF^*_{\mu/\rho}(y_1,\ldots,y_m\mid\overline{\calB},\overline{\calA})\,
	\sfF_{\nu/\rho}(x_1,\ldots,x_n\mid\calA,\calB).
	\end{multline}
    \end{proposition}
	
	Setting $\mu=\nu=\varnothing$ yields a spin version of the usual Cauchy identity:
	\begin{equation}\label{eq:sHLcauchy}
	\sum_{\lambda}
	\sfF_{\lambda}(x_1,\ldots,x_n\mid\calA,\calB)\,
	\sfF^*_{\lambda}(y_1,\ldots,y_m\mid\overline{\calB},\overline{\calA})
    =
	\prod_{i=1}^{n}\prod_{j=1}^{m}
	\frac{1-qx_iy_j}{1-x_iy_j}.
	\end{equation}

    \begin{proposition}[{\cite[Theorem 4.14]{BP18}}]\label{prop: symsHL}
    For $\lambda$ such that $l(\lambda)\leq n$ we have
    \begin{multline*}
        \mathsf{F}_{\lambda}(x_1, \dots, x_n\mid \calA,\calB)=\frac{(1-q)^n}{(q;q)_{n-l(\lambda)}}\prod_{j\geq1}b_{j-1}^{-\lambda'_{j}}\frac{(a_j/b_j;q)_{m_j(\lambda)}}{(q;q)_{m_j(\lambda)}}\\
        \times\sum_{\sigma\in S_n}\sigma\left(\prod_{1\leq i<j\leq n}\frac{x_i-qx_j}{x_i-x_j}\prod_{i=1}^n\upvarphi_{\lambda_i}(x_i\mid\calA,\calB)\right),
    \end{multline*}
    where
    $$
        \upvarphi_0(x\mid\calA,\calB)=1,\qquad \upvarphi_k(x\mid\calA,\calB)=\frac{-x}{1-x/b_k}\prod_{j=1}^{k-1}\frac{a_j-x}{1-x/b_j}\quad \text{for}\ k\geq 1.
    $$
    \end{proposition}

Note that setting $a_j=0$ and $b_j=\infty$ in Proposition \ref{prop: symsHL} yields
    \begin{multline*}
        (-1)^{|\lambda|}\prod_{j\geq1}b_{j-1}^{\lambda'_{j}}\mathsf{F}_{\lambda}(x_1, \dots, x_n\mid \calA,\calB)\Big |_{\substack{a_0=a_1=\dots=0\\ b_0=b_1=\dots=\infty}}\\
        =\frac{(1-q)^n}{(q;q)_{n-l(\lambda)}\prod_{j\geq 1}(q;q)_{m_j(\lambda)}}\sum_{\sigma\in S_n}\sigma\left(\prod_{1\leq i<j\leq n}\frac{x_i-qx_j}{x_i-x_j}\prod_{i=1}^nx_i^{\lambda_i}\right)=P_\lambda(x_1,\dots, x_n; q),
    \end{multline*}
where $P_\lambda(x_1,\dots, x_n; q)$ are the classical Hall-Littlewood polynomials \cite[Chapter III]{Mac95}. More generally, using the branching rule \eqref{eq:branching-rulesHL} and comparing the single row functions $\sfF_{\lambda / \mu}(x\mid\calA,\calB)=\langle\lambda|\ B(x\mid\calA,\calB)\ |\mu\rangle$ with \cite[Chapter III, (5.8')]{Mac95}, we have
$$
        (-1)^{|\lambda|-|\mu|}\prod_{j\geq1}b_{j-1}^{\lambda'_{j}-\mu'_{j}}\mathsf{F}_{\lambda/\mu}(x_1, \dots, x_n\mid \calA,\calB)\Big |_{\substack{a_0=a_1=\dots=0\\ b_0=b_1=\dots=\infty}}=P_{\lambda/\mu}(x_1,\dots, x_n; q).
$$

\section{Littlewood identity for spin Hall-Littlewood functions}\label{sec: sHLLit}

In this section we prove the Littlewood identity for the spin Hall-Littlewood functions. 

\subsection{Al-Salam-Chihara polynomials} For $n\in\mathbb Z_{\geq 0}$ let 
\begin{equation}\label{eq: hdef}
h_n(u,v\mid a)=\sum_{i+j=n} u^iv^{i}\frac{(a/v;q)_i(av;q)_j}{(a^2;q)_n}\frac{(q;q)_n}{(q;q)_i(q;q)_j}.
\end{equation}
Note that this function is a reparameterization of $\omega_{a;\xi,\zeta}(n)$ from \eqref{omegaex}:
\begin{equation}\label{eq: omegatoh}
\omega_{a;\xi,\zeta}(n)=\zeta^{-n}  h_n(\xi\zeta,\zeta/\xi\mid a).
\end{equation}

Up to a change of variables and rescaling, the functions $h_n$ are known as \emph{Al-Salam-Chihara} polynomials \cite{ASC76}. To be precise, we have 
$$
h_n(u,v\mid a)=\frac{(\sqrt{uv})^n}{(a^2;q)_n}Q_n\left(\frac{\sqrt{uv}^{-1}+\sqrt{uv}}{2};a\frac{\sqrt{u}}{\sqrt{v}},a\frac{\sqrt{v}}{\sqrt{u}}\mid q\right)
$$ 
where $Q_n$ are the Al-Salam-Chihara polynomials from \cite{KoekoekSwarttouw1998}:
\begin{multline*}
Q_n\left(\frac{t+t^{-1}}{2};\alpha,\beta\mid q\right)=\frac{(\alpha\beta;q)_n}{\alpha^n}\pPq{3}{2}{q^{-n}, \alpha t, \alpha/t}{\alpha\beta, 0}{q}{q}\\
=(\beta/t;q)_n t^n\ \pPq{2}{1}{q^{-n}, \alpha t}{q^{1-n}t/\beta}{q}{\frac{q}{\beta t}}= (\alpha t;q)_n t^{-n}\ \pPq{2}{1}{q^{-n}, \beta/t}{q^{1-n}t^{-1}\alpha^{-1}}{q}{\frac{qt}{\alpha}}.
\end{multline*}
When $a=0$, the functions $h_n(u,v\mid 0)=h_n(uv)$ become \emph{Rogers-Szeg\"o} polynomials:
$$
h_n(uv)= \sum_{i+j=n} (uv)^{i}\frac{(q;q)_n}{(q;q)_i(q;q)_j}.
$$

Below we provide several basic properties of $h_n(u,v\mid a)$.
\begin{proposition}\label{prop: ASCprop}
    (1) $h_n(u,v\mid a)$ have the following generating function
    $$
    \sum_{n\geq 0} h_n(u,v\mid a)\frac{(a^2;q)_n}{(q;q)_n} T^n=\frac{(auT, avT;q)_\infty}{(uvT, T;q)_\infty};
    $$
    
    (2) Both $h_n(u,v\mid a)$ and $v^n h_n(u,1/v\mid a)$ are symmetric polynomials in $u,v$;

    (3) With the convention $h_{-1}(u,v\mid a) = 0$, for every $n\geq 0$ the following recurrence relation holds
    $$h_{n+1}(u,v\mid a)
=
\frac{
1+uv-aq^n\left(u+v\right)
}{
1-a^2q^n
}
h_n(u,v\mid a)
-
\frac{
uv\left(1-q^n\right)
}{
1-a^2q^n
}
h_{n-1}(u,v\mid a);
    $$

    (4) We have
    $$
    h_n(a, v\mid a)=1,\qquad h_n(1/a, v\mid a)=(v/a)^n,\qquad h_n(0,v\mid a)=\frac{(av;q)_n}{(a^2;q)_n},
    $$
    $$
    h_n(-1,1\mid a)=\begin{cases} \prod_{i=1}^{n/2}\frac{1-q^{2i-1}}{1-a^2q^{2i-1}}\qquad &\text{if}\ n-even,\\
    0\qquad &\text{otherwise}.
    \end{cases}
    $$

\end{proposition}
\begin{proof} The expression for the generating function follows from the $q$-binomial identity in the same way as in the proof of Proposition \ref{Bomega1}. The symmetry follows immediately from the symmetry of the generating function. To get the recurrence, note that 
    $$
    \sum_{n\geq 0} h_n(u,v\mid a)\frac{(a^2;q)_n}{(q;q)_n} T^n=\frac{(1-auT)(1-avT)}{(1-uvT)(1-T)}\sum_{n\geq 0} h_n(u,v\mid a)\frac{(a^2;q)_n}{(q;q)_n} (qT)^n.
    $$
Multiplying both sides by $(1-uvT)(1-T)$ and comparing the coefficients of $T^n$ yields the recurrence relation. The last part follows from the generating function and the $q$-binomial theorem:
$$
\frac{(a^2T, avT;q)_\infty}{(avT, T;q)_\infty}=\frac{(a^2T;q)_\infty}{(T;q)_\infty}=\sum_{n\geq 0}\frac{(a^2;q)_n}{(q;q)_n} T^n,
$$
$$
\frac{(T, avT;q)_\infty}{(vT/a, T;q)_\infty}=\frac{(avT;q)_\infty}{(vT/a;q)_\infty}=\sum_{n\geq 0}\frac{(a^2;q)_n}{(q;q)_n} (vT/a)^n,
$$
$$
\frac{(0, avT;q)_\infty}{(0, T;q)_\infty}=\frac{(avT;q)_\infty}{(T;q)_\infty}=\sum_{n\geq 0}\frac{(av;q)_n}{(q;q)_n} T^n,
$$
$$
\frac{(aT, -aT;q)_\infty}{(T, -T;q)_\infty}=\frac{(a^2T^2;q^2)_\infty}{(T^2;q^2)_\infty}=\sum_{n\geq 0}\frac{(a^2;q^2)_n}{(q^2;q^2)_n} T^{2n}.
$$
\end{proof}

\subsection{Statement of the Littlewood identity}

Let $\calA=(a_0, a_1, a_2,\dots)$ be a sequence of complex parameters. For a partition $\lambda$ we define
$$
h_{\lambda}(u,v\mid \calA)=\prod_{\substack{t\geq2\\t-even}}h_{m_t(\lambda)}(u,v\mid a_t)\prod_{\substack{t\geq1\\t-odd}}v^{m_t(\lambda)}h_{m_t(\lambda)}(u,1/v\mid a_t).
$$
We also set 
$$
\calA^{\lambda}=\prod_{j\geq1} a_{j-1}^{\lambda'_{j}}, \qquad \overline{\calA}^{\lambda}=\prod_{j\geq1} a_{j-1}^{-\lambda'_{j}}.
$$


\begin{assumption}\label{asspt: sHL} Impose the following restrictions on parameters $\calA = (a_0, a_1, a_2,\dots)$ and a set of complex variables $x_1, \dots, x_n$:
\begin{enumerate}
\item All $a_0, a_1, a_2, \dots$ are nonzero.

\item All $a_1,a_2,\dots$ are contained in a compact subset $S\subset \mathbb C$ such that $x_1^{-1},\dots, x_n^{-1}\notin S$ and $\{\pm q^{-\frac{n}{2}} \mid n \in\mathbb Z_{\geq 0}\}\cap S=\varnothing$.

\item There exists $\varepsilon\in(0,1)$ such that for every $i=1,\dots, n$ and $r\geq 1$ we have
$$
\left|\frac{a_r-x_i}{1-x_ia_r}\right|<\varepsilon.
$$
\end{enumerate}
\end{assumption}

	\begin{theorem}[Skew Littlewood identity]\label{thm:sHLlittlewood}
		Let $\calA=(a_0, a_1,\dots)$, $x_1, \dots, x_n$ and $u,v$ be parameters satisfying Assumption \ref{asspt: sHL}. Then for every partition $\mu$ we have 
		\begin{multline*}
			\sum_{\lambda} h_{\lambda}(u,v\mid\calA)\overline{\calA}^{\lambda}\,
			\sfF_{\lambda/\mu}(x_1, \dots, x_{n}\mid\calA,\overline{\calA})
			\\
            =\prod_{i=1}^n\frac{(1-ux_i)(1-vx_i)}{1-x^2_i}\prod_{1\leq i<j\leq n}\frac{1-qx_ix_j}{1-x_ix_j}\sum_{\lambda} h_{\lambda}(u,v\mid\calA)\,
			\overline{\calA}^{\lambda}\sfF^*_{\mu/\lambda}(x_1, \dots, x_{n}\mid\calA,\overline{\calA}).
		\end{multline*}
		\end{theorem}

Setting $\mu=\varnothing$ we get the ordinary Littlewood identity:

	\begin{corollary}[Littlewood identity]\label{cor:sHLlittlewood}
		Let $\calA=(a_0, a_1,\dots)$, $x_1, \dots, x_n$ and $u,v$ be parameters satisfying Assumption \ref{asspt: sHL}. Then
		$$
			\sum_{\lambda} h_{\lambda}(u,v\mid\calA)\overline{\calA}^{\lambda}\,
			\sfF_{\lambda}(x_1, \dots, x_{n}\mid\calA,\overline{\calA})=\prod_{i=1}^n\frac{(1-ux_i)(1-vx_i)}{1-x^2_i}\prod_{1\leq i<j\leq n}\frac{1-qx_ix_j}{1-x_ix_j}.
		$$
	\end{corollary}

Before proving Theorem \ref{thm:sHLlittlewood}, we list several notable cases of our Littlewood identity.

First we set $u=0$. By Proposition \ref{prop: ASCprop} we have
\begin{equation}
		\sum_{\lambda} h_\lambda(0,v\mid\calA) \overline{\calA}^{\lambda}\,
			\sfF_{\lambda}(x_1, \dots, x_{n}\mid\calA,\overline{\calA})=\prod_{i=1}^n\frac{1-vx_i}{1-x^2_i}\prod_{1\leq i<j\leq n}\frac{1-qx_ix_j}{1-x_ix_j},
\end{equation}
where
$$
h_\lambda(0,v\mid\calA)=\prod_{\substack{t\geq 2\\t-even}}\frac{(va_t;q)_{m_t(\lambda)}}{(a_t^2;q)_{m_t(\lambda)}}\prod_{\substack{t\geq 1\\t-odd}}v^{m_t(\lambda)}\frac{(a_t/v;q)_{m_t(\lambda)}}{(a_t^2;q)_{m_t(\lambda)}}.
$$
In particular, setting $v=1$ yields 
\begin{equation}
		\sum_{\lambda} \left(\prod_{t\geq 1}\frac{(a_t;q)_{m_t(\lambda)}}{(a_t^2;q)_{m_t(\lambda)}}\right) \overline{\calA}^{\lambda}\,
			\sfF_{\lambda}(x_1, \dots, x_{n}\mid\calA,\overline{\calA})=\prod_{i=1}^n\frac{1}{1+x_i}\prod_{1\leq i<j\leq n}\frac{1-qx_ix_j}{1-x_ix_j}.
\end{equation}

Next set $v=1$ and $u=-1$. Then we get the following version of Littlewood identity:
\begin{equation}\label{eventransposed_li}
\begin{aligned} 
&\sum_{\lambda:\,\lambda'\text{even}}
\prod_{t\ge1}\prod_{r=1}^{m_t(\lambda)/2}
\frac{1-q^{2r-1}}{1-a_t^2q^{2r-1}}\,
\overline{\mathcal A}^{\,\lambda}\,
\sfF_{\lambda/\mu}
(x_1,\ldots,x_n\mid\mathcal A,\overline{\mathcal A})
\\[2mm]
&\quad=
\prod_{1\le i<j\le n}
\frac{1-qx_ix_j}{1-x_ix_j}
\sum_{\lambda:\,\lambda'\text{even}}
\prod_{t\ge1}\prod_{r=1}^{m_t(\lambda)/2}
\frac{1-q^{2r-1}}{1-a_t^2q^{2r-1}}\,
\overline{\mathcal A}^{\,\lambda}\,
\sfF^{*}_{\mu/\lambda}
(x_1,\ldots,x_n\mid\mathcal A,\overline{\mathcal A}),
\end{aligned}
\end{equation}
where we call a partition $\mu$ even if all its parts are even. In other words, the sums are over partitions $\lambda$ with columns of even length.

Several additional interesting cases appear in the homogeneous case, when $a_0=a_1=a_2=\dots=a$. In this case we write $a$ instead of $\calA$,  writing $\sfF_{\lambda/\mu}(x_1,\dots, x_n\mid a, a^{-1})$ and $h_\lambda(u,v\mid a)$. Then Proposition \ref{prop: ASCprop} implies
$$
h_\lambda(u,a\mid a)=u^{l(\lambda_{odd})},
$$
where $\lambda_{odd}$ is the partition obtained by removing all nonzero even parts from $\lambda$. Plugging this value into the Littlewood identity yields 
\begin{equation}\label{eq: evenLittlewood1}
		\sum_{\lambda} u^{l(\lambda_{odd})} a^{-|\lambda|}\,
			\sfF_{\lambda}(x_1, \dots, x_{n}\mid a,a^{-1})=\prod_{i=1}^n\frac{(1-ax_i)(1-ux_i)}{1-x^2_i}\prod_{1\leq i<j\leq n}\frac{1-qx_ix_j}{1-x_ix_j}.
\end{equation}
In particular, setting $u=0$ results in
\begin{equation} \label{even_li}
		\sum_{\lambda:\lambda \text{ even}}a^{-|\lambda|}\,
			\sfF_{\lambda}(x_1, \dots, x_{n}\mid a,a^{-1})=\prod_{i=1}^n\frac{1-ax_i}{1-x^2_i}\prod_{1\leq i<j\leq n}\frac{1-qx_ix_j}{1-x_ix_j}.
\end{equation}

Finally, when $a_0=a_1=a_2=\dots=a\to0$, the functions $h_m(u,v\mid 0)$ become Rogers-Szeg\"o polynomials, while
$$
        \overline{\calA}^{\lambda}\mathsf{F}_{\lambda}(x_1, \dots, x_n\mid \calA,\overline{\calA})\Big |_{a_0=a_1=\dots=0}=P_{\lambda}(-x_1,\dots, -x_n; q).
$$
So, we get the Littlewood identity for the classical Hall-Littlewood polynomials found in \cite{War06}.

\begin{remark}
    We can get a slightly more general version of \eqref{eq: evenLittlewood1} and \eqref{even_li} by  keeping the even parameters $a_{2n}$ and fixing only the odd ones $a_{1}=a_3=\dots=a$. However, we did not find a way to obtain a similar identity for the full family $\calA$.
\end{remark}

\subsection{Row reflection} For the remainder of the section we focus on the proof of Theorem \ref{thm:sHLlittlewood}. First we describe an application of the reflection equation from Proposition \ref{prop: hsref} to row operators. 

Fix parameters $\calA, u,v, x$ such that $a_i^2\notin q^{-\mathbb Z}$ and $a_i\neq x^{-1}$ for all $i$. This ensures that all vertex weights $w_{x;a_i, a_i^{-1}}$ are well-defined, as well as $h_\lambda(u,v\mid\calA)$. Let $\langle\Omega_{u,v}(\calA)|\in V^*$ denote the following dual vector
\begin{equation}\label{eq: dOmegadef}
\langle\Omega_{u,v}(\calA)|=\sum_{\lambda}h_{\lambda}(u,v\mid\calA)\overline{\calA}^{\lambda} \langle\lambda|.
\end{equation}
Additionally assuming $v\neq0$ and setting $\xi=\sqrt{u}/\sqrt{v}$ and $\zeta=\sqrt{u}\sqrt{v}$, \eqref{eq: omegatoh} implies
\begin{equation*}
\langle\Omega_{u,v}(\calA)|=\sum_{\lambda}\overline{\calA}^{\lambda}\zeta^{\lambda_1'}\prod_{\substack{t\geq2\\ t-even}}\omega_{a_{t};\xi,\zeta}(m_t(\lambda))\prod_{\substack{t\geq1\\ t-odd}}\omega_{a_{t};\zeta,\xi}(m_t(\lambda))\ \langle\lambda|.
\end{equation*}

For $N\geq 0$ we define a partial version of $\langle\Omega_{u,v}(\calA)|$ by
\begin{equation*}
\langle\Omega_{u,v}^{\leq N}(\calA)|=\sum_{\lambda:\lambda_1\leq N} \langle\Omega_{u,v}(\calA)\mid\lambda\rangle\ \langle\lambda|.
\end{equation*}
For $N>0$ we also define the following row operator, which only acts on partitions $\lambda,\mu$ satisfying $\lambda_1,\mu_1\leq N$:
$$
    \langle\lambda|\ T^{-,\leq N}_i(x\mid\calA,\calB)\ |\mu\rangle=		
    \begin{gathered}
    \begin{tikzpicture}
			\draw[unfused*] (1,2) -- (6.6,2);
			\node[right] at (6.6,2) {$\dots$};
			\draw[unfused] (7.2,2) -- (10,2);
			\foreach\x in {1,...,2}{
				\draw[fused] (2.5*\x-0.5,1) -- (2.5*\x-0.5,3);
			}
				\draw[fused] (8.5,1) -- (8.5,3);
			\node[above right] at (2,2) {\tiny $w$};
			\node[above right] at (4.5,2) {\tiny $w$};
            \node[above right] at (8.5,2) {\tiny $w$};
			\node[above] at (8.5,3) {\tiny $m_N(\lambda); a_N, b_N$};
			\node[above] at (4.5,3) {\tiny $m_2(\lambda); a_2, b_2$};
			\node[above] at (2,3) {\tiny $m_1(\lambda); a_1, b_1$};
			\node[left] at (1,2) {\tiny $i; x$};
			\node[below] at (8.5,1) {\tiny $m_N(\mu); a_{N}, b_N$};
			\node[below] at (4.5,1) {\tiny $m_2(\mu); a_{2}, b_2$};
			\node[below] at (2,1) {\tiny $m_1(\mu); a_{1}, b_1$};
			\node[right] at (10,2) {\tiny $1;x$};
		\end{tikzpicture}
        \end{gathered}
    $$

    \begin{proposition}
        Let $N\in2\mathbb Z_{>0}$. For every $\mu$ such that $\mu_1\leq N$ and $i\in \{0,1\}$ we have
        \begin{multline}\label{eq: Trowreflection}
       \langle \Omega_{u,v}^{\leq N}(\calA)|\ T_i(x\mid\calA,\overline{\calA})\ |\mu\rangle +  \frac{u+v-xuv-x}{(1-x^2)\prod_{r= 0}^Na_r} \langle \Omega_{u,v}^{\leq N}(\calA)|\ T^{-,\leq N}_i(x\mid\calA,\overline{\calA})\ |\mu\rangle\\
       =\sum_{j\in\{0,1\}}q^{-\frac j2}\tK_{x,qx;\zeta,\xi}(i,j)(\zeta/a_0)^{i+j}\langle \Omega_{u,v}(\calA)|\  T^*_j(x^{-1}\mid\calA,\overline{\calA})\ |\mu\rangle,
        \end{multline}
        where $\xi=\sqrt{u}/\sqrt{v}$ and $\zeta=\sqrt{u}\sqrt{v}$.
    \end{proposition}
     Using Proposition \ref{prop: hsref}, one can check that $\tK_{x,qx;\zeta,\xi}(i,j)\zeta^{i+j}$ depends polynomially on $u$ and $v$. So, both sides of \eqref{eq: Trowreflection} are polynomials in $u,v$ and are well-defined even when $v=0$. 
    \begin{proof}
        We use a version of the standard ``zipper" argument. Set $\xi=\sqrt{u}/\sqrt{v}$ and $\zeta=\sqrt{u}\sqrt{v}$. Consider the following row partition function with $N$ vertices of type $\widetilde{w}$, followed by a reflection weight $\tK_{x,qx;\zeta,\xi}$:
        \begin{equation}\label{eq: sHLrowright}
        \begin{gathered}
          \begin{tikzpicture}
			\draw[unfused*] (1,2) -- (6.1,2);
			\node[right] at (6.1,2) {$\dots$};
			\draw[unfused] (6.8,2) -- (9,2)--(10,3);
            \draw[unfused] (11,2)--(10,3);
			\foreach\x in {1,...,2}{
				\draw[fused] (2.5*\x-0.5,1) -- (2.5*\x-0.5,3);
			}
				\draw[fused] (8,1) -- (8,3);
			\node[above right] at (2,2) {\tiny $\widetilde{w}$};
			\node[above right] at (4.5,2) {\tiny $\widetilde{w}$};
            \node[above right] at (8,2) {\tiny $\widetilde{w}$};
                   \node at (8,3)  {$\text{---}$};
                    \node at (4.5,3)  {$\text{---}$};
                    \node at (2,3)  {$\text{---}$};
			\node[above] at (8,3) {\tiny $\omega_{a_N;\xi,\zeta}$};
			\node[above] at (4.5,3) {\tiny $\omega_{a_2;\xi,\zeta}$};
			\node[above] at (2,3) {\tiny $\omega_{a_1;\zeta,\xi}$};
            \node[above] at (10,3) {\tiny $\tK_{\zeta,\xi}$};
            \node at (10,3) {$\bullet$};
			\node[left] at (1,2) {\tiny $i; x$};
			\node[below] at (8.5,1) {\tiny $m_N(\mu); a_{N}, a^{-1}_N$};
			\node[below] at (4.5,1) {\tiny $m_2(\mu); a_{2}, a^{-1}_2$};
			\node[below] at (2,1) {\tiny $m_1(\mu); a_{1}, a^{-1}_1$};
			\node[right] at (11,2) {\tiny $0;q^{-1}x^{-1}$};
	   	\end{tikzpicture}
        \end{gathered}
        \end{equation}
        Repeatedly using Proposition \ref{prop: hsref}, we can move the reflection vertex to the left end of the row. Note that each application of Proposition \ref{prop: hsref} swaps the reflection weights 
        $\tK_{\zeta,\xi}$ and $\tK_{\xi,\zeta}$. However, the matching alternation between $\omega_{a_t;\zeta,\xi}$ and $\omega_{a_t;\xi,\zeta}$ allows us to repeat the argument throughout the entire row. So, \eqref{eq: sHLrowright} is equal to the following partition function:
        \begin{equation}\label{eq: sHLrowleft}
        \begin{gathered}
          \begin{tikzpicture}
            \draw[unfused] (-1,2) -- (0,3);
			\draw[unfused] (6.1,2) -- (1,2) -- (0,3);
			\node[right] at (6.1,2) {$\dots$};
			\draw[unfused*] (6.8,2) -- (9,2);
			\foreach\x in {1,...,2}{s
				\draw[fused] (2.5*\x-0.5,1) -- (2.5*\x-0.5,3);
			}
				\draw[fused] (8,1) -- (8,3);
			\node[above right] at (2,2) {\tiny $\widetilde{w}^*$};
			\node[above right] at (4.5,2) {\tiny $\widetilde{w}^*$};
            \node[above right] at (8,2) {\tiny $\widetilde{w}^*$};
                   \node at (8,3)  {$\text{---}$};
                    \node at (4.5,3)  {$\text{---}$};
                    \node at (2,3)  {$\text{---}$};
			\node[above] at (8,3) {\tiny $\omega_{a_N;\xi,\zeta}$};
			\node[above] at (4.5,3) {\tiny $\omega_{a_2;\xi,\zeta}$};
			\node[above] at (2,3) {\tiny $\omega_{a_1;\zeta,\xi}$};
            \node[above] at (0,3) {\tiny $\tK_{\zeta,\xi}$};
            \node at (0,3) {$\bullet$};
			\node[left] at (-1,2) {\tiny $i; x$};
			\node[below] at (8.5,1) {\tiny $m_N(\mu); a_{N}, a^{-1}_N$};
			\node[below] at (4.5,1) {\tiny $m_2(\mu); a_{2}, a^{-1}_2$};
			\node[below] at (2,1) {\tiny $m_1(\mu); a_{1}, a^{-1}_1$};
			\node[right] at (9,2) {\tiny $0;q^{-1}x^{-1}$};
	   	\end{tikzpicture}
        \end{gathered}
        \end{equation}
        Using Proposition \ref{rescale} to change the weights $\widetilde{w}$ to $w$, we get that the partition function in \eqref{eq: sHLrowright} is equal to
        \begin{multline*}
            q^{-\frac{\mu'_1}{2}}a_0^{i}\zeta^{-\mu'_1-i}\prod_{r\geq 1}a_{r-1}^{\mu'_{r}}\langle\Omega_{u,v}^{\leq N}|\ T_i(x\mid\calA,\overline{\calA})\ |\mu\rangle \tK_{x,qx;\zeta,\xi}(0,0) \\
            +q^{\frac{-\mu'_1}{2}}a_0^{i}\zeta^{-\mu'_1-i+1}\prod_{r= 1}^{N+1}a_{r-1}^{\mu'_{r}-1}\langle\Omega_{u,v}^{\leq N}|\ T^{-,\leq N}_i(x\mid\calA,\overline{\calA})\ |\mu\rangle \tK_{x,qx;\zeta,\xi}(1,0).
        \end{multline*}
        Similarly, the partition function in \eqref{eq: sHLrowleft} is
        $$
        \sum_{j\in\{0,1\}} \tK_{x,qx; \zeta,\xi}(i,j) q^{\frac{-\mu_1'-j}{2}}a_0^{-j}\zeta^{-\mu_1'+j}\prod_{r\geq 1} a_{r-1}^{\mu'_{r}} \langle\Omega_{u,v}^{\leq N}|\ T^*_j(x^{-1}\mid\calA,\overline{\calA})\ |\mu\rangle.
        $$
        Using the expressions above, the equality between \eqref{eq: sHLrowright} and \eqref{eq: sHLrowleft} implies the claim. Note that $\langle\lambda|\ T_j^*(y\mid\calA,\calB)\ |\mu\rangle$ vanishes unless $\lambda\prec \mu$, so we can replace $\Omega^{\leq N}_{u,v}$ by $\Omega_{u,v}$ in the right-hand side.
    \end{proof}

    \begin{proposition}\label{prop: Tinfiniterowreflection} Assume that 
    \begin{equation}\label{eq: vanishingcondition}
     \prod_{r=M}^\infty \frac{a_r-x}{1-xa_r} =0\qquad \text{for\ every\ }M\in\mathbb{Z}_{>0}.
    \end{equation}
    Then, for every partition $\mu$ and $i\in\{0,1\}$ we have
        \begin{equation}\label{eq: Tinfiniterowreflection}
       \langle \Omega_{u,v}(\calA)|\ T_i(x\mid\calA,\overline{\calA})\ |\mu\rangle =\sum_{j\in\{0,1\}}q^{-\frac j2}\tK_{x,qx;\zeta,\xi}(i,j)(\zeta/a_0)^{i+j}\langle \Omega_{u,v}(\calA)| T^*_j(x^{-1}\mid\calA,\overline{\calA}) |\mu\rangle,
        \end{equation}
        where $\xi=\sqrt{u}/\sqrt{v}$, $\zeta=\sqrt{u}\sqrt{v}$ and
        $$
        \langle \Omega_{u,v}(\calA)|\ T_i(x\mid\calA,\overline{\calA})\ |\mu\rangle:=\lim_{N\to\infty}\langle \Omega^{\leq N}_{u,v}(\calA)|\ T_i(x\mid\calA,\overline{\calA})\ |\mu\rangle.
        $$
    \end{proposition}
    \begin{proof} Fix $\mu$ and take the $N\to\infty$ limit of \eqref{eq: Trowreflection}. Note that the right-hand side does not depend on $N$, and the limit of the $T_i$ term in the left-hand side is $\langle \Omega_{\xi,\zeta}(\calA)|\ T_i(x\mid\calA,\overline{\calA})\ |\mu\rangle$. So, we only need to show that the contribution of the $T^{-,\leq N}$ term vanishes under this limit. To see it, consider the partition function $\langle\lambda|\ T^{-,\leq N}(x\mid\calA,\calB)\ |\mu\rangle$. If $N>\mu_1$, the bottom edge of the right-most vertex has value $0$, while the right edge of this vertex has value $1$. By the conservation law, the only possible non-zero weight of this vertex is
    $$
    w_{x;a_N,a_N^{-1}}(1,0,0,1)=\frac{a_N^2-xa_N}{1-xa_N}.
    $$
\begin{figure}
    \centering
    \begin{tikzpicture}
			\draw[unfused*] (1,2) -- (6.6,2);
			\node[right] at (6.6,2) {$\dots$};
			\draw[unfused] (7.3,2) -- (10,2);
			\foreach\x in {1,...,2}{
				\draw[fused] (2.5*\x-0.5,1) -- (2.5*\x-0.5,3);
			}
				\draw[fused] (8.5,1) -- (8.5,3);
                \draw[black, line width = 1pt] (3,2) -- (6.6,2);
                \draw[black, line width = 1pt] (7.3,2) -- (10,2);
			\node[above right] at (2,2) {\tiny $w$};
			\node[above right] at (4.5,2) {\tiny $w$};
            \node[above right] at (8.5,2) {\tiny $w$};
			\node[above] at (8.5,3) {\tiny $0; a_N, a^{-1}_{N}$};
			\node[above] at (4.5,3) {\tiny $0; a_{L+1}, a^{-1}_{L+1}$};
			\node[above] at (2,3) {\tiny $m_L(\lambda); a_{L}, a^{-1}_L$};
			\node[left] at (1,2) {\tiny $\dots$};
			\node[below] at (8.5,1) {\tiny $0; a_{N}, a^{-1}_N$};
			\node[below] at (4.5,1) {\tiny $0; a_{L+1}, a^{-1}_{L+1}$};
			\node[below] at (2,1) {\tiny $m_L(\mu); a_{L}, a^{-1}_L$};
			\node[right] at (10,2) {\tiny $1;x$};
		\end{tikzpicture}
    \caption{The frozen ``tail" of $\langle\lambda|\ T^{-,\leq N}(x\mid\calA,\calB)\ |\mu\rangle$. Here $L=\mu_1$ and the filled horizontal line indicates the unique configuration with non-zero weight.}
    \label{fig:tailTm}
\end{figure}
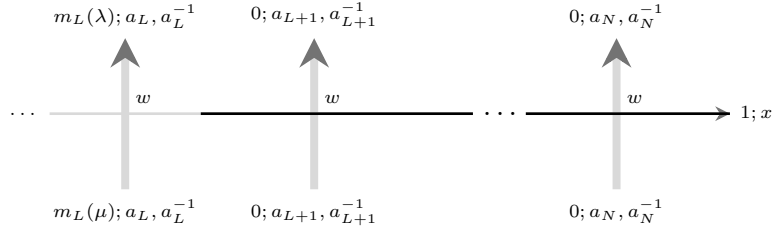

    In this case the horizontal edge between columns $N-1$ and $N$ has value $1$ and we can repeat the same argument for the remaining partition function with $N-1$ columns, as long as $N>\mu_1$. So, the only nontrivial row configuration has the form depicted in Figure \ref{fig:tailTm}, and we have
    $$
    \langle\lambda| T_i^{-,\leq N}(x\mid\calA,\overline{\calA})|\mu\rangle=\begin{cases} \langle\lambda| T_i^{-,\leq \mu_1}(x\mid\calA,\overline{\calA})|\mu\rangle \prod_{r=\mu_1+1}^N\frac{a_r^2-xa_r}{1-xa_r}\qquad &\text{if}\ \lambda_1\leq\mu_1,\\ 0\qquad &\text{otherwise}.
    \end{cases}
    $$
    Using this expression, the absolute value of the $T_i^{-,\leq N}$ term in \eqref{eq: Trowreflection} is bounded above by
    $$
    C\left|\prod_{r=\mu_1+1}^N\frac{a_r-x}{1-xa_r} \right|,
    $$
    where $C$ is a constant which does not depend on $N$. Due to \eqref{eq: vanishingcondition}, this term vanishes in the limit $N\to\infty$. 
    \end{proof}

    \begin{corollary}\label{cor: TinfiniteBrowreflection} Assume that 
    $$
     \prod_{r=M}^\infty \frac{a_r-x}{1-xa_r} =0\qquad \text{for\ every\ }M\in\mathbb{Z}_{>0}.
    $$
    Then, for every partition $\mu$, we have
        \begin{equation}\label{eq: TinfiniteBrowreflection}
       \langle \Omega_{u,v}(\calA)|\ B(x\mid\calA,\overline{\calA})\ |\mu\rangle =\frac{(1-xu)(1-xv)}{1-x^2}\langle \Omega_{u,v}(\calA)|\ B^*(x^{-1}\mid\calA,\overline{\calA})\ |\mu\rangle,
        \end{equation}
        where 
        $$
        \langle \Omega_{u,v}(\calA)|\ B(x\mid\calA,\overline{\calA})\ |\mu\rangle:=\lim_{N\to\infty}\langle \Omega^{\leq N}_{u,v}(\calA)|\ B(x\mid\calA,\overline{\calA})\ |\mu\rangle.
        $$
    \end{corollary}
    \begin{proof}
    Follows immediately from the definition of $B, B^*$, Proposition \ref{prop: Tinfiniterowreflection} and the relation
    $$
    \sum_{i\in\{0,1\}}(-x)^iq^{-\frac j2}\zeta^{i+j}\tK_{x,qx;\zeta,\xi}(i,j)=\frac{(1-x\zeta/\xi)(1-x\zeta\xi)}{1-x^2} x^j.
    $$
    \end{proof}

\subsection{Proof of Theorem \ref{thm:sHLlittlewood}} From now we fix $\calA$ and $x_1,\dots, x_n$ satisfying Assumption \ref{asspt: sHL}. The following upper bound ensures absolute convergence of all series throughout the proof. 

\begin{lemma}\label{lemma: bounds} (1) Let $k\in\mathbb Z_{\geq0}$. There exists a constant $C$ depending on $k$ such that for all partitions $\lambda$ satisfying $l(\lambda)\leq k$ we have 
$$
|h_\lambda(u,v\mid\calA)|\leq C.
$$

(2) Let $k\in\mathbb Z_{\geq0}$, $n\in\mathbb Z_{>0}$. There exists a constant $C$ depending on $k$ and $n$ such that for every pair of partitions $\lambda,\mu$ satisfying $l(\mu)\leq k$ we have
$$
|\overline{\calA}^{\lambda-\mu}\sfF_{\lambda/\mu}(x_1, \dots, x_n\mid\calA,\overline\calA)|\leq C\lambda_1^{(n-1)(k+\frac{n}{2})}\varepsilon^{|\lambda/\mu|},
$$
$$
|\overline{\calA}^{\mu-\lambda}\sfF^*_{\lambda/\mu}(x_1, \dots, x_n\mid\calA,\overline\calA)|\leq C\lambda_1^{(n-1)(k+\frac{n}{2})}\varepsilon^{|\lambda/\mu|}.
$$
Here $\varepsilon\in(0,1)$ is the number from part (3) of Assumption \ref{asspt: sHL}.
\end{lemma} 
\begin{proof} (1) Let $S$ be the compact set from Assumption \ref{asspt: sHL}. Note that $h_m(u,v\mid a)$ is continuous as a function of $a\in S$ for all $m$, so there are constants $c_m$ such that $|h_m(u,v\mid a)|,|v^mh_m(u,1/v\mid a)| <c_m$ for all $a\in S$. Since $l(\lambda)\leq k$, there are at most $k$ indices $t$ such that $m_t(\lambda)>0$, and for such $t$ we have $m_t(\lambda)\leq k$. So $h_\lambda(u,v\mid\calA)$ is bounded by $(\max_{i=1,\dots,k} c_i)^k$.

(2) We focus on $\sfF$, the bound for $\sfF^*$ then follows due to
$$
\overline{\calA}^{\mu-\lambda}\sfF^*_{\lambda/\mu}(x_1, \dots, x_n\mid\calA,\overline\calA)=\overline{\calA}^{\lambda-\mu}\prod_{r\geq 1}\frac{(a_r^2;q)_{m_r(\mu)}}{(q;q)_{m_r(\mu)}}\frac{(q;q)_{m_r(\lambda)}}{(a_r^2;q)_{m_r(\lambda)}}\sfF_{\lambda/\mu}(x_1, \dots, x_n\mid\calA,\overline\calA).
$$
We use induction on $n$. For the base case $n=1$, the function $\sfF_{\lambda/\mu}(x_1\mid\calA,\overline\calA)$ is computed using the single row partition function $\langle\lambda|B(x\mid\calA,\overline{\calA})|\mu\rangle$. Multiplication by $\overline{\calA}^{\lambda-\mu}=\prod_{r\geq0}a_{r-1}^{\mu'_{r}-\lambda'_{r}}$ is equivalent to multiplying the weights $w_{x;a_r,a_r^{-1}}(i,j,k,l)$ in this row partition function by $a_r^{-l}$, which results in the weights 

\begin{align*}
\begin{array}{cccc}
\begin{tikzpicture}[scale=0.9, baseline=-4pt]
	\draw[unfused] (-1,0) -- (1,0);
	\draw[fused] (0,-1) -- (0,1);
	\node[left] at (-1,0) {\tiny $0;x$};\node[right] at (1,0) {\tiny $0;x$};
	\node[below] at (0,-1) {\tiny $g;a_r,a_r^{-1}$};\node[above] at (0,1) {\tiny $g;a_r,a_r^{-1}$};
	\node[above right] at (0,0) {};
\end{tikzpicture}
&
\begin{tikzpicture}[scale=0.9, baseline=-4pt]
	\draw[unfused] (-1,0) -- (1,0);
	\draw[fused] (0,-1) -- (0,1);
	\node[left] at (-1,0) {\tiny $0;x$};\node[right] at (1,0) {\tiny $1;x$};
	\node[below] at (0,-1) {\tiny $g;a_r,a_r^{-1}$};\node[above] at (0,1) {\tiny $g-1;a_r,a_r^{-1}$};
	\node[above right] at (0,0) {};
\end{tikzpicture}
&
\begin{tikzpicture}[scale=0.9, baseline=-4pt]
	\draw[unfused] (-1,0) -- (1,0);
	\draw[fused] (0,-1) -- (0,1);
	\node[left] at (-1,0) {\tiny $1;x$};\node[right] at (1,0) {\tiny $0;x$};
	\node[below] at (0,-1) {\tiny $g;a_r,a^{-1}_{r}$};\node[above] at (0,1) {\tiny $g+1;a_r,a_r^{-1}$};
	\node[above right] at (0,0) {};
\end{tikzpicture}
&
\begin{tikzpicture}[scale=0.9, baseline=-4pt]
	\draw[unfused] (-1,0) -- (1,0);
	\draw[fused] (0,-1) -- (0,1);
	\node[left] at (-1,0) {\tiny $1;x$};\node[right] at (1,0) {\tiny $1;x$};
	\node[below] at (0,-1) {\tiny $g;a_r,a_r^{-1}$};\node[above] at (0,1) {\tiny $g;a_r,a_r^{-1}$};
	\node[above right] at (0,0) {};
\end{tikzpicture}
\\
\\
\dfrac{1-q^gxa_r}{1-xa_r}
\qquad
&
\dfrac{(q^g-1)x}{1-xa_r}
\qquad
&
\dfrac{1-q^ga_r^2}{1-xa_r}
\qquad
&
\dfrac{q^ga_r-x}{1-xa_r}
\vspace{1ex}
\end{array}
\end{align*}
Note that, due to conservation laws, the unique nonzero configuration of $\langle\lambda|B(x\mid\calA,\overline{\calA})|\mu\rangle$ has $|\lambda/\mu|$ horizontal occupied edges. At the same time, there are $k$ occupied vertical edges at the bottom and $k+1$ vertical edges at the top boundary. So, there are at least $|\lambda/\mu|-2k-1$ vertices with the weight $\frac{a_r-x}{1-xa_r}$, and at most $2k+1$ vertices with the weight other than $1$ and  $\frac{a_r-x}{1-xa_r}$. Due to our assumptions on the parameters, $\left|\frac{a_r-x}{1-xa_r}\right|<\varepsilon$, while all other weights can be bounded above by a sufficiently large constant $D$. So
$$
\left|\langle\lambda|\, B(x\mid\calA,\overline{\calA})\, |\mu\rangle\right|\leq D^{2k+1} \varepsilon^{|\lambda/\mu|-2k-1},
$$
which proves the induction base. 

For the induction step, assume that we have proved the claim for $n$. Then, by the branching rule and the inductive hypothesis, we have
\begin{multline*}
|\overline{\calA}^{\lambda-\mu}\sfF_{\lambda/\mu}(x_1, \dots, x_n, x_{n+1}\mid\calA,\overline\calA)|\\
\leq\sum_{\nu:\nu\prec\lambda} |\overline{\calA}^{\lambda-\nu}\sfF_{\lambda/\nu}(x_{n+1}\mid\calA,\overline\calA)||\overline{\calA}^{\nu-\mu}\sfF_{\nu/\mu}(x_1, \dots, x_n\mid\calA,\overline\calA)|\\
\leq \sum_{\nu:\nu\prec\lambda} C_1\varepsilon^{|\lambda/\nu|} C_2 \nu_1^{(n-1)(k+\frac{n}{2})}\varepsilon^{|\nu/\mu|}
\end{multline*}
for some constants $C_1, C_2$ depending only on $n,k$. Since $\sfF_{\lambda/\mu}(x_1, \dots, x_n, x_{n+1}\mid\calA,\overline\calA)$ vanishes unless $l(\lambda)\leq l(\mu)+n+1=n+k+1$, we have $\nu_1\leq \lambda_1$ and $l(\nu)\leq l(\lambda)-1\leq n+k$. In particular, there are at most $\lambda_1^{n+k}$ potential choices for $\nu\prec\lambda$. This implies $\sum_{\nu:\nu\prec\lambda}\nu_1^{(n-1)(k+\frac{n}{2})}\leq\lambda_1^{n(k+\frac{n+1}{2})}$, finishing the computation of the upper bound. 
\end{proof}

Now we prove Theorem \ref{thm:sHLlittlewood} by induction on $n$.

The base case $n=1$ follows immediately from Corollary \ref{cor: TinfiniteBrowreflection}: by the definitions of $\sfF_{\lambda/\mu}$ and $\sfF^*_{\lambda/\mu}$, \eqref{eq: TinfiniteBrowreflection} is equivalent to
\begin{equation*}
\sum_{\lambda} \langle\Omega_{u,v}(\calA)\mid\lambda\rangle\ \sfF_{\lambda/\mu}(x\mid\calA,\overline\calA)=\frac{(1-xu)(1-xv)}{1-x^2} \sum_{\lambda} \langle\Omega_{u,v}(\calA)\mid\lambda\rangle\ \sfF^*_{\mu/\lambda}(x\mid\calA,\overline\calA).
\end{equation*}
Since $\langle\Omega_{u,v}(\calA)\mid\lambda\rangle=h_\lambda(u,v\mid\calA)\overline{\calA}^\lambda$, this relation implies the inductive base.

For the inductive step, assume that Theorem \ref{thm:sHLlittlewood} holds for $n$. Let  
$$
\Upsilon_n=\prod_{i=1}^n\frac{(1-ux_i)(1-vx_i)}{1-x_i^2}\prod_{1\leq i<j\leq n}\frac{1-qx_ix_j}{1-x_ix_j}.
$$
Using the branching rule, we have
\begin{multline}\label{eq: sHLLproof1}
\sum_{\lambda} \overline{\calA}^\lambda h_\lambda(u,v\mid\calA)  \sfF_{\lambda/\mu}(x_1, \dots, x_{n+1}\mid\calA,\overline\calA)\\
=\sum_{\lambda}\sum_\nu \overline{\calA}^\lambda h_\lambda(u,v\mid\calA) \sfF_{\lambda/\nu}(x_1, \dots, x_{n}\mid\calA,\overline\calA)\sfF_{\nu/\mu}(x_{n+1}\mid\calA,\overline\calA).
\end{multline}
Due to Lemma \ref{lemma: bounds}, the sums converge absolutely, hence we can exchange the order of summation in the right-hand side. After applying the inductive hypothesis to the sum over $\lambda$, the right-hand side becomes 
\begin{equation}\label{eq: sHLLproof2}
\Upsilon_n \sum_{\nu}\sum_\lambda \overline{\calA}^\lambda h_\lambda(u,v\mid\calA) \sfF^*_{\nu/\lambda}(x_1, \dots, x_{n}\mid\calA,\overline\calA)\sfF_{\nu/\mu}(x_{n+1}\mid\calA,\overline\calA).
\end{equation}
Now we can switch the summation order again and use the skew Cauchy identity from Theorem \ref{thm:sHLskew-cauhy} to the sum with respect to $\nu$, getting
$$
\Upsilon_n\prod_{i=1}^n\frac{1-qx_ix_{n+1}}{1-x_ix_{n+1}}\sum_{\lambda}\sum_\nu \overline{\calA}^\lambda h_\lambda(u,v\mid\calA) \sfF^*_{\mu/\nu}(x_1, \dots, x_{n}\mid\calA,\overline\calA)\sfF_{\lambda/\nu}(x_{n+1}\mid\calA,\overline\calA).
$$
Applying the single variable case of Theorem \ref{thm:sHLlittlewood}, we obtain
\begin{multline*}
\Upsilon_n\frac{(1-ux_{n+1})(1-vx_{n+1})}{1-x_{n+1}^2}\prod_{i=1}^n\frac{1-qx_ix_{n+1}}{1-x_ix_{n+1}}\\\times\sum_{\nu}\sum_\lambda\overline{\calA}^\lambda h_\lambda(u,v\mid\calA) \sfF^*_{\mu/\nu}(x_1, \dots, x_{n}\mid\calA,\overline\calA)\sfF^*_{\nu/\lambda}(x_{n+1}\mid\calA,\overline\calA).
\end{multline*}
Theorem \ref{thm:sHLlittlewood} now follows from the branching rule and the relation
$$
\Upsilon_n\frac{(1-ux_{n+1})(1-vx_{n+1})}{1-x_{n+1}^2}\prod_{i=1}^n\frac{1-qx_ix_{n+1}}{1-x_ix_{n+1}}=\Upsilon_{n+1}.
$$
\qed

\begin{remark} One can show that Assumption \ref{asspt: sHL} is stronger than is needed for Theorem \ref{thm:sHLlittlewood} to hold. In fact, we only need \eqref{eq: vanishingcondition} and a way to justify the changes of summation orders in \eqref{eq: sHLLproof1} and \eqref{eq: sHLLproof2}. The latter can be done in other ways, for instance, when
$$
\overline{\calA}^{\lambda-\mu}\sfF_{\lambda/\mu}(x_1, \dots, x_n\mid\calA, \overline{\calA}), \overline{\calA}^{\mu-\lambda}\sfF^*_{\lambda/\mu}(x_1, \dots, x_n\mid\calA, \overline{\calA}), h_\lambda(u,v\mid\calA)\geq 0
$$
for all $\lambda,\mu$, the summation orders can be exchanged without additional justification due to all terms having the same sign.
\end{remark}

\section{Inhomogeneous spin $q$-Whittaker polynomials}\label{sec:sqW}

	In this section we recall the definition and properties of the inhomogeneous spin $q$-Whittaker polynomials $\F_{\lambda}$ introduced in~\cite{Kor24}.

    We also treat $q^{\frac12}$ as a formal parameter, with $\calF=\mathbb C(q^{\frac12})$. For the entire section we use $\calA=(a_0, a_1, a_2, \dots)$ and $\calB=(b_0, b_1, b_2, \dots)$ to denote a pair of sequences of generic parameters. Note that these sequences of parameters start from $a_0$ and $b_0$, which is more convenient here. 

	\subsection{Partition function}\label{sec: partition} Let $\lambda=(\lambda_1\geq\lambda_2\geq\dots\geq 0)$, $\mu=(\mu_1\geq\mu_2\geq\dots\geq 0)$ be a pair of partitions. For each $N>l(\lambda),l(\mu)$ and $(r_1,\dots,r_n)\in\Z_{\geq 0}^n$ we consider the partition function $ZW^{r_1,\dots,r_n}_{\lambda/\mu,N}(x_1,\dots, x_n\mid \calA,\calB)$ given by the following lattice configuration:
    
	
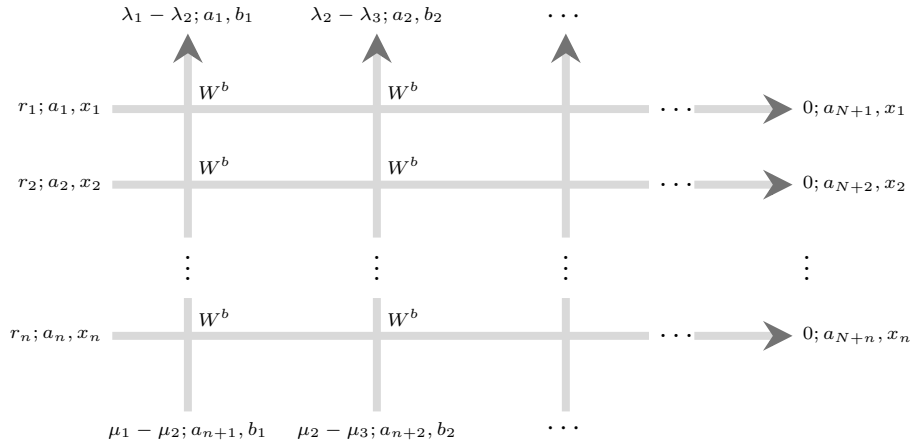
\begin{center}
		\begin{tikzpicture}
			\foreach\y in {4,...,5}{
				\draw[fused*] (1,\y) -- (8.1,\y);
				\node[right] at (8.1,\y) {$\dots$};
				\draw[fused] (8.7,\y) -- (10,\y);
			}
			\draw[fused*] (1,2) -- (8.1,2);
			\node[right] at (8.1,2) {$\dots$};
			\draw[fused] (8.7,2) -- (10,2);
			\foreach\x in {1,...,3}{
				\draw[fused] (2.5*\x-0.5,3.3) -- (2.5*\x-0.5,6);
				\draw[fused*] (2.5*\x-0.5,1) -- (2.5*\x-0.5,2.5);
				\node at (2.5*\x-0.5, 3) {$\vdots$};
			}
			\node[above right] at (2,5) {\tiny $W^b$};
			\node[above right] at (4.5,5) {\tiny $W^b$};
			\node[above right] at (2,4) {\tiny $W^b$};
			\node[above right] at (4.5,4) {\tiny $W^b$};
			\node[above right] at (2,2) {\tiny $W^b$};
			\node[above right] at (4.5,2) {\tiny $W^b$};
			\node[above] at (7,6) {$\cdots$};
			\node[above] at (4.5,6) {\tiny $\lambda_2-\lambda_3; a_2, b_2$};
			\node[above] at (2,6) {\tiny $\lambda_1-\lambda_2; a_1, b_1$};
			\node[left] at (1,2) {\tiny $r_n; a_n, x_n$};
			\node[left] at (1,4) {\tiny $r_2; a_2, x_2$};
			\node[left] at (1,5) {\tiny $r_1; a_1, x_1$};
			\node[below] at (7,1) {$\cdots$};
			\node[below] at (4.5,1) {\tiny $\mu_2-\mu_3; a_{n+2}, b_2$};
			\node[below] at (2,1) {\tiny $\mu_1-\mu_2; a_{n+1}, b_1$};
			\node[right] at (10,2) {\tiny $0;a_{N+n}, x_n$};
			\node[right] at (10,3) {$\vdots$};
			\node[right] at (10,4) {\tiny $0; a_{N+2}, x_2$};
			\node[right] at (10,5) {\tiny $0; a_{N+1}, x_1$};
		\end{tikzpicture}
	\end{center}
	The configuration occupies a strip with $n$ horizontal rows and $N$ columns. Every internal vertex carries the stochastic $W^b$ weight from Section \ref{stochW} and the boundary conditions are
	\begin{itemize}
		\item the top edge of the $j$-th column carries the label $\lambda_j-\lambda_{j+1}$ and parameter pair $(a_j,b_j)$;
		\item the bottom edge of the $j$-th column carries the label $\mu_j-\mu_{j+1}$ and parameter pair $(a_{n+j},b_j)$;
		\item the left edge of the $i$-th row carries the label $r_i$ and parameter pair $(a_i,x_i)$;
		\item the right edge of the $i$-th row carries the label $0$ and has parameters $(a_{N+i},x_i)$.
	\end{itemize}
Similarly to the spin Hall-Littlewood case, for fixed $\lambda,\mu$ the partition function $ZW^{r_1,\dots,r_n}_{\lambda/\mu,N}$ does not depend on $N$ as long as $N>l(\lambda),l(\mu)$. So, we omit $N$ from our notation, writing $ZW^{r_1,\dots,r_n}_{\lambda/\mu}$ instead. Equivalently, one can think of $ZW^{r_1,\dots,r_n}_{\lambda/\mu}$ as a partition function of a semi-infinite strip with infinitely many trivial columns.
    

	\begin{definition}
		The \emph{inhomogeneous spin $q$-Whittaker polynomial} $\F_{\lambda/\mu}(x_1,\dots,x_n\mid\calA,\calB)$ is defined by
		\begin{equation}\label{eq:def-F}
			\F_{\lambda/\mu}(x_1, \dots, x_n\mid\calA,\calB)
			=\sum_{r_1, \dots, r_n\geq 0}
			\prod_{i=1}^n(x_i/b_0)^{r_i}
			\frac{(a_ix_i^{-1};q)_{r_i}}{(q;q)_{r_i}}\,
			ZW^{r_1, \dots, r_n}_{\lambda/\mu}.
		\end{equation}
	\end{definition}

    The functions $\F_{\lambda/\mu}(x_1,\dots,x_n\mid\calA,\calB)$ can also be described with row operators. Let $V$ denote the vector space over $\calF$ spanned by the basis vectors $\{|\lambda\rangle\}_{\lambda}$ enumerated by partitions. For $i\in\mathbb Z_{\geq0}$ define the row operator $\T_i(x\mid\calA,\calB)$ as the matrix
    $$
    \langle\lambda'|\ \T_i(x\mid\calA,\calB)|\mu'\rangle=		
    \begin{gathered}
    \begin{tikzpicture}
			\draw[fused*] (1,2) -- (8.1,2);
			\node[right] at (8.1,2) {$\dots$};
			\draw[fused] (8.7,2) -- (10,2);
			\foreach\x in {1,...,3}{
				\draw[fused] (2.5*\x-0.5,1) -- (2.5*\x-0.5,3);
			}
			\node[above right] at (2,2) {\tiny $W^b$};
			\node[above right] at (4.5,2) {\tiny $W^b$};
            \node[above right] at (7,2) {\tiny $W^b$};
			\node[above] at (7,3) {\tiny $\lambda_3-\lambda_4; a_3, b_3$};
			\node[above] at (4.5,3) {\tiny $\lambda_2-\lambda_3; a_2, b_2$};
			\node[above] at (2,3) {\tiny $\lambda_1-\lambda_2; a_1, b_1$};
			\node[left] at (1,2) {\tiny $i; a_1, x$};
			\node[below] at (7,1) {\tiny $\mu_3-\mu_4; a_{4}, b_3$};
			\node[below] at (4.5,1) {\tiny $\mu_2-\mu_3; a_{3}, b_2$};
			\node[below] at (2,1) {\tiny $\mu_1-\mu_2; a_{2}, b_1$};
			\node[right] at (10,2) {\tiny $0$};
		\end{tikzpicture}
        \end{gathered}
    $$
    The partitions $\lambda,\mu$ encode the top and bottom boundary configurations in a way identical to the row operators $T_i$ from Section \ref{sec: sHL-def}, but due to the conjugation $\lambda'$ we change $m_i(\lambda)$ to $m_i(\lambda')=\lambda_i-\lambda_{i+1}$. 
    
     Due to the conservation law, the partition function $\langle\lambda'|\ \T_i(x\mid\calA,\calB)|\mu'\rangle$ has a unique configuration with nonzero weight. This configuration exists only when $i=\lambda_1-\mu_1$, and the horizontal edge between columns $j$ and $j+1$ is forced to carry the value $\lambda_{j+1}-\mu_{j+1}$. In particular, $\lambda_{j+1}-\mu_{j+1}\geq 0$. Moreover, since the weight $W^b(i,j,k,l)$ vanishes unless $j\geq l$, the partition function is $0$ unless $\lambda_{j+1}-\mu_{j+1}\leq \mu_{j}-\mu_{j+1}$ for all $j$. So
    $$
    \langle\lambda'|\ \T_i(x\mid\calA,\calB)|\mu'\rangle=0 \qquad \text{unless}\ i=\lambda_1-\mu_1\ \text{and}\ \mu\prec\lambda.
    $$

    To get the spin $q$-Whittaker polynomials, define
    $$
    \B(x\mid\calA,\calB)=\sum_{r\geq 0}(x/b_0)^{r}\frac{(a_1x^{-1};q)_{r}}{(q;q)_{r}} \T_r(x\mid\calA,\calB).
    $$
    Then the vertex model definition of $\F_{\lambda/\mu}$ is equivalent to
    $$
    \F_{\lambda/\mu}(x_1,\dots,x_n\mid\calA,\calB)=\langle\lambda'|\ \B(x_1\mid\calA,\calB)\B(x_2\mid\tau\calA,\calB)\dots\B(x_n\mid\tau^{n-1}\calA,\calB)\ |\mu'\rangle.
    $$

    We can also define the spin $q$-Whittaker polynomials using the weights $W^a$. Let
    $\T^*_i(y\mid\calA,\calB)$ denote the matrix
    $$
    \langle\lambda'|\ \T_i^*(y\mid\calA,\calB)|\mu'\rangle=(b_0/y)^{\mu_1}\left(		
    \begin{gathered}
    \begin{tikzpicture}
			\draw[fused] (7.9,2) -- (1,2);
			\node[right] at (7.9,2) {$\dots$};
			\draw[fused*] (8.5,2) -- (9.3,2);
			\foreach\x in {1,...,3}{
				\draw[fused] (2.5*\x-0.5,1) -- (2.5*\x-0.5,3);
			}
			\node[above right] at (2,2) {\tiny $W^a$};
			\node[above right] at (4.5,2) {\tiny $W^a$};
            \node[above right] at (7,2) {\tiny $W^a$};
			\node[above] at (7,3) {\tiny $\lambda_3-\lambda_4; a_3, b_4$};
			\node[above] at (4.5,3) {\tiny $\lambda_2-\lambda_3; a_2, b_3$};
			\node[above] at (2,3) {\tiny $\lambda_1-\lambda_2; a_1, b_2$};
			\node[below] at (1,1.9) {\tiny $i; y, b_1$};
			\node[below] at (7,1) {\tiny $\mu_3-\mu_4; a_{3}, b_3$};
			\node[below] at (4.5,1) {\tiny $\mu_2-\mu_3; a_{2}, b_2$};
			\node[below] at (2,1) {\tiny $\mu_1-\mu_2; a_{1}, b_1$};
			\node[below] at (9.3,2) {\tiny $0$};
		\end{tikzpicture}
        \end{gathered}\right)
    $$
    and set
    $$
    \B^*(y\mid\calA,\calB)=\sum_{r\geq 0}\T^*_r(y\mid\calA,\calB).
    $$
    Using
    $$
    \begin{gathered}
			\begin{tikzpicture}[scale=1.3]
				\draw[fused] (1,0) -- (-1,0);
				\draw[fused] (0,-1) -- (0,1);
				\node[below] at (-0.9,-0.1) {\tiny \shortstack[c]{$k; y,b_1$}};
				\node[below] at (0.9,-0.1) {\tiny \shortstack[c]{$j; y,b_2$}};
				\node[below] at (0,-1) {\tiny $i; a,b_1$};
				\node[above] at (0,1) {\tiny $l; a,b_2$};
				\node[above right] at (0,0) {\tiny $W^a$};
			\end{tikzpicture}
		\end{gathered}
        = \frac{(y/b_1)^{l}}{(a/b_1)^{j}}\frac{(q;q)_i(q;q)_j(y/b_1;q)_k(a/b_2;q)_l}{(a/b_1;q)_i(y/b_2;q)_j(q;q)_k(q;q)_l}
        \begin{gathered}
			\begin{tikzpicture}[scale=1.3]
				\draw[fused] (-1,0) -- (1,0);
				\draw[fused] (0,-1) -- (0,1);
				\node[below] at (-0.9,-0.1) {\tiny \shortstack[c]{$k;b_1^{-1},y^{-1}$}};
				\node[below] at (0.9,-0.1) {\tiny \shortstack[c]{$j; b_2^{-1},y^{-1}$}};
				\node[below] at (0,-1) {\tiny $l; b_2^{-1},a^{-1}$};
				\node[above] at (0,1) {\tiny $i; b_1^{-1},a^{-1}$};
				\node[above right] at (0,0) {\tiny $W^b$};
			\end{tikzpicture}
		\end{gathered}
    $$
    we can show the following relation between $\B$ and $\B^*$:
    \begin{proposition}[{\cite[Proposition 4.4]{Kor24}}]\label{renorm}
        For every pair of partitions $\lambda,\mu$ we have
        $$
        \langle\mu'|\B^*(y\mid\calA,\calB)|\lambda'\rangle=\frac{\psi_{\lambda'}(\calA,\calB)}{\psi_{\mu'}(\calA,\tau\calB)}\langle\lambda'|\B(y^{-1}\mid\overline{\calB},\overline{\calA})|\mu'\rangle,
        $$
        where 
    	\begin{equation}\label{eq:def-psi-old}
	       \psi_{\lambda'}(\calA,\calB)=\psi_{\lambda'}(\overline{\calB},\overline{\calA})
	        =
	       \prod_{r\geq 1}\left(\frac{b_{r-1}}{a_{r-1}}\right)^{\lambda_{r}}
	       \frac{(q;q)_{\lambda_r-\lambda_{r+1}}}
	       {(a_r/b_r;q)_{\lambda_r-\lambda_{r+1}}}.
	\end{equation}
    \end{proposition}

    We define the dual spin $q$-Whittaker polynomials by
    $$
    \F^*_{\lambda/\mu}(y_1,\dots, y_m\mid\calA, \calB)=\langle\mu'|\B^*(y^{-1}_m\mid\overline{\calB},\tau^{m-1}\overline{\calA})\dots\B^*(y^{-1}_2\mid\overline{\calB},\tau\overline{\calA})\B^*(y^{-1}_1\mid\overline{\calB},\overline{\calA})|\lambda'\rangle.
    $$
    Then Proposition \ref{renorm} implies
	\begin{equation}\label{eq:def-dual-F}
	\F^*_{\lambda/\mu}(y_1,\ldots,y_m\mid\calA,\calB)
	=
	\frac{\psi_{\lambda'}(\calA,\calB)}
	{\psi_{\mu'}(\tau^m\calA,\calB)}
	\F_{\lambda/\mu}(y_1,\ldots,y_m\mid\calA,\calB).
	\end{equation}

	\subsection{Properties of $\F_{\lambda/\mu}$}\label{sec:basic-properties}

	Here we list several properties of our polynomials: the branching and symmetry properties proved in ~\cite{BK24}, and the Cauchy identity proved in~\cite{Kor24}.

    \begin{proposition}[{\cite[Theorem 5.9]{BK24}}]  The function $\F_{\lambda/\mu}(x_1,\dots,x_n\mid\calA,\calB)$ is a symmetric polynomial in $x_1,\dots, x_n$.
    \end{proposition}

    \begin{proposition}[{\cite[Proposition 5.5]{BK24}}]  For every pair of partitions $\lambda,\nu$ we have	the following branching rule:
    \begin{equation}\label{eq:branching-rule}
	\F_{\lambda/\nu}(x_1,x_2,\ldots,x_n\mid\calA,\calB)
	=
	\sum_{\mu\prec\lambda}
	\F_{\lambda/\mu}(x_1\mid\calA,\calB)\,
	\F_{\mu/\nu}(x_2,\ldots,x_n\mid\tau\calA,\calB),
	\end{equation}
    where the sum is over partitions $\mu$ interlacing $\lambda$.
    \end{proposition}
    
    Plugging in the values of the vertex weights, one can compute the one-variable skew function:
	\begin{equation}\label{eq:one-variable-F}
	\F_{\lambda/\mu}(x\mid \calA,\calB)=
	\begin{cases}
	\begin{aligned}
	&x^{|\lambda|-|\mu|}
	\prod_{r\geq 1} b_{r-1}^{\mu_{r}-\lambda_{r}}
	\\
	&\quad{}\times
	\dfrac{(a_r/x;q)_{\lambda_r-\mu_r}(x/b_r;q)_{\mu_r-\lambda_{r+1}}(q;q)_{\mu_r-\mu_{r+1}}}
	{(q;q)_{\lambda_r-\mu_r}(q;q)_{\mu_r-\lambda_{r+1}}(a_{r+1}/b_r;q)_{\mu_r-\mu_{r+1}}},
	\end{aligned}
	& \text{if }\mu\prec\lambda,\\[1.2em]
	0, & \text{otherwise.}
	\end{cases}
	\end{equation}
	Together with the branching rule, this determines the skew polynomials $\F_{\lambda/\mu}$ recursively. In particular, $\F_{\lambda/\mu}(x_1,\ldots,x_n\mid\calA,\calB)$ vanishes unless $\mu\subset\lambda$ and $l(\lambda)\leq l(\mu)+n$. Comparing \eqref{eq:one-variable-F} with \cite[Chapter VI, (6.24)(ii)]{Mac95}, we have
    $$
    \prod_{r\geq 1} b_{r-1}^{\lambda_{r}-\mu_{r}}\F_{\lambda/\mu}(x_1, \dots, x_n\mid \calA,\calB)\bigg|_{\substack{a_0=a_1=\dots=0\\b_0=b_1=\dots=\infty}}=Q_{\lambda/\mu}(x_1, \dots, x_n;q,0),
    $$
    where the right-hand side is the $t=0$ Macdonald polynomial $Q_{\lambda/\mu}(x_1,\dots, x_n;q,0)$, also known as the \emph{$q$-Whittaker polynomial}.

    \begin{theorem}[{\cite[Theorem 5.8]{Kor24}}]\label{thm:skew-cauhy} Let $\mu,\nu$ be a pair of partitions. Then, as an identity of formal power series in the variables $x_i,y_j$ and $q$, we have
	\begin{multline}\label{eq:skew-cauchy}
	\sum_{\lambda}
	\F_{\lambda/\mu}(x_1,\ldots,x_n\mid\calA,\calB)\,
	\F^*_{\lambda/\nu}(y_1,\ldots,y_m\mid\overline{\calB},\overline{\calA})
	\\
	=
	\prod_{i=1}^{n}\prod_{j=1}^{m}
	\frac{(a_i y_j;q)_\infty(x_i/b_j;q)_\infty}
	{(x_i y_j;q)_\infty(a_i/b_j;q)_\infty}
	\sum_{\rho}
	\F^*_{\mu/\rho}(y_1,\ldots,y_m\mid\overline{\calB},\tau^n\overline{\calA})\,
	\F_{\nu/\rho}(x_1,\ldots,x_n\mid\calA,\tau^m\calB).
	\end{multline}
    \end{theorem}
	
	In the non-skew case this gives
	\begin{equation}\label{eq:cauchy}
	\sum_{\lambda}
	\F_{\lambda}(x_1,\ldots,x_n\mid\calA,\calB)\,
	\F^*_{\lambda}(y_1,\ldots,y_m\mid\overline{\calB},\overline{\calA})
	=
	\prod_{i=1}^{n}\prod_{j=1}^{m}
	\frac{(a_i y_j;q)_\infty(x_i/b_j;q)_\infty}
	{(x_i y_j;q)_\infty(a_i/b_j;q)_\infty}.
	\end{equation}

\section{Littlewood identity for spin $q$-Whittaker functions}\label{sec:proof}

In this section we use the reflection equation from Theorem~\ref{reflection-theo} together with the Cauchy identity from Theorem \ref{thm:skew-cauhy} to find Littlewood identities for the inhomogeneous spin $q$-Whittaker polynomials. These identities can be interpreted as the dual versions of the identities from Section \ref{sec: sHLLit}. However, we do not know a way to rigorously justify this duality, so we give an independent proof.

\subsection{Statement and notable examples} We are going to use the Al-Salam-Chihara polynomials $h_n(u,v\mid a)$ and $h_\lambda(u,v\mid \calA)$ from Section \ref{sec: sHLLit}. Note that
$$
h_{\lambda'}(u,v\mid \calA)=\prod_{\substack{t\geq 1\\t-odd}}v^{\lambda_t-\lambda_{t+1}}h_{\lambda_t-\lambda_{t+1}}(u,1/v\mid a_{t})\prod_{\substack{t\geq 1\\t-even}}h_{\lambda_t-\lambda_{t+1}}(u,v\mid a_{t}).
$$
We also use the notation $\overline{\calA}^{\lambda'}=\prod_{r\geq 1}a_{r-1}^{-\lambda_{r}}$.

	\begin{theorem}[Skew Littlewood identity]\label{thm:littlewood}
		Let $\calA=(a_0,a_1,a_2,\dots)$ and $u,v$ be generic parameters. Then, for every partition $\mu$, we have 
		\begin{multline}\label{eq: littlewoodsqW}
			\sum_{\lambda} h_{\lambda'}(u,v\mid\calA)\overline\calA^{\lambda'}\,
			\F_{\lambda/\mu}(x_1, \dots, x_{n}\mid\calA,\overline{\calA})
			\\
            =\Psi_n(x_1,\dots, x_n; u,v\mid\calA)\sum_{\lambda} h_{\lambda'}(u,v\mid\tau^n\calA)(\tau^n\overline\calA)^{\lambda'}\,
			\F^*_{\mu/\lambda}(x_1, \dots, x_{n}\mid\calA,\tau^n\overline{\calA}),
		\end{multline}
		where both sides are formal power series in $x_1, \dots, x_n, q$ and
        $$
		\Psi_n(x_1,\dots x_n; u,v\mid \calA):=\prod_{i=1}^n\frac{(a_iu, a_iv;q)_\infty}{(a_i^2, x_iu, x_iv;q)_\infty}\cdot\frac{\prod_{i=1}^n\prod_{j=1}^n(x_ia_j;q)_\infty}{\prod_{1\leq i<j\leq n}(x_ix_j, a_ia_j;q)_\infty}.
		$$
        The identity also holds numerically when $q,u,v, x_1, \dots, x_n$ and $\calA=(a_0,a_1,a_2,\dots)$ are complex numbers satisfying
        \begin{equation}\label{eq: littlewoodsqWcond}
        \begin{gathered}
        |q|<1,\qquad |x_1|, \dots, |x_n|< \max(1, |u|, |v|)^{-1},\\
        a_ia_jq^m\neq 1\quad \text{for\ all\ } i,j\in\mathbb Z_{>0}, m\in\mathbb Z_{\geq0}.
        \end{gathered}
        \end{equation}
		\end{theorem}

        \begin{corollary}[Littlewood identity]\label{cor:littlewood}
		Let $\calA=(a_0,a_1,\dots)$ and $u,v$ be generic parameters. Then
		$$
			\sum_{\lambda} h_{\lambda'}(u,v\mid\calA)\overline\calA^{\lambda'}\,
			\F_{\lambda}(x_1, \dots, x_{n}\mid\calA,\overline{\calA})
			\\
            =\Psi_n(x_1,\dots x_n; u,v\mid \calA).
		$$
        where both sides are formal power series in $x_1, \dots, x_n, q$. The identity also holds numerically for complex numbers satisfying \eqref{eq: littlewoodsqWcond}.
	\end{corollary}

We can use the same specializations of $u,v$ as in Section \ref{sec: sHLLit} to get several simplifications of the identity. When  $u=0$ we have
\begin{equation}
		\sum_{\lambda} h_{\lambda'}(0,v\mid\calA) \overline{\calA}^{\lambda'}\,
			\F_{\lambda}(x_1, \dots, x_{n}\mid\calA,\overline{\calA})=\prod_{i=1}^n\frac{(a_iv;q)_\infty}{(a_i^2, x_iv;q)_\infty}\frac{\prod_{i=1}^n\prod_{j=1}^n(x_ia_j;q)_\infty}{\prod_{1\leq i<j\leq n}(x_ix_j, a_ia_j;q)_\infty},
\end{equation}
where
$$
h_{\lambda'}(0,v\mid\calA)=\prod_{\substack{t\geq 2\\t \text{\ even}}}\frac{(va_t;q)_{\lambda_t-\lambda_{t+1}}}{(a_t^2;q)_{\lambda_t-\lambda_{t+1}}}\prod_{\substack{t\geq 1\\t\text{\ odd}}}v^{\lambda_t-\lambda_{t+1}}\frac{(a_t/v;q)_{\lambda_t-\lambda_{t+1}}}{(a_t^2;q)_{\lambda_t-\lambda_{t+1}}}.
$$
Setting $v=1$ and $u=-1$ results in
\begin{multline*}
\sum_{\lambda:\,\lambda\text{ even}}
\prod_{t\ge1}\prod_{r=1}^{\frac{\lambda_t-\lambda_{t+1}}2}
\frac{1-q^{2r-1}}{1-a_t^2q^{2r-1}}\,
\overline{\mathcal A}^{\,\lambda'}\,
\F_{\lambda}
(x_1,\ldots,x_n\mid\mathcal A,\overline{\mathcal A})\\
=
\prod_{i=1}^n\frac{(a_i^2;q^2)_\infty}{(x_i^2;q^2)_\infty}\cdot\frac{\prod_{i,j}(x_ia_j;q)_\infty}{\prod_{i<j}(x_ix_j;q)_\infty \prod_{i\leq j} (a_ia_j;q)_\infty},
\end{multline*}
where in all products $i,j\in[1,n]$.

In the homogeneous case $a_0=a_1=a_2=a_3=\dots=a$ we can set $v=a$ to get
\begin{equation}
		\sum_{\lambda} u^{l((\lambda')_{odd})} a^{-|\lambda|}\,
			\F_{\lambda}(x_1, \dots, x_{n}\mid a,a^{-1})=\frac{(au;q)^n_\infty}{\prod_{i=1}^n(x_iu;q)_\infty}\frac{\prod_{i=1}^n(x_ia;q)^{n-1}_\infty}{\prod_{1\leq i<j\leq n}(x_ix_j,a^2;q)_\infty},
\end{equation}
\begin{equation}
		\sum_{\lambda:\lambda'\text{\ even}}a^{-|\lambda|}\,
			\F_{\lambda}(x_1, \dots, x_{n}\mid a,a^{-1})=\frac{\prod_{i=1}^n(x_ia;q)^{n-1}_\infty}{\prod_{1\leq i<j\leq n}(x_ix_j,a^2;q)_\infty}.
\end{equation}

Finally, when $a_0=a_1=a_2=\dots=a\to0$, the functions $\overline\calA^{\lambda'}\,
			\F_{\lambda}(x_1, \dots, x_{n}\mid\calA,\overline{\calA})$ degenerate to $Q_\lambda(x_1,\dots, x_n; q,0)$. So, we get
		$$
			\sum_{\lambda} h_{\lambda'}(u,v\mid 0) Q_{\lambda}(x_1, \dots, x_{n}; q,0)=\prod_{i=1}^n\frac{1}{(x_iu, x_iv;q)_\infty}\prod_{1\leq i<j\leq n}\frac{1}{(x_ix_j;q)_\infty},
		$$
        which is equivalent to the identities from \cite{War06} after applying $\omega_{q,0}$ \cite[Chapter VI]{Mac95}.

\subsection{Summation of the reflection weights}

To prove Theorem \ref{thm:littlewood} we need several properties of the reflection weights $\tK_{\xi,\zeta}$.

\begin{lemma}\label{Krat} For every $i,j\in\mathbb Z_{\geq0}$ the function
$$
(x/c)^{j/2}\zeta^{i+j}(c/x;q)_j(c^2;q)_{i+j}\tK_{c,x;\zeta,\xi}(i,j)
$$
is a polynomial in $\mathbb Z[x,c,q,\zeta,\xi^{\pm1}]$.
\end{lemma}
\begin{proof} Follows from the second line of \eqref{tKdef}, which implies
\begin{multline*}
(x/c)^{\frac j2}\zeta^{i+j}(c/x;q)_j(c^2;q)_{i+j}\tK_{c,x;\zeta,\xi}(i,j)\\
=\sum_{\substack{k,l,p,r\geq 0\\k+l+p+r=i+j}}\xi^{k-l}\zeta^{p+k}\frac{(c^2;q)_{i+j}}{(c^2;q)_{k+l}}\frac{(q;q)_{k+l}}{(q;q)_k(q;q)_l}\frac{(q;q)_j}{(q;q)_p(q;q)_r(q;q)_{j-p-r}}\\ \times\left(\frac{c\zeta}{\xi};q\right)_k\left[\zeta^l\left(\frac{c\xi}{\zeta};q\right)_l\right]\left[(\zeta x)^p\left(\frac{\xi}{x\zeta};q\right)_p\right]\left[x^r\left(\frac{\zeta}{\xi x};q\right)_r\right](cx;q)_{j-p-r}.
\end{multline*}
All expressions in square brackets are indeed polynomials in $x,q,\zeta$, even though they involve $x^{-1}$ and $\zeta^{-1}$. The $q$-binomial expressions are polynomials in $q$ due to the well-known expression
$$
\frac{(q;q)_{k+l}}{(q;q)_k(q;q)_l}=\sum_{1\leq i_1<i_2<\dots<i_k\leq k+l}q^{\sum_{s=1}^k(i_s-s)}.
$$
\end{proof}

\begin{proposition} \label{prop: Ksum} For every $l\in\mathbb Z_{\geq 0}$ we have
\begin{equation}\label{eq: Ksum}
        \sum_{p\geq 0}x^{p}\zeta^{p+l}\frac{(c/x;q)_p}{(q;q)_p} (x/c)^{l/2}(c/x;q)_l\tK_{c,x;\zeta,\xi}(p,l)\\
        = x^l(c/x;q)_l\frac{(cx,c\zeta/\xi,c\xi\zeta;q)_\infty}{(c^2,x\zeta/\xi,x\xi\zeta;q)_\infty}.
\end{equation}
where the identity holds in the space of formal power series in $x,q$ with rational coefficients in $\mathbb Q(c,\xi,\zeta)$. 
\end{proposition} 
\begin{proof} Apply Theorem \ref{reflection-theo} with $i=0$ to get
\begin{multline}\label{eq: proofKstep1}
    \sum_{p\geq0}\widetilde{W}^b_{a,b,c,a^{-1}}(0,j,j-p,p)\omega_{a;\xi,\zeta}(j-p)\widetilde{K}_{c,b;\zeta,\xi}(p,l)\\
    =\sum_{p\geq 0}\widetilde{K}_{a,b;\xi,\zeta}(0,p)\widetilde{W}^a_{c,a^{-1},b^{-1},c}(j,l,p,j+l-p)\omega_{c;\xi,\zeta}(j+l-p).
\end{multline}
Set $a=b=x$. Using the value for $\tW^b$, the left-hand side becomes
\begin{equation}
\label{eq: proofofabs}
    \sum_{p\geq0}x^{p}\frac{(c/x;q)_p(x^2;q)_{j-p}(q;q)_j}{(cx;q)_j(q;q)_p(q;q)_{j-p}}\omega_{x;\xi,\zeta}(j-p)\tK_{c,x;\zeta,\xi}(p,l).
\end{equation}
Plugging $a=b=x$ in the right-hand side requires some care. From \eqref{tKdef} we have
$$
\widetilde{K}_{a,b;\xi,\zeta}(0,p)=\sum_{s,t\geq 0}\omega_{b^{-1};\xi,\zeta}(s)\omega_{a;\zeta,\xi}(t)\tW^a_{a,b,b^{-1},a^{-1}}(0,p,s,t)=\omega_{a;\zeta,\xi}(p),
$$
where we use the relation $\tW^a_{a,b,b^{-1},a^{-1}}(0,p,s,t)=\delta_{t=p}$. At the same time,
$$
\widetilde{W}^a_{c,x^{-1},x^{-1},c}(j,l,p,j+l-p)=(xc)^{\frac{l}2}\frac{(1;q)_{p}(xc;q)_{j-p}}{(xc;q)_{j}}\frac{(q;q)_{j}}{(q;q)_{p}(q;q)_{j-p}}=(xc)^{\frac{l}2}\delta_{p=0}
$$
due to the $(1;q)_p$ factor forcing $p=0$. So, setting $a=b=x$ in the right-hand side of \eqref{eq: proofKstep1} yields
$$
(xc)^{\frac{l}2}\omega_{c;\xi,\zeta}(j+l).
$$
So, we get
\begin{equation}\label{eq: proofKstep2}
        \sum_{p\geq0}x^{p}\frac{(c/x;q)_p(x^2;q)_{j-p}(q;q)_j}{(cx;q)_j(q;q)_p(q;q)_{j-p}}\omega_{x;\xi,\zeta}(j-p)\tK_{c,x;\zeta,\xi}(p,l)=(xc)^{\frac{l}2}\omega_{c;\xi,\zeta}(j+l).
\end{equation}
From \eqref{omegaex} we have
$$
\omega_{a;\xi,\zeta}(n)=\zeta^{-n}\sum_{i\geq 0}\zeta^{2i}\frac{(a\xi/\zeta;q)_i(a\zeta/\xi;q)_{n-i}}{(a^2;q)_{n}}\frac{(q;q)_{n}}{(q;q)_i(q;q)_{n-i}}.
$$
Plugging this expression into \eqref{eq: proofKstep2} and multiplying both sides by $(x/c)^{l/2}\zeta^{j+l}(c/x;q)_l$ yields the following identity of rational functions for every $j\in\mathbb Z_{\geq 0}$: 
\begin{multline}\label{eq: proofKstep3}
        \sum_{p,i\geq0}x^{p}\zeta^{2i+p+l}\frac{(c/x;q)_p(q;q)_j(x\xi/\zeta;q)_i(x\zeta/\xi;q)_{j-p-i}}{(cx;q)_j(q;q)_p(q;q)_i(q;q)_{j-p-i}}(x/c)^{l/2}(c/x;q)_l\tK_{c,x;\zeta,\xi}(p,l)\\
        =x^l(c/x;q)_l\sum_{i\geq 0}\zeta^{2i}\frac{(c\xi/\zeta;q)_i(c\zeta/\xi;q)_{j+l-i}(q;q)_{j+l}}{(c^2;q)_{j+l}(q;q)_i(q;q)_{j+l-i}}.
\end{multline}

Up until now we have worked with finite rational expressions: all the summations above have only finitely many nonzero terms. To finish the proof we need to take the limit $j\to\infty$, and for this step we need to work with formal power series in $q,x,\zeta$ over the field of rational functions $\mathbb Q(c,\xi)$. For $f\in \mathbb Q(c,\xi)[[x,q,\zeta]]$ let $\deg f$ denote the minimal total degree with respect to $x,q$ and $\zeta$. Let 
$$
\mathcal P_N=\{f\in \mathbb Q(c,\xi)[[x,q,\zeta]]: \deg(f)\geq N\}
$$ 
denote the degree filtration of $\mathbb Q(c,\xi)[[x,q,\zeta]]$. We also set $\mathcal P=\mathcal P_0=\mathbb Q(c,\xi)[[x,q,\zeta]]$. Note that both sides of \eqref{eq: proofKstep3} can be treated as elements of $\mathcal P$. More precisely:
\begin{itemize}
\item $x^p(c/x;q)_p\in\mathcal P_{p-1}$;
\item $\zeta^{2i}(x\xi/\zeta;q)_i, \zeta^{2i}(c\xi/\zeta;q)_i\in \mathcal P_{i}$;
\item $(x/c)^{l/2}\zeta^{p+l}(c/x;q)_l\tK_{c,x;\zeta,\xi}(p,l)\in\mathcal P$ by Lemma \ref{Krat};
\item all other $q$-Pochhammer symbols appearing in \eqref{eq: proofKstep3} are clearly in $\mathcal P$. 
\end{itemize}

Now, treating both sides of \eqref{eq: proofKstep3} as elements of $\mathcal P$, we take $j\to\infty$. To do it, fix $N\in\mathbb Z_{\geq 0}$ and consider the component of total degree $\leq N$. On the left-hand side each summand has degree at least $p+i-1$, while the summands on the right-hand side have degree at least $i$. Hence, modulo $\mathcal P_{N+1}$, we have
\begin{multline}\label{eq: proofKstep4}
        \sum_{p,i:p+i\leq N+1}x^{p}\zeta^{2i+p+l}\frac{(c/x;q)_p(q;q)_j(x\xi/\zeta;q)_i(x\zeta/\xi;q)_{j-p-i}}{(cx;q)_j(q;q)_p(q;q)_i(q;q)_{j-p-i}}(x/c)^{l/2}(c/x;q)_l\tK_{c,x;\zeta,\xi}(p,l)\\
        \equiv x^l(c/x;q)_l\sum_{i= 0}^{N}\zeta^{2i}\frac{(c\xi/\zeta;q)_i(c\zeta/\xi;q)_{j+l-i}(q;q)_{j+l}}{(c^2;q)_{j+l}(q;q)_i(q;q)_{j+l-i}}\qquad (\mathrm{mod}\ \mathcal P_{N+1}).
\end{multline}
Note that on both sides $j$ enters the expression only as an index of $q$-Pochhammer symbols, and due to the restrictions on $p,i$ all these indices are at least $j-N-1$. Hence, for $j>2N+1$ these symbols $(\bullet;q)_{j-m}$ do not depend on $j$ modulo $\mathcal P_{N+1}$ and we can replace them by $(\bullet;q)_{\infty}$. So,
\begin{multline*}
        \sum_{p,i:p+i\leq N+1}x^{p}\zeta^{2i+p+l}\frac{(c/x;q)_p(x\xi/\zeta;q)_i(x\zeta/\xi;q)_{\infty}}{(cx;q)_\infty(q;q)_p(q;q)_i}(x/c)^{l/2}(c/x;q)_l\tK_{c,x;\zeta,\xi}(p,l)\\
        \equiv x^l(c/x;q)_l\sum_{i= 0}^{N}\zeta^{2i}\frac{(c\xi/\zeta;q)_i(c\zeta/\xi;q)_{\infty}}{(c^2;q)_{\infty}(q;q)_i}\qquad (\mathrm{mod}\ \mathcal P_{N+1}).
\end{multline*}
Removing the restrictions on the sums, rearranging the terms and sending $N\to\infty$ yields
\begin{multline*}
        \left(\sum_{p\geq 0}x^{p}\zeta^{p+l}\frac{(c/x;q)_p}{(q;q)_p}(x/c)^{l/2}(c/x;q)_l\tK_{c,x;\zeta,\xi}(p,l)\right)\left(\sum_{i\geq 0}\zeta^{2i}\frac{(x\xi/\zeta;q)_i}{(q;q)_i}\right)\\
        = x^l(c/x;q)_l\frac{(cx;q)_\infty(c\zeta/\xi;q)_{\infty}}{(c^2;q)_{\infty}(x\zeta/\xi;q)_{\infty}}\sum_{i\geq 0}\zeta^{2i}\frac{(c\xi/\zeta;q)_i}{(q;q)_i}.
\end{multline*}
Simplifying both sums over $i$ using the $q$-binomial theorem, we get 
\begin{multline*}
        \sum_{p\geq 0}x^{p}\zeta^{p+l}\frac{(c/x;q)_p}{(q;q)_p}(x/c)^{l/2}(c/x;q)_l\tK_{c,x;\zeta,\xi}(p,l)\frac{(x\xi\zeta;q)_\infty}{(\zeta^2;q)_\infty}\\
        = x^l(c/x;q)_l\frac{(cx;q)_\infty(c\zeta/\xi;q)_{\infty}(c\xi\zeta;q)_\infty}{(c^2;q)_{\infty}(x\zeta/\xi;q)_{\infty}(\zeta^2;q)_\infty}.
\end{multline*}
This is equivalent to \eqref{eq: Ksum}. Note that, while we have worked over $\mathbb Q(c,\xi)[[x,q,\zeta]]$, both sides of \eqref{eq: Ksum} are power series in $x,q$ with rational coefficients in $\mathbb Q(c,\xi,\zeta)$. Since every rational function in $\zeta$ can be uniquely represented by a Laurent series in $\zeta$, \eqref{eq: Ksum} holds in  $\mathbb Q(c,\xi,\zeta)[[x,q]]$ as well.
\end{proof}

\subsection{Reflection of row operators} From this point the proof is similar to the proof of Theorem \ref{thm:sHLlittlewood}.  First, we focus on the row operators from Section \ref{sec: partition}.

Let $\sqrt{u}$, $\sqrt{v}$, $\sqrt{x}$ and $\sqrt{a_0},\sqrt{a_1},\dots$ be generic parameters in $\calF$. Define the following element of $V^*$:
\begin{equation}\label{eq: dconjOmegadef}
\langle\Omega'_{u,v}(\calA)|=\sum_{\lambda}h_{\lambda'}(u,v\mid\calA)\overline{\calA}^{\lambda'} \langle\lambda'|.
\end{equation}
Setting $\xi=\sqrt{u}/\sqrt{v}$, $\zeta=\sqrt{u}\sqrt{v}$ and applying \eqref{eq: omegatoh} we get
\begin{equation}\label{eq: dconjOmegadefalt}
\langle\Omega'_{u,v}(\calA)|=\sum_{\lambda}\overline{\calA}^{\lambda'}\zeta^{\lambda_1}\prod_{\substack{t\geq1\\ t \text{\ even}}}\omega_{a_{t};\xi,\zeta}(\lambda_t-\lambda_{t+1})\prod_{\substack{t\geq1\\ t \text{\ odd}}} \omega_{a_{t};\zeta,\xi}(\lambda_t-\lambda_{t+1})\ \langle\lambda'|.
\end{equation}

\begin{lemma}\label{lem: row flip} For every partition $\mu$ and $i\in\mathbb Z_{\geq 0}$ we have
\begin{multline}\label{eq:row flip}
\langle\Omega'_{u,v}(\calA)|\ \T_i(x\mid\calA,\overline{\calA})\ |\mu'\rangle \\
= \sum_{l\geq 0} a_0^{-i}(xa_1)^{-l/2}\zeta^{i+l}\tK_{a_1,x;\zeta,\xi}(i,l)\langle\Omega_{u,v}'(\tau\calA)|\, \T^*_l(x^{-1}\mid\tau\calA,\overline{\calA})\, |\mu'\rangle.
\end{multline}
\end{lemma}
\begin{proof} Fix the partition $\mu$, and let $N$ be an even integer such that $N>l(\mu)+1$. Since $W^b(i,j,k,l)=0$ unless $j\geq l$, in the row partition function $\langle\lambda'|\T_i(x\mid\calA,\calB)|\mu'\rangle$ all horizontal edges after column $l(\mu)+1$ are forced to be empty. In particular, regardless of $\lambda$, it is enough to use $N$ columns to compute this partition function. Moreover, due to the conservation law and the argument above,
$$
\langle\lambda'|\ \T_i(x\mid\calA,\calB)\ |\mu'\rangle=0 \qquad \text{unless}\ \lambda_1=\mu_1+i\ \text{and}\  l(\lambda)\leq l(\mu)+1,
$$
$$
\langle\lambda'|\ \T_l^*(y\mid\calA,\calB)\ |\mu'\rangle=0 \qquad \text{unless}\ \lambda_1+l=\mu_1\ \text{and}\ l(\lambda)\leq l(\mu).
$$
So, both sides of \eqref{eq:row flip} involve a finite computation and are well-defined.

From Proposition \ref{rescale} we have
\begin{multline*}
\langle\lambda'|\ \T_i(x\mid\calA,\overline{\calA})\ |\mu'\rangle=(x/a_1)^{\frac{\mu_1}2}\prod_{t\geq1}(a_ta_{t+1})^{-\frac{\mu_{t+1}}2}a_{t}^{\lambda_{t+1}}\\ \times   
\begin{gathered}
    \begin{tikzpicture}
			\draw[fused*] (1,2) -- (8.1,2);
			\node[right] at (8.1,2) {$\dots$};
			\draw[fused] (8.7,2) -- (10,2);
			\foreach\x in {1,...,3}{
				\draw[fused] (2.5*\x-0.5,1) -- (2.5*\x-0.5,3);
			}
			\node[above right] at (2,2) {\tiny $\tW^b$};
			\node[above right] at (4.5,2) {\tiny $\tW^b$};
            \node[above right] at (7,2) {\tiny $\tW^b$};
			\node[above] at (7,3) {\tiny $\lambda_3-\lambda_4; a_3, a^{-1}_3$};
			\node[above] at (4.5,3) {\tiny $\lambda_2-\lambda_3; a_2, a^{-1}_2$};
			\node[above] at (2,3) {\tiny $\lambda_1-\lambda_2; a_1, a^{-1}_1$};
			\node[left] at (1,2) {\tiny $i; a_1, x$};
			\node[below] at (7,1) {\tiny $\mu_3-\mu_4; a_{4}, a^{-1}_3$};
			\node[below] at (4.5,1) {\tiny $\mu_2-\mu_3; a_{3}, a^{-1}_2$};
			\node[below] at (2,1) {\tiny $\mu_1-\mu_2; a_{2}, a^{-1}_1$};
			\node[right] at (10,2) {\tiny $0;a_{N+1},x$};
		\end{tikzpicture}
        \end{gathered}.
\end{multline*}
Then \eqref{eq: dconjOmegadefalt} yields
\begin{multline*}
\langle\Omega'_{u,v}(\calA)|\ \T_i(x\mid\calA,\overline{\calA})\ |\mu'\rangle=(\zeta/a_0)^{\mu_1+i}(x/a_1)^{\frac{\mu_1}2}\prod_{t\geq1}(a_ta_{t+1})^{-\frac{\mu_{t+1}}2}\\ \times   
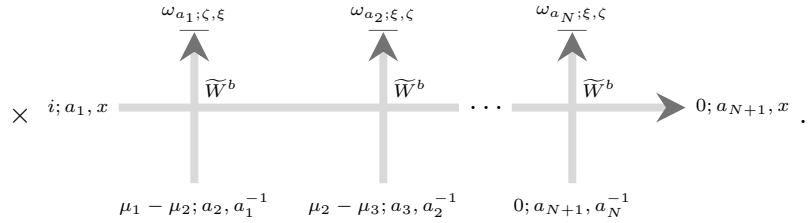
\begin{gathered}
    \begin{tikzpicture}
			\draw[fused*] (1,2) -- (5.5,2);
			\node[right] at (5.5,2) {$\dots$};
			\draw[fused] (6.2,2) -- (8.5,2);
			\foreach\x in {1,...,3}{
				\draw[fused] (2.5*\x-0.5,1) -- (2.5*\x-0.5,3);
			}
			\node[above right] at (2,2) {\tiny $\tW^b$};
			\node[above right] at (4.5,2) {\tiny $\tW^b$};
			\node[above right] at (7,2) {\tiny $\tW^b$};
            \node[above] at (7,3) {\tiny $\omega_{a_N;\xi,\zeta}$};
            \node at (7,3)  {$\text{---}$};
			\node[above] at (4.5,3) {\tiny $\omega_{a_2;\xi,\zeta}$};
            \node at (4.5,3)  {$\text{---}$};
			\node[above] at (2,3) {\tiny $\omega_{a_1;\zeta,\xi}$};
            \node at (2,3)  {$\text{---}$};
			\node[left] at (1,2) {\tiny $i; a_1, x$};
			\node[below] at (7,1) {\tiny $0; a_{N+1}, a^{-1}_N$};
			\node[below] at (4.5,1) {\tiny $\mu_2-\mu_3; a_{3}, a^{-1}_2$};
			\node[below] at (2,1) {\tiny $\mu_1-\mu_2; a_{2}, a^{-1}_1$};
			\node[right] at (8.5,2) {\tiny $0;a_{N+1},x$};
		\end{tikzpicture}
        \end{gathered}.
\end{multline*}
Attach a reflection weight $\tK_{\zeta,\xi}$ at the right end of the partition function above:
$$
    
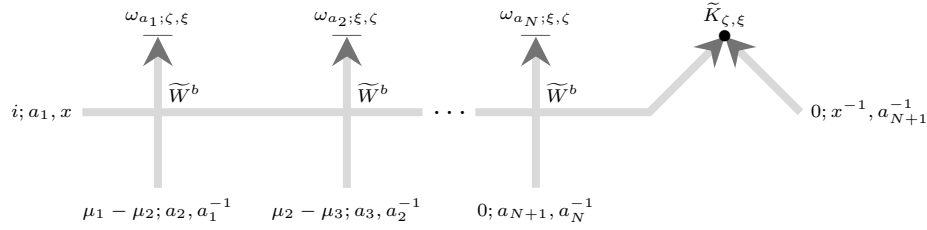
\begin{tikzpicture}
			\draw[fused*] (1,2) -- (5.5,2);
			\node[right] at (5.5,2) {$\dots$};
			\draw[fused] (6.2,2) -- (8.5,2) -- (9.5,3);
            \draw[fused] (10.5,2) -- (9.5,3);
			\foreach\x in {1,...,3}{
				\draw[fused] (2.5*\x-0.5,1) -- (2.5*\x-0.5,3);
			}
            
			\node[above right] at (2,2) {\tiny $\tW^b$};
			\node[above right] at (4.5,2) {\tiny $\tW^b$};
			\node[above right] at (7,2) {\tiny $\tW^b$};
            \node[above] at (7,3) {\tiny $\omega_{a_N;\xi,\zeta}$};
            \node at (7,3)  {$\text{---}$};
			\node[above] at (4.5,3) {\tiny $\omega_{a_2;\xi,\zeta}$};
            \node at (4.5,3)  {$\text{---}$};
			\node[above] at (2,3) {\tiny $\omega_{a_1;\zeta,\xi}$};
            \node at (2,3)  {$\text{---}$};
            \node[above] at (9.5,3) {\tiny $\tK_{\zeta,\xi}$};
            \node at (9.5,3) {$\bullet$};
			\node[left] at (1,2) {\tiny $i; a_1, x$};
			\node[below] at (7,1) {\tiny $0; a_{N+1}, a^{-1}_{N}$};
			\node[below] at (4.5,1) {\tiny $\mu_2-\mu_3; a_{3}, a^{-1}_2$};
			\node[below] at (2,1) {\tiny $\mu_1-\mu_2; a_{2}, a^{-1}_1$};
			\node[right] at (10.5,2) {\tiny $0;x^{-1},a^{-1}_{N+1}$};
		\end{tikzpicture}
$$
Note that adding this reflection does not change the partition function. Indeed, since $\tW^b(i,j,k,l)=0$ unless $j\geq l$, the edge entering $\tK_{\zeta,\xi}$ from the left is forced to be empty, so the reflection weight is $1$.

Next we repeatedly use Theorem \ref{reflection-theo} to move the reflection point to the left end of the row. This results in
\begin{multline*}
\langle\Omega'_{u,v}(\calA)|\ \T_i(x\mid\calA,\overline{\calA})\ |\mu'\rangle=(\zeta/a_0)^{\mu_1+i}(x/a_1)^{\frac{\mu_1}2}\prod_{t\geq1}(a_ta_{t+1})^{-\frac{\mu_{t+1}}2}\\ \times   
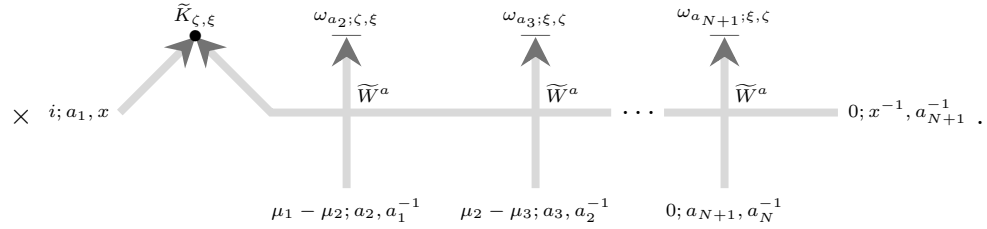
\begin{gathered}
    \begin{tikzpicture}
			\draw[fused] (-1,2) -- (0,3);
            \draw[fused] (5.5,2) -- (1,2) -- (0,3); 
			\node[right] at (5.5,2) {$\dots$};
			\draw[fused*] (6.2,2) -- (8.5,2);
			\foreach\x in {1,...,3}{
				\draw[fused] (2.5*\x-0.5,1) -- (2.5*\x-0.5,3);
			}
			\node[above right] at (2,2) {\tiny $\tW^a$};
			\node[above right] at (4.5,2) {\tiny $\tW^a$};
			\node[above right] at (7,2) {\tiny $\tW^a$};
            \node[above] at (0,3) {\tiny $\tK_{\zeta,\xi}$};
            \node at (0,3) {$\bullet$};
            \node[above] at (7,3) {\tiny $\omega_{a_{N+1};\xi,\zeta}$};
            \node at (7,3)  {$\text{---}$};
			\node[above] at (4.5,3) {\tiny $\omega_{a_3;\xi,\zeta}$};
            \node at (4.5,3)  {$\text{---}$};
			\node[above] at (2,3) {\tiny $\omega_{a_2;\zeta,\xi}$};
            \node at (2,3)  {$\text{---}$};
			\node[left] at (-1,2) {\tiny $i; a_1, x$};
			\node[below] at (7,1) {\tiny $0; a_{N+1}, a^{-1}_N$};
			\node[below] at (4.5,1) {\tiny $\mu_2-\mu_3; a_{3}, a^{-1}_2$};
			\node[below] at (2,1) {\tiny $\mu_1-\mu_2; a_{2}, a^{-1}_1$};
			\node[right] at (8.5,2) {\tiny $0;x^{-1},a_{N+1}^{-1}$};
		\end{tikzpicture}
        \end{gathered}.
\end{multline*}
Applying Proposition \ref{rescale} and the definition of $\T^*$ yields \eqref{eq:row flip}.
\end{proof}

\begin{lemma}\label{lem: Brow flip} For every partition $\mu$ we have
$$
\langle\Omega'_{u,v}(\calA)|\ \B(x\mid\calA,\overline{\calA})\ |\mu'\rangle=\frac{(a_1x;q)_\infty(a_1u;q)_\infty(a_1v;q)_\infty}{(a^2_1;q)_\infty(xu;q)_\infty(xv;q)_\infty} \langle\Omega'_{u,v}(\tau\calA)|\ \B^*(x^{-1}\mid\tau\calA,\overline{\calA})\ |\mu'\rangle,
$$
where both sides are power series in $x,q$ with coefficients in $\mathbb Q(u,v, a_0,a_1,\dots)$.
\end{lemma}
\begin{proof}
First let us explain why both sides are well-defined. Note that the vertex weights $W^b_{a_t, x, a_{t+1},a_t^{-1}}$ are polynomials in $x$ and rational functions in $q$, regular at $q=0$. Since $\langle\Omega'_{u,v}(\calA)|\ \T_i(x\mid\calA,\overline{\calA})\ |\mu'\rangle$ is a finite sum of products of such weights, it is clearly a well-defined series in $q,x$. At the same time $(xa_0)^i(a_1/x;q)_i$ is a polynomial in $q,x$ with the total degree of each monomial at least $i-1$, so
$$
\langle\Omega'_{u,v}(\calA)|\ \B(x\mid\calA,\overline{\calA})\ |\mu'\rangle=\sum_{i\geq 0}(xa_0)^i\frac{(a_1/x;q)_i}{(q;q)_i}\langle\Omega'_{u,v}(\calA)|\ \T_i(x\mid\calA,\overline{\calA})\ |\mu'\rangle
$$
is well-defined. 

For the right-hand side, using the values of $W^a$ weights yields
\begin{multline}\label{eq: computeTd}
\langle\lambda'|\  \T^*_{\mu_1-\lambda_1}(x^{-1}\mid\tau\calA,\overline{\calA})\ |\mu'\rangle\\
=(x/a_0)^{\mu_1}\prod_{t\geq1}
(x/a_t)^{\mu_{t+1}-\lambda_t}\frac{(a_t/x;q)_{\mu_t-\lambda_t}(xa_{t+1};q)_{\lambda_t-\mu_{t+1}}(q;q)_{\mu_t-\mu_{t+1}}}{(a_{t}a_{t+1};q)_{\mu_t-\mu_{t+1}}(q;q)_{\mu_t-\lambda_t}(q;q)_{\lambda_t-\mu_{t+1}}}\\
=a_0^{-\mu_1}\frac{(a_1/x;q)_{l}}{(q;q)_l}\prod_{t\geq1}x^{\mu_{t}-\lambda_{t}}a_t^{\lambda_t-\mu_{t+1}}\frac{(a_{t+1}/x;q)_{\mu_{t+1}-\lambda_{t+1}}(xa_{t+1};q)_{\lambda_t-\mu_{t+1}}(q;q)_{\mu_t-\mu_{t+1}}}{(a_{t}a_{t+1};q)_{\mu_t-\mu_{t+1}}(q;q)_{\mu_{t+1}-\lambda_{t+1}}(q;q)_{\lambda_t-\mu_{t+1}}}.
\end{multline}
For all other values of $l$ the matrix coefficient $\langle\lambda'|\  \T^*_{l}(x^{-1}\mid\tau\calA,\overline{\calA})\ |\mu'\rangle$ is $0$ due to the conservation law. In particular, the coefficient vanishes unless $l\leq\mu_1$, so 
$$
\langle\Omega'_{u,v}(\tau\calA)|\ \B^*(x^{-1}\mid\tau\calA,\overline\calA)\ |\mu'\rangle
$$
consists of a finite sum of such coefficients. This implies that the right-hand side is also well-defined. 

To prove the claim, use Lemma \ref{lem: row flip} to get
\begin{multline*}
\langle\Omega'_{u,v}(\calA)|\ \B(x\mid\calA,\overline{\calA})\ |\mu'\rangle=\sum_{i\geq0}(xa_0)^i\frac{(a_1/x;q)_i}{(q;q)_i}\langle\Omega'_{u,v}(\calA)|\ \T_i(x\mid\calA,\overline{\calA})\ |\mu'\rangle\\
=\sum_{i\geq0}\sum_{l=0}^{\mu_1}\left[x^i\zeta^{i+l}\frac{(a_1/x;q)_i}{(q;q)_i}(x/a_1)^{l/2}(a_1/x;q)_l\tK_{a_1,x;\zeta,\xi}(i,l)\right]\\
\times\left[\frac{1}{x^l(a_1/x;q)_l}\langle\Omega'_{u,v}(\tau\calA)|\ \T_l^*(x^{-1}\mid\tau\calA,\overline{\calA})\ |\mu'\rangle\right].
\end{multline*}
Note that the expressions in both sets of square brackets are well-defined formal power series in $x,q$, due to Lemma \ref{Krat} and \eqref{eq: computeTd}. Proposition \ref{prop: Ksum} then implies
$$
\langle\Omega'_{u,v}(\calA)|\ \B(x\mid\calA,\overline{\calA})\ |\mu'\rangle=\frac{(a_1x,a_1u,a_1v;q)_\infty}{(a^2_1,xu,xv;q)_\infty}  \sum_{l=0}^{\mu_1}\langle\Omega'_{u,v}(\tau\calA)| \T_l^*(x^{-1}\mid\tau\calA,\overline{\calA})\ |\mu'\rangle.\qedhere
$$
\end{proof}

\subsection{Proof of Theorem \ref{thm:littlewood}} The proof is by induction on $n$. We start with the base case $n=1$.

From the definitions of $\F_{\lambda/\mu}$ and $\F^*_{\lambda/\mu}$, Lemma \ref{lem: Brow flip} is equivalent to
\begin{multline*}
\sum_{\lambda} \langle\Omega'_{u,v}(\calA)\mid\lambda'\rangle\ \F_{\lambda/\mu}(x\mid\calA,\overline\calA)\\
=\frac{(a_1x;q)_\infty(a_1u;q)_\infty(a_1v;q)_\infty}{(a^2_1;q)_\infty(xu;q)_\infty(xv;q)_\infty} \sum_{\lambda} \langle\Omega'_{u,v}(\tau\calA)\mid\lambda'\rangle\ \F^*_{\mu/\lambda}(x\mid\calA,\tau\overline\calA).
\end{multline*}
This is exactly the skew Littlewood identity with $n=1$.

For the induction step, assume that Theorem \ref{thm:littlewood} holds for $n$. Using the branching rule, we have
\begin{multline*}
\sum_{\lambda}  h_{\lambda'}(u,v\mid\calA)\overline{\calA}^{\lambda'} \F_{\lambda/\mu}(x_1, \dots, x_{n+1}\mid\calA,\overline\calA)\\
=\sum_{\lambda,\nu} h_{\lambda'}(u,v\mid\calA)  \overline{\calA}^{\lambda'}\F_{\lambda/\nu}(x_1, \dots, x_{n}\mid\calA,\overline\calA)\F_{\nu/\mu}(x_{n+1}\mid\tau^n\calA,\overline\calA).
\end{multline*}
Applying the inductive assumption to the expression above yields 
\begin{multline*}
\Psi_n(x_1,\dots, x_n;u,v\mid\calA)\\\times\sum_{\lambda,\nu} h_{\lambda'}(u,v\mid\tau^n\calA)  (\tau^n\overline{\calA})^{\lambda'}\F^*_{\nu/\lambda}(x_1, \dots, x_{n}\mid\calA,\tau^n\overline\calA)\F_{\nu/\mu}(x_{n+1}\mid\tau^n\calA,\overline\calA).
\end{multline*}
Now we can use the skew Cauchy identity from Theorem \ref{thm:skew-cauhy} to the sum with respect to $\nu$, getting
\begin{multline*}
\Psi_n(x_1,\dots, x_n;u,v\mid\calA)\prod_{i=1}^n\frac{(x_ia_{n+1};q)_\infty(a_ix_{n+1};q)_\infty}{(x_ix_{n+1};q)_\infty(a_ia_{n+1};q)_\infty}\\\times\sum_{\lambda,\nu} h_{\lambda'}(u,v\mid\tau^n\calA)(\tau^n\overline{\calA})^{\lambda'} \F^*_{\mu/\nu}(x_1, \dots, x_{n}\mid\calA,\tau^{n+1}\overline\calA)\F_{\lambda/\nu}(x_{n+1}\mid\tau^n\calA,\tau^n\overline\calA).
\end{multline*}
Applying the single variable case of Theorem \ref{thm:littlewood}, we obtain
\begin{multline*}
\Psi_n(x_1,\dots, x_n;u,v\mid\calA)\frac{(a_{n+1}x_{n+1},a_{n+1}u,a_{n+1}v;q)_\infty}{(a^2_{n+1},x_{n+1}u,x_{n+1}v;q)_\infty}\prod_{i=1}^n\frac{(x_ia_{n+1};q)_\infty(a_ix_{n+1};q)_\infty}{(x_ix_{n+1};q)_\infty(a_ia_{n+1};q)_\infty}\\\times\sum_{\lambda,\nu} h_{\lambda'}(u,v\mid\tau^{n+1}\calA) (\tau^{n+1}\overline{\calA})^{\lambda'}\F^*_{\mu/\nu}(x_1, \dots, x_{n}\mid\calA,\tau^{n+1}\overline\calA)\F^*_{\nu/\lambda}(x_{n+1}\mid\tau^n\calA,\tau^{n+1}\overline\calA).
\end{multline*}
The inductive step now follows from the branching rule and the relation
\begin{multline*}
\Psi_n(x_1,\dots, x_n;u,v\mid\calA)\frac{(a_{n+1}x_{n+1},a_{n+1}u,a_{n+1}v;q)_\infty}{(a^2_{n+1},x_{n+1}u,x_{n+1}v;q)_\infty}\prod_{i=1}^n\frac{(x_ia_{n+1};q)_\infty(a_ix_{n+1};q)_\infty}{(x_ix_{n+1};q)_\infty(a_ia_{n+1};q)_\infty}\\=\Psi_{n+1}(x_1,\dots, x_{n+1};u,v\mid\calA). \qedhere
\end{multline*}
\qed

\subsection{Proof of the numerical version} We finish this section by proving that Theorem \ref{thm:littlewood} holds numerically as well. We need the following upper bounds.

\begin{lemma}\label{lem: bound1} Let $u,v,a\in\mathbb C$ and $\delta\in(0,1)$. There exists a constant $C$  such that for all $n\in\mathbb Z_{\geq 0}$ and $q\in\{z\in\mathbb C: |z|\leq \delta\}$ we have
$$
\Big|(a^2;q)_\infty\ h_n(u,v\mid a)\Big|\leq (n+1) C \max(|uv|, 1)^n.  
$$
\end{lemma}
\begin{proof} Note that for all $q\in\{z\in\mathbb C: |z|\leq \delta\}$ and $n\in\mathbb Z_{\geq 1}$ we have
\begin{equation}
\label{eq: bound1qw}
0<\prod_{i\geq 1} (1-\delta^i)\leq|(q;q)_n|\leq \prod_{i\geq 1} (1+\delta^i)<\infty
\end{equation}
and for any fixed $x\in\mathbb C$ we additionally have
\begin{equation}
\label{eq: bound2qw}
|(x;q)_n|\leq \prod_{i\geq 0} (1+|x|\delta^i)<\infty.
\end{equation}
Combining these inequalities with \eqref{eq: hdef} yields
$$
|(a^2;q)_nh_n(u,v\mid a)|< (n+1)C_1 \max(uv,1)^n
$$
for a constant $C_1$ which does not depend on $n\in\mathbb Z_{\geq0}$ and $q: |q|<\delta$. Then
$$
|(a^2;q)_\infty\ h_n(u,v\mid a)|<|(q^na^2;a)_\infty|(n+1)C_1 \max(uv,1)^n.
$$
Since $|(q^na^2;q)_\infty|<\prod_{i\geq 1}(1+|a|^2\delta^{i-1})<\infty$, the claim follows.
\end{proof}

\begin{lemma}\label{lem: bound2} Fix $k\in\mathbb Z_{\geq 0}$, $n\in\mathbb Z_{>0}$, $\delta\in(0,1)$, two parameters $\varepsilon, \varepsilon'$ satisfying $0<\varepsilon'<\varepsilon$ and two families of complex parameters $\calA=(a_0,a_1,a_2,\dots)$, $\calB=(b_0,b_1,b_2,\dots)$. Then there is a constant $C$ such that for all $x_1, \dots, x_n,q\in\mathbb C$ such that $|x_i|<\varepsilon'$, $|q|<\delta$ and every pair of partitions $\lambda,\mu$ such that $l(\mu)\leq k$ we have
$$
\left|\calB^{\lambda'-\mu'}\F_{\lambda/\mu}(x_1,\dots, x_n\mid\calA,\calB)\prod_{i=1}^n\prod_{j=1}^{k+n-i}(a_{j+i}/b_j;q)_\infty\right|<C(\lambda_1+1)^{(n-1)(k+\frac{n}{2})}\varepsilon^{|\lambda|-|\mu|}.
$$  
\end{lemma}
\begin{proof} We use induction on $n$. When $n=1$ we can use the expression \eqref{eq:one-variable-F} for the one-variable function $\F_{\lambda/\mu}(x\mid\calA,\calB)$. Since $l(\mu)\leq k$ we can limit the product over $r\geq 1$ in  \eqref{eq:one-variable-F} to the product over $r\in[1;k+1]$.  Using \eqref{eq: bound1qw} and \eqref{eq: bound2qw} we can bound the contributions of almost all $q$-Pochhammer symbols in \eqref{eq:one-variable-F}, except $(a_r/x;q)_{\lambda_r-\mu_r}$ in the numerator and $(a_{r+1}/b_{r};q)_{\mu_r-\mu_{r+1}}$ in the denominator. So, we can find a constant $C_1$ such that for all $x,q,\lambda,\mu$ satisfying $|x|<\varepsilon', |q|<\delta,l(\mu)\leq k$ and $\mu\prec\lambda$ we have
$$
\left|\calB^{\lambda'-\mu'}\F_{\lambda/\mu}(x\mid\calA,\calB)\right|\leq C_1 \prod_{r=1}^{k+1}\left|\frac{x^{\lambda_r-\mu_r}(a_r/x;q)_{\lambda_r-\mu_r}}{(a_{r+1}/b_r;q)_{\mu_r-\mu_{r+1}}}\right|.
$$
Note that 
$$
|x^m(a_r/x;q)_m|=\left|\prod_{i=1}^m(x-a_rq^{i-1})\right|<\prod_{i=1}^m(\varepsilon'+|a_r|\delta^{i-1}).
$$
Since $\varepsilon'+|a_r|\delta^{i-1}<\varepsilon$ for $i\gg 0$, there is a constant $C_2$ such that
$$
\prod_{r=1}^{k+1}|x^{\lambda_r-\mu_r}(a_r/x;q)_{\lambda_r-\mu_r}|< C_2 \varepsilon^{|\lambda|-|\mu|}.
$$
Finally, since $\frac{(a_{r+1}/b_r;q)_{\infty}}{(a_{r+1}/b_r;q)_{n}}=(q^na_{r+1}/b_r;q)_{\infty}$ is uniformly bounded for all $n$, $q:|q|<\delta$, we get the desired upper bound for $\calB^{\lambda'-\mu'}\F_{\lambda/\mu}(x\mid\calA,\calB)\prod_{j=1}^k(a_{j+1}/b_j;q)_{\infty}$.

The step of induction follows from the branching rule in the same way as in Lemma \ref{lemma: bounds}.

\end{proof}

Now, fix a collection of complex parameters $\calA$ and $u,v\in\mathbb C$ and a partition $\mu$. Without loss of generality, assume $|u|\leq|v|$. Consider each side of \eqref{eq: littlewoodsqW} as an analytic function in $x_1, \dots, x_n, q$. We first focus on the sum in the left-hand side. Choose $\delta\in(0,1)$. The sum over $\lambda$ can be restricted to partitions $\lambda$ such that $\lambda/\mu$ is a union of $n$ horizontal strips, in particular, $l(\lambda)\leq n+l(\mu)$. For such $\lambda$ Lemma \ref{lem: bound1} implies the following uniform upper bound for all $q\in\{q\in\mathbb C: |q|<\delta\}$:
\begin{multline*}
\left|\prod_{i=1}^{n+l(\mu)}(a_i^2;q)_\infty h_{\lambda'}(u,v\mid\calA)\right|\\
\leq C_1 \prod_{\substack{i\in[1,n+l(\mu)]\\i-odd}} (\lambda_i-\lambda_{i+1}+1)\max(|u|,|v|)^{\lambda_i-\lambda_{i+1}}  \prod_{\substack{i\in[1,n+l(\mu)]\\i-even}} (\lambda_i-\lambda_{i+1}+1)\max(1,|uv|)^{\lambda_i-\lambda_{i+1}}
\\\leq C_1(\lambda_1+1)^{n+l(\mu)} \max(|v|^{\lambda_1}|u|^{m_{even}(\lambda')}, |v|^{m_{odd}(\lambda')}).
\end{multline*}
Here $C_1$ is a constant and $m_{odd}(\mu)$ ($m_{even}(\mu)$) denotes the number of odd parts of $\mu$ (nonzero even parts of $\mu$ respectively). Clearly $m_{odd}(\lambda')\leq\lambda_1$ and $m_{even}(\lambda')\leq\lambda_2$, so  
$$
\left|\prod_{i=1}^{n+l(\mu)}(a_i^2;q)_\infty h_{\lambda'}(u,v\mid\calA)\right|\leq C_1(\lambda_1+1)^{n+l(\mu)} \max(|v|,1)^{\lambda_1}\max(|u|,1)^{\lambda_2}.
$$
Now choose $\varepsilon,\varepsilon'$ such that $0<\varepsilon'<\varepsilon<\min(1,|v|^{-1})$. Lemma \ref{lem: bound2} implies that there is a constant $C_2$ such that
$$
\left|\overline{\calA}^{\lambda'-\mu'}\F_{\lambda/\mu}(x_1,\dots, x_n\mid\calA,\overline{\calA})\prod_{i=1}^n\prod_{j=1}^{l(\mu)+n-i}(a_ja_{j+i};q)_\infty\right|<C_2(\lambda_1+1)^{(n-1)(l(\mu)+\frac{n}{2})}\varepsilon^{|\lambda|-|\mu|}
$$
for every partition $\lambda$ and $x_1, \dots, x_n,q$ such that $|x_1|,\dots,|x_n|<\varepsilon'$ and $|q|<\delta$. Combining these upper bounds and using that $\mu$ is fixed, we get a constant $C$ such that
\begin{multline*}
\left|h_{\lambda'}(u,v\mid\calA)\overline{\calA}^{\lambda'}\F_{\lambda/\mu}(x_1,\dots, x_n\mid\calA,\overline{\calA})\prod_{i=0}^n\prod_{j=1}^{l(\mu)+n-i}(a_ja_{j+i};q)_\infty\right|\\
<C(\lambda_1+1)^{n(l(\mu)+\frac{n+1}{2})}\varepsilon^{|\lambda|}\max(|v|,1)^{\lambda_1}\max(|u|,1)^{\lambda_2}
\end{multline*}
for all $\lambda, x_1, \dots, x_n, q$ as above. Due to our conditions on $\varepsilon$, the bound above is summable over $\lambda$. Dividing by $\prod_{i=0}^n\prod_{j=1}^{l(\mu)+n-i}(a_ja_{j+i};q)_\infty$  and taking $\delta\to 1$, $\varepsilon,\varepsilon'\to \min(1,|v|^{-1})$ shows that the left-hand side of \eqref{eq: littlewoodsqW} is an analytic function on
\begin{multline*}
D=\{(x_1, \dots, x_n, q)\in\mathbb C^{n+1}: |q|<1; |x_1|,\dots, |x_n|< \min(1, |v|^{-1});\\
q^{-m}\neq a_ia_j\ \text{for\ all\ } i,j\geq1,m\geq0\}.
\end{multline*}

The sum in the right-hand side of \eqref{eq: littlewoodsqW} is finite, so the right-hand side is clearly analytic on $D$. Since $D$ is a connected domain containing $0$, \eqref{eq: littlewoodsqW} follows from the identity of Taylor expansions of both sides about $0$. The latter is equivalent to the identity of the formal power series. 
\qed

\appendix

\section{Match with non-stable spin Hall-Littlewood functions} \label{sHLmatch}

Here we describe the correspondence between our spin Hall-Littlewood functions $\sfF$ and the higher spin Hall-Littlewood functions  from \cite{BP18}, which we denote by $\mathsf{F}^{BP}$, $\mathsf{G}^{BP}$.

We start by matching the vertex models. In \cite{BP18} the fused edges are parametrized by their spin $s$ and spectral $\xi$ parameters. In terms of our parameters $a,b$ they are given by $\mathsf{s}=\sqrt{a}/\sqrt{b}$ and $\xi=(\sqrt{a}\sqrt{b})^{-1}$. With this notation our higher spin six-vertex weights $w_{z;a,b}$ are exactly the stochastic weights $\mathsf{L}_{z\xi, \mathsf{s}}$ from \cite{BP18}. However, a different normalization of the weights is mainly used in \cite{BP18}. We denote it by $w^{BP}_{u;\mathsf{s}}$ and it is given by
\begin{align*}
\begin{array}{cccc}
\begin{tikzpicture}[scale=0.9, baseline=-4pt]
	\draw[unfused] (-1,0) -- (1,0);
	\draw[fused] (0,-1) -- (0,1);
	\node[left] at (-1,0) {\tiny $0;z$};\node[right] at (1,0) {\tiny $0;z$};
	\node[below] at (0,-1) {\tiny $g;\xi,\mathsf{s}$};\node[above] at (0,1) {\tiny $g; \xi,\mathsf{s}$};
	\node[above right] at (0,0) {\tiny $w^{BP}_{z\xi,\mathsf{s}}$};
\end{tikzpicture}
&
\begin{tikzpicture}[scale=0.9, baseline=-4pt]
	\draw[unfused] (-1,0) -- (1,0);
	\draw[fused] (0,-1) -- (0,1);
	\node[left] at (-1,0) {\tiny $0;z$};\node[right] at (1,0) {\tiny $1;z$};
	\node[below] at (0,-1) {\tiny $g;\xi,\mathsf{s}$};\node[above] at (0,1) {\tiny $g-1; \xi,\mathsf{s}$};
	\node[above right] at (0,0) {\tiny $w^{BP}_{z\xi,\mathsf{s}}$};
\end{tikzpicture}
&
\begin{tikzpicture}[scale=0.9, baseline=-4pt]
	\draw[unfused] (-1,0) -- (1,0);
	\draw[fused] (0,-1) -- (0,1);
	\node[left] at (-1,0) {\tiny $1;z$};\node[right] at (1,0) {\tiny $0;z$};
	\node[below] at (0,-1) {\tiny $g;\xi,\mathsf{s}$};\node[above] at (0,1) {\tiny $g+1; \xi,\mathsf{s}$};
	\node[above right] at (0,0) {\tiny $w^{BP}_{z\xi,\mathsf{s}}$};
\end{tikzpicture}
&
\begin{tikzpicture}[scale=0.9, baseline=-4pt]
	\draw[unfused] (-1,0) -- (1,0);
	\draw[fused] (0,-1) -- (0,1);
	\node[left] at (-1,0) {\tiny $1;z$};\node[right] at (1,0) {\tiny $1;z$};
	\node[below] at (0,-1) {\tiny $g;\xi,\mathsf{s}$};\node[above] at (0,1) {\tiny $g; \xi,\mathsf{s}$};
	\node[above right] at (0,0) {\tiny $w^{BP}_{z\xi,\mathsf{s}}$};
\end{tikzpicture}
\\
\\
\dfrac{1-q^g\mathsf{s}z\xi}{1-\mathsf{s}z\xi}
\qquad
&
\dfrac{(1-\mathsf{s}^2q^{g-1})z\xi}{1-\mathsf{s}z\xi}
\qquad
&
\dfrac{1-q^{g+1}}{1-\mathsf{s}z\xi}
\qquad
&
\dfrac{z\xi-\mathsf{s}q^g}{1-\mathsf{s}z\xi}
\vspace{1ex}
\end{array}
\end{align*}
Note that
\begin{equation}\label{eq: BPrenorm}
\begin{tikzpicture}[scale=1.1, baseline=-4pt]
	\draw[unfused] (-1,0) -- (1,0);
	\draw[fused] (0,-1) -- (0,1);
	\node[left] at (-1,0) {\tiny $i;z$};\node[right] at (1,0) {\tiny $l;z$};
	\node[below] at (0,-1) {\tiny $j;\xi,\mathsf{s}$};\node[above] at (0,1) {\tiny $k; \xi,\mathsf{s}$};
	\node[above right] at (0,0) {\tiny $w^{BP}_{z\xi,\mathsf{s}}$};
\end{tikzpicture}
=(-\mathsf{s})^{-l}\frac{(\mathsf{s}^2;q)_j}{(q;q)_{j}} \frac{(q;q)_{k}}{(\mathsf{s}^2;q)_k}
\begin{tikzpicture}[scale=1.1, baseline=-4pt]
	\draw[unfused] (-1,0) -- (1,0);
	\draw[fused] (0,-1) -- (0,1);
	\node[left] at (-1,0) {\tiny $i;z$};\node[right] at (1,0) {\tiny $l;z$};
	\node[below] at (0,-1) {\tiny $j; \mathsf{s}/\xi, (\mathsf{s}\xi)^{-1}$};\node[above] at (0,1) {\tiny $k; \mathsf{s}/\xi, (\mathsf{s}\xi)^{-1}$};
	\node[above right] at (0,0) {\tiny $w$};
\end{tikzpicture}
\end{equation}

Using the vertex weights $w^{BP}$, two families of functions $\mathsf{F}^{BP}$, $\mathsf{G}^{BP}$ are defined in \cite{BP18}. Both families are parametrized by pairs of \emph{signatures} $\lambda,\mu$, which are integer sequences of finite length. We use $\mathrm{Sign}_N^+$ to denote the set of non-negative signatures $\lambda_1\geq\lambda_2\geq\dots \geq \lambda_N\geq0$. Additionally, these functions depend on two infinite sequences of parameters, which are denoted by $\mathsf{S}=(\mathsf{s}_0, \mathsf{s}_1, \dots)$, $\Xi=(\xi_0, \xi_1,\dots)$. With this notation, for $\lambda\in\mathrm{Sign}_{N+n}^+$ and $\mu\in\mathrm{Sign}_{N}^+$ the function $\mathsf{F}_{\lambda/\mu}^{BP}(x_1, \dots, x_n\mid \Xi,\mathsf{S})$ is defined by the partition function depicted on the left part of Figure \ref{fig: BPF}, where all the edges on the left boundary have value $1$. For $\lambda\in\mathrm{Sign}_{N}^+$ and $\mu\in\mathrm{Sign}_{N}^+$ the function $\mathsf{G}_{\lambda/\mu}^{BP}(x_1, \dots, x_m\mid \Xi,\mathsf{S})$ is defined by the right part of Figure \ref{fig: BPF}, where the left boundary is empty. Below we list some of the properties of these functions from \cite{BP18}:

\begin{figure}

	\begin{center}
    		\begin{tikzpicture}[scale=1, xscale=0.9]
			\foreach\y in {4,...,5}{
				\draw[unfused*] (1,\y) -- (6.1,\y);
				\node[right] at (6.1,\y) {$\dots$};
				\draw[unfused] (6.8,\y) -- (7.2,\y);
			}
			\draw[unfused*] (1,2) -- (6.1,2);
			\node[right] at (6.1,2) {$\dots$};
			\draw[unfused] (6.8,2) -- (7.2,2);
			\foreach\x in {1,...,3}{
				\draw[fused] (1.7*\x+0.3,3.3) -- (1.7*\x+0.3,6);
				\draw[fused*] (1.7*\x+0.3,1) -- (1.7*\x+0.3,2.5);
				\node at (1.7*\x+0.3, 3) {$\vdots$};
			}
			\node[above right] at (2,5) {\tiny $w^{BP}$};
			\node[above right] at (3.7,5) {\tiny $w^{BP}$};
			\node[above right] at (5.4,5) {\tiny $w^{BP}$};
			\node[above right] at (2,4) {\tiny $w^{BP}$};
			\node[above right] at (3.7,4) {\tiny $w^{BP}$};
			\node[above right] at (5.4,4) {\tiny $w^{BP}$};
			\node[above right] at (2,2) {\tiny $w^{BP}$};
			\node[above right] at (3.7,2) {\tiny $w^{BP}$};
			\node[above right] at (5.4,2) {\tiny $w^{BP}$};
			\node[above] at (5.7,6) {\tiny $m_2(\lambda);  \xi_2,\mathsf{s}_2$};
			\node[above] at (3.7,6) {\tiny $m_1(\lambda);  \xi_1,\mathsf{s}_1$};
			\node[above] at (1.7,6) {\tiny $m_0(\lambda); \xi_0,\mathsf{s}_0$};
			\node[left] at (1,2) {\tiny $1; x_n$};
			\node[left] at (1,4) {\tiny $1; x_2$};
			\node[left] at (1,5) {\tiny $1; x_1$};
			\node[below] at (5.7,1) {\tiny $m_2(\mu);  \xi_2,\mathsf{s}_2$};
			\node[below] at (3.7,1) {\tiny $m_1(\mu);  \xi_1,\mathsf{s}_1$};
			\node[below] at (1.7,1) {\tiny $m_0(\mu); \xi_0,\mathsf{s}_0$};
			\node[right] at (7.2,2) {\tiny $0;x_n$};
			\node[right] at (7.2,3) {$\vdots$};
			\node[right] at (7.2,4) {\tiny $0; x_2$};
			\node[right] at (7.2,5) {\tiny $0; x_1$};
		\end{tikzpicture}
        \quad
		\begin{tikzpicture}[scale=1,xscale=0.9]
			\foreach\y in {4,...,5}{
				\draw[unfused*] (1,\y) -- (6.1,\y);
				\node[right] at (6.1,\y) {$\dots$};
				\draw[unfused] (6.8,\y) -- (7.2,\y);
			}
			\draw[unfused*] (1,2) -- (6.1,2);
			\node[right] at (6.1,2) {$\dots$};
			\draw[unfused] (6.8,2) -- (7.2,2);
			\foreach\x in {1,...,3}{
				\draw[fused] (1.7*\x+0.3,3.3) -- (1.7*\x+0.3,6);
				\draw[fused*] (1.7*\x+0.3,1) -- (1.7*\x+0.3,2.5);
				\node at (1.7*\x+0.3, 3) {$\vdots$};
			}
			\node[above right] at (2,5) {\tiny $w^{BP}$};
			\node[above right] at (3.7,5) {\tiny $w^{BP}$};
			\node[above right] at (5.4,5) {\tiny $w^{BP}$};
			\node[above right] at (2,4) {\tiny $w^{BP}$};
			\node[above right] at (3.7,4) {\tiny $w^{BP}$};
			\node[above right] at (5.4,4) {\tiny $w^{BP}$};
			\node[above right] at (2,2) {\tiny $w^{BP}$};
			\node[above right] at (3.7,2) {\tiny $w^{BP}$};
			\node[above right] at (5.4,2) {\tiny $w^{BP}$};
			\node[above] at (5.7,6) {\tiny $m_2(\lambda);  \xi_2,\mathsf{s}_2$};
			\node[above] at (3.7,6) {\tiny $m_1(\lambda);  \xi_1,\mathsf{s}_1$};
			\node[above] at (1.7,6) {\tiny $m_0(\lambda); \xi_0,\mathsf{s}_0$};
			\node[left] at (1,2) {\tiny $0; x_n$};
			\node[left] at (1,4) {\tiny $0; x_2$};
			\node[left] at (1,5) {\tiny $0; x_1$};
			\node[below] at (5.7,1) {\tiny $m_2(\mu);  \xi_2,\mathsf{s}_2$};
			\node[below] at (3.7,1) {\tiny $m_1(\mu);  \xi_1,\mathsf{s}_1$};
			\node[below] at (1.7,1) {\tiny $m_0(\mu); \xi_0,\mathsf{s}_0$};
			\node[right] at (7.2,2) {\tiny $0;x_n$};
			\node[right] at (7.2,3) {$\vdots$};
			\node[right] at (7.2,4) {\tiny $0; x_2$};
			\node[right] at (7.2,5) {\tiny $0; x_1$};
		\end{tikzpicture}
	\end{center}
    \caption{\label{fig: BPF}Left: the partition function defining $\mathsf{F}_{\lambda/\mu}^{BP}(x_1, \dots, x_n\mid\Xi,\mathsf{S})$. Right: the partition function defining $\mathsf{G}_{\lambda/\mu}^{BP}(x_1, \dots, x_n\mid\Xi,\mathsf{S})$. In both pictures $m_0(\mu)$ denotes the number of $0$s in the signature $\mu\in\mathrm{Sign}^+_{N}$, that is, $m_0(\mu)=N-l(\mu)$.}
\end{figure}
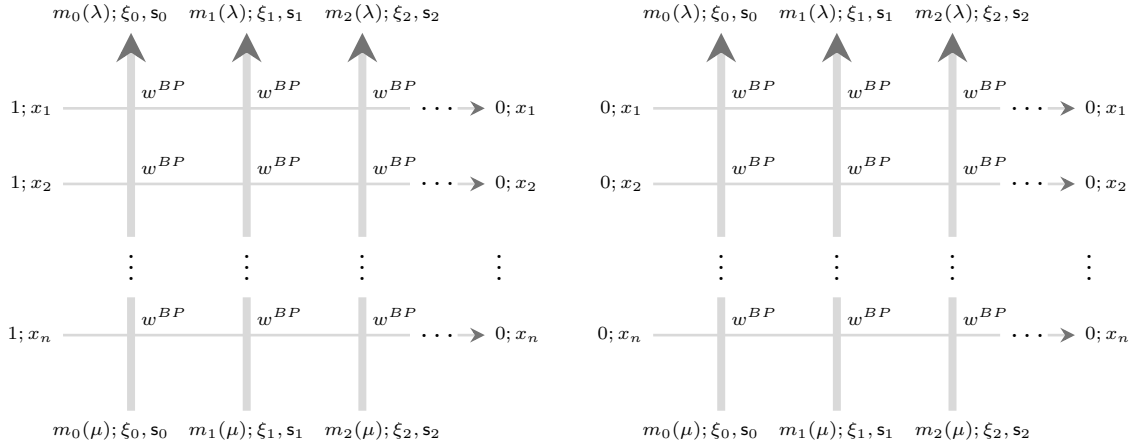

\begin{proposition}[{\cite[Proposition 4.5]{BP18}}]\label{prop: BPsym}
    Both functions $\mathsf{F}_{\lambda/\mu}^{BP}(x_1, \dots, x_n\mid \Xi,\mathsf{S})$ and $\mathsf{G}_{\lambda/\mu}^{BP}(x_1, \dots, x_n\mid \Xi,\mathsf{S})$ are symmetric in $x_1, \dots, x_n$.
\end{proposition}

\begin{proposition}\label{prop: BPsymfor}
For any $n\geq 0$ and $\mu\in\mathrm{Sign}^+_{n}$ we have
$$
\mathsf{F}_{\mu}^{BP}(x_1, \dots, x_n\mid \Xi,\mathsf{S})=\sum_{\sigma\in S_n}\sigma\left(\prod_{1\leq i<j\leq n}\frac{x_i-qx_j}{x_i-x_j}\prod_{i=1}^n\upvarphi^{BP}_{\mu_i}(x_i\mid\Xi,\mathsf{S})\right),
$$
where
$$
\upvarphi^{BP}_k(x\mid\Xi,\mathsf{S})=\frac{1-q}{1-\mathsf{s}_k\xi_k x}\prod_{j=0}^{k-1}\frac{\xi_j x-\mathsf{s}_j}{1-\mathsf{s}_j\xi_jx}.
$$
\end{proposition}

\begin{proposition}[{\cite[Corollary 4.10]{BP18}}]\label{prop: BPCauchy} Assume that both $\Xi$ and $\mathsf{S}$ are sequences of complex parameters, and let $x_1, \dots, x_n, y_1,\dots, y_m\in\mathbb C$ satisfy
$$
\prod_{k=M}^\infty \left|\frac{\xi_kx_i-\mathsf{s}_k}{1-\mathsf{s}_k\xi_kx_i}\cdot\frac{\xi^{-1}_ky_j-\mathsf{s}_k}{1-\mathsf{s}_k\xi^{-1}_ky_j}\right|=0
$$
for every $i,j,M$. Then for every pair of signatures $\lambda\in\mathrm{Sign}^+_{N}$, $\nu\in\mathrm{Sign}^+_{N+n}$ we have
\begin{multline*}
\sum_{\kappa\in\mathrm{Sign}_{N+n}^+} \mathsf{F}^{BP}_{\kappa/\lambda}(x_1,\dots, x_n\mid\Xi,\mathsf{S}) \mathsf{G}^{BP*}_{\kappa/\nu}(y_1,\dots, y_m\mid\overline{\Xi},\mathsf{S})\\
=\prod_{i=1}^n\prod_{j=1}^m\frac{1-qx_iy_j}{1-x_iy_j}\sum_{\mu\in\mathrm{Sign}_{N}^+} \mathsf{F}^{BP}_{\nu/\mu}(x_1,\dots, x_n\mid\Xi,\mathsf{S}) \mathsf{G}^{BP*}_{\lambda/\mu}(y_1,\dots, y_m\mid\overline{\Xi},\mathsf{S}),
\end{multline*}
where
$$
\mathsf{G}^{BP*}_{\kappa/\nu}(y_1,\dots, y_m\mid\overline{\Xi},\mathsf{S})=\prod_{j\geq 0}\frac{(s_j^2;q)_{m_{j}(\kappa)}}{(q;q)_{m_{j}(\kappa)}}\frac{(q;q)_{m_{j}(\nu)}}{(s_j^2;q)_{m_{j}(\nu)}}\mathsf{G}^{BP}_{\kappa/\nu}(y_1,\dots, y_m\mid\overline{\Xi},\mathsf{S}).
$$
\end{proposition}

To recover our stable Hall-Littlewood functions $\sfF_\lambda$ we should set $\mathsf{s}_j=\sqrt{a_j}/\sqrt{b_j}$, $\xi_j=(\sqrt{a_j}\sqrt{b_j})^{-1}$ for every $j\geq1$, and $\mathsf{s}_0=0$, $\xi_0=1$. This substitution results in the following weights in the 0th column:

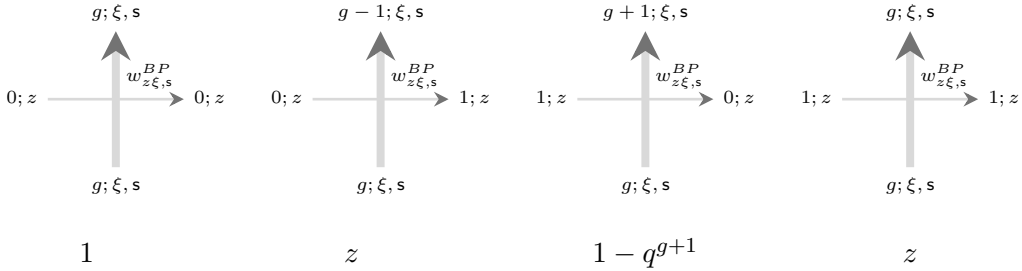
\begin{align*}
\begin{array}{cccc}
\begin{tikzpicture}[scale=0.9, baseline=-4pt]
	\draw[unfused] (-1,0) -- (1,0);
	\draw[fused] (0,-1) -- (0,1);
	\node[left] at (-1,0) {\tiny $0;z$};\node[right] at (1,0) {\tiny $0;z$};
	\node[below] at (0,-1) {\tiny $g;\xi,\mathsf{s}$};\node[above] at (0,1) {\tiny $g; \xi,\mathsf{s}$};
	\node[above right] at (0,0) {\tiny $w^{BP}_{z\xi,\mathsf{s}}$};
\end{tikzpicture}
&
\begin{tikzpicture}[scale=0.9, baseline=-4pt]
	\draw[unfused] (-1,0) -- (1,0);
	\draw[fused] (0,-1) -- (0,1);
	\node[left] at (-1,0) {\tiny $0;z$};\node[right] at (1,0) {\tiny $1;z$};
	\node[below] at (0,-1) {\tiny $g;\xi,\mathsf{s}$};\node[above] at (0,1) {\tiny $g-1; \xi,\mathsf{s}$};
	\node[above right] at (0,0) {\tiny $w^{BP}_{z\xi,\mathsf{s}}$};
\end{tikzpicture}
&
\begin{tikzpicture}[scale=0.9, baseline=-4pt]
	\draw[unfused] (-1,0) -- (1,0);
	\draw[fused] (0,-1) -- (0,1);
	\node[left] at (-1,0) {\tiny $1;z$};\node[right] at (1,0) {\tiny $0;z$};
	\node[below] at (0,-1) {\tiny $g;\xi,\mathsf{s}$};\node[above] at (0,1) {\tiny $g+1; \xi,\mathsf{s}$};
	\node[above right] at (0,0) {\tiny $w^{BP}_{z\xi,\mathsf{s}}$};
\end{tikzpicture}
&
\begin{tikzpicture}[scale=0.9, baseline=-4pt]
	\draw[unfused] (-1,0) -- (1,0);
	\draw[fused] (0,-1) -- (0,1);
	\node[left] at (-1,0) {\tiny $1;z$};\node[right] at (1,0) {\tiny $1;z$};
	\node[below] at (0,-1) {\tiny $g;\xi,\mathsf{s}$};\node[above] at (0,1) {\tiny $g; \xi,\mathsf{s}$};
	\node[above right] at (0,0) {\tiny $w^{BP}_{z\xi,\mathsf{s}}$};
\end{tikzpicture}
\\
\\
1
\qquad
&
z
\qquad
&
\qquad
1-q^{g+1}
\qquad
&
z
\vspace{1ex}
\end{array}
\end{align*}
This allows us to remove the 0th column from the partition functions in Figure \ref{fig: BPF}, and instead consider the sum over possible $\{0,1\}$ values of the horizontal edges entering the 1st column. Each such edge with value $1$ produces the weight $x_i$, which coincides with the factors used in \eqref{eq: def-sHL} up to $(-b_0)^{l(\lambda)-l(\mu)}$. The weights $1-q^{g+1}$ appear only for $\mathsf{F}^{BP}_{\lambda/\mu}$, and they produce a factor $\frac{(q;q)_{m_0(\lambda)}}{(q;q)_{m_0(\mu)}}$. 
Finally, we use \eqref{eq: BPrenorm} to change the weights $w^{BP}_{z\xi;\mathsf{s}}$ to the weights $w_{z;a,b}$ in the other columns of Figure \ref{fig: BPF}. The factor $\frac{(\mathsf{s}_j^2;q)_i}{(q;q)_{i}}$ from \eqref{eq: BPrenorm} forms a telescoping product in each column, leaving only
$$
\prod_{j\geq 1}\frac{(\mathsf{s}^2_j;q)_{m_j(\mu)}}{(q;q)_{m_j(\mu)}}\frac{(q;q)_{m_j(\lambda)}}{(\mathsf{s}^2_j;q)_{m_j(\lambda)}}=\frac{\psi_{\lambda}(\calA,\calB)}{\psi_{\mu}(\calA,\calB)}\prod_{j\geq 0}\left(\frac{a_j}{b_j}\right)^{\lambda'_{j+1}-\mu'_{j+1}},
$$
where we use the coefficients $\psi_\lambda$ from \eqref{eq:def-psi}. The factors $(-\mathsf{s}_j)^{-l}$ can be counted using the conservation law: for fixed $j$ the total number of such factors is the number of horizontal edges with value $1$ entering column $j+1$, and there are $\lambda'_{j+1}-\mu'_{j+1}$ such edges. Collecting all the factors together, we get the following:
\begin{proposition} \label{prop: BPrenorm} Let $\mathsf{s}_0=0$, $\xi_0=1$ and  $\mathsf{s}_j=\sqrt{a_j}/\sqrt{b_j}$, $\xi_j=(\sqrt{a_j}\sqrt{b_j})^{-1}$ for every $j\geq1$. Then for $\lambda\in\mathrm{Sign}^+_{N+n}$, $\mu\in\mathrm{Sign}^+_{N}$ we have
\begin{multline*}
\mathsf{F}^{BP}_{\lambda/\mu}(x_1, \dots, x_n\mid\Xi,\mathsf{S})\\
=\frac{(q;q)_{N+n-l(\lambda)}}{(q;q)_{N-l(\mu)}}\frac{\psi_{\lambda}(\calA,\calB)}{\psi_{\mu}(\calA,\calB)}(-a_0)^{\lambda'_1-\mu'_1}\prod_{j\geq 1}\left(-\sqrt{\frac{a_j}{b_j}}\right)^{\lambda'_{j+1}-\mu'_{j+1}} \sfF_{\lambda/\mu}(x_1,\dots, x_n\mid\calA,\calB).
\end{multline*}
For $\lambda\in\mathrm{Sign}^+_{N}$, $\mu\in\mathrm{Sign}^+_{N}$ we have
\begin{multline*}
\mathsf{G}^{BP}_{\lambda/\mu}(x_1, \dots, x_m\mid\Xi,\mathsf{S})\\
=\frac{\psi_{\lambda}(\calA,\calB)}{\psi_{\mu}(\calA,\calB)}(-a_0)^{\lambda_1'-\mu_1'}\prod_{j\geq 1}\left(-\sqrt{\frac{a_j}{b_j}}\right)^{\lambda'_{j+1}-\mu'_{j+1}} \sfF_{\lambda/\mu}(x_1,\dots, x_m\mid\calA,\calB).
\end{multline*}
\end{proposition}

Propositions now follow immediately from Propositions \ref{prop: BPsym}-\ref{prop: BPCauchy} by using the substitution from Proposition \ref{prop: BPrenorm}. For the Cauchy identity, note that changing $(\Xi, \mathsf{S})$ to $(\overline{\Xi}, \mathsf{S})$ is equivalent to $(\calA,\calB)\to(\overline{\calB},\overline{\calA})$. Moreover,
$$
\prod_{j\geq 0}\frac{(s_j^2;q)_{m_{j}(\lambda)}}{(q;q)_{m_{j}(\lambda)}}\frac{(q;q)_{m_{j}(\mu)}}{(s_j^2;q)_{m_{j}(\mu)}}=\frac{(q;q)_{m_0(\mu)}}{(q;q)_{m_0(\lambda)}}\frac{\psi_{\mu}(\calA,\calB)}{\psi_{\lambda}(\calA,\calB)}\prod_{j\geq 0}\left(\frac{b_j}{a_j}\right)^{\lambda'_{j+1}-\mu'_{j+1}}.
$$
So, for $\lambda,\mu\in\mathrm{Sign}^+_{N}$ we have
\begin{multline*}
\mathsf{G}^{BP*}_{\lambda/\mu}(x_1, \dots, x_m\mid\overline{\Xi},\mathsf{S})\\
=\frac{(q;q)_{N-l(\mu)}}{(q;q)_{N-l(\lambda)}}(-a_0)^{\mu_1'-\lambda_1'}\prod_{j\geq 1}\left(-\sqrt{\frac{b_j}{a_j}}\right)^{\lambda'_{j+1}-\mu'_{j+1}} \sfF_{\lambda/\mu}(x_1,\dots, x_m\mid\overline\calB,\overline\calA).
\end{multline*}

\printbibliography

	\bigskip
	
	\textsc{S. Gavrilova, MIT, Cambridge, MA, USA}
	
	E-mail: \texttt{sveta117@mit.edu}

    \textsc{S. Korotkikh, The Ohio State University, Columbus, OH, USA}
    
    E-mail: \texttt{shortkih@gmail.com}
    
\end{document}